\documentclass{article}

\usepackage[numbers]{natbib}
\usepackage[bookmarkstype=toc,colorlinks=true,citecolor=blue,linkcolor=red]{hyperref}  

\usepackage{graphicx}
\usepackage{amsmath,amssymb,amsfonts,amsthm}

\usepackage{geometry}
\usepackage{enumerate}

\newtheorem{thm}{Theorem}[section]
\newtheorem{lem}{Lemma}[section]
\newtheorem{prop}{Proposition}[section]

\theoremstyle{definition}
\newtheorem{Def}{Definition}[section]
\newtheorem*{rmk*}{Remark}
\newtheorem{rmk}{Remark}[section]

\makeatletter
    
    \@addtoreset{equation}{section}
\makeatother

\makeatletter
    \renewcommand*{\section}{\@startsection{section}{1}{\z@}%
    {21pt}{12pt}{\reset@font\normalsize\bfseries}}
\makeatother
    
\makeatletter
    \renewcommand*{\subsection}{\@startsection{subsection}{2}{\z@}%
    {15pt}{6pt}{\reset@font\normalsize\mdseries\itshape}}
\makeatother

\makeatletter
\def\@listi{\leftmargin\leftmargini
  \topsep=.5\baselineskip 
  \partopsep=0pt \parsep=0pt \itemsep=0pt}
\let\@listI\@listi
\@listi
\def\@listii{\leftmargin\leftmarginii
  \labelwidth\leftmarginii \advance\labelwidth-\labelsep
  \topsep=0pt \partopsep=0pt \parsep=0pt \itemsep=0pt}
\def\@listiii{\leftmargin\leftmarginiii
  \labelwidth\leftmarginiii \advance\labelwidth-\labelsep
  \topsep=0pt \partopsep=0pt \parsep=0pt \itemsep=0pt}
\def\@listiv{\leftmargin\leftmarginiv
  \labelwidth\leftmarginiv \advance\labelwidth-\labelsep
  \topsep=0pt \partopsep=0pt \parsep=0pt \itemsep=0pt}
\makeatother

\title{Limit theorems for the pre-averaged Hayashi-Yoshida estimator with random sampling}
\author{Yuta Koike\thanks{University of Tokyo, Graduate School of Mathematical Sciences, 3-8-1 Komaba, Meguro-ku, Tokyo 153, Japan, Email: kyuta@ms.u-tokyo.ac.jp}}

\begin{document}

\maketitle 

\begin{abstract}
We will focus on estimating the integrated covariance of two diffusion processes observed in a nonsynchronous manner. The observation data is contaminated by some noise, which is possibly correlated with the returns of the diffusion processes, while the sampling times also possibly depend on the observed processes. 
In a high-frequency setting, we consider a modified version of the pre-averaged Hayashi-Yoshida estimator, and we show that such a kind of estimators has the consistency and the asymptotic mixed normality, and attains the optimal rate of convergence. \vspace{3mm}

\end{abstract}

\section{Introduction}\label{intro}

In financial econometrics, measuring the covariation of two assets is the central problem because it serves as a basis for many areas of finance, such as risk management, portfolio allocation and hedging strategies. In recent years there has been a considerable development of the statistical approaches to this problem using high frequency data. Such approaches were pioneered by \citet{AB1998} and \citet{BNS2002}, and their methods are based on the semimartingale theory. 
In fact, no-arbitrage based characterizations of asset prices suggest that price processes must follow a semimartingale (see \cite{DS1994} for instance). 
Recently, however, it has become common recognition that at ultra-high frequencies the financial data is contaminated by market microstructure noise such as rounding errors, bid-ask bounds and misprints. 
Motivated by this, the statistical inference for semimartingales observed at a high frequency with additive observation noise has become an active research area during the past decade.

On the other hand, since in this paper we are interested in the statistical inference for two assets observed at a high frequency, we face another important problem. That is, we may observe the data in a nonsynchronous manner. The classical theory of stochastic calculus suggests that the so-called realized covariance can be used for measuring the covariation of two assets if the sampling is synchronous. Therefore, it is a naive idea that first we fix a sampling frequency (e.g. per five minutes) and generate new data sampled at this fixed grid by the previous-tick interpolation scheme and then we compute the realized covariance from the synchronized data. However, \citet{HY2005} shows that this method suffers from a serious bias known as the \textit{Epps effect} described in \cite{Epps1979}, so we need a different approach to deal with this problem. 
\cite{HY2005} proposed the so-called \textit{Hayashi-Yoshida estimator}, which is identical with the realized covariance in the synchronous case and a consistent estimator for the quadratic covariation of two discretely observed continuous semimartingales even in the nonsynchronous case. 
The asymptotic theory of the Hayahsi-Yoshida estimator has further been developed in \cite{HK2008}, \cite{HY2008,HY2011} and \cite{DY2011}. Another important theoretical approach to nonsynchronicity, a Fourier analytic approach, has been developed in \citet{MM2002,MM2009} and \citet{CG2011}; besides \citet{OY2012} have recently developed the quasi-likelihood analysis of nonsynchronously observed diffusions in a parametric setting.

In this paper we consider two diffusion processes observed in a nonsynchronous manner as well as contaminated by microstructure noise. Our aim is to estimate the integrated covariance of the diffusion processes in a high-frequency setting by coping with both of the observation noise and the nonsynchronous sampling simultaneously. Recently, various authors proposed hybrid approaches combining a method to de-noise the data with another method to deal with the nonsynchronicity in order to attack this problem. One direction in such approaches is that we first use the \textit{refresh sampling method} for synchronizing the data and then construct a noise-robust estimator. This method was first applied in \citet{BNHLS2011} in which the \textit{realized kernel method} proposed in \cite{BNHLS2008} was used for de-noising, and further developed by \cite{AFX2010}, \cite{CKP2010}, \cite{Ikeda2010}, \cite{Varneskov2011b} and \cite{Zhang2011} with using other de-noising methods. Another direction is using a Hayashi-Yoshida type approach to deal with the nonsynchronicity. 
\citet{B2011a} proposed to synchronize the data by applying a Hayashi-Yoshida type synchronization called the \textit{pseudo-aggregation algorithm} first. In a second step, a multiscale type estimator like in \cite{Z2006} is constructed from this synchronized data. 
The obtained estimator is called the \textit{generalized multiscale estimator}.
\citet{CKP2010} proposed to de-noise the data by applying the \textit{pre-averaging method} introduced in \cite{PV2009} (and further studied in \cite{JLMPV2009}) first. After that, they construct a Hayashi-Yoshida type estimator called the \textit{pre-averaged Hayashi-Yoshida estimator} from the pre-averaged data. 
On the other hand, recently \citet{CPA2012} proposed a new estimator which is not the hybrid one.

Our estimation approach is based on the pre-averaged Hayashi-Yoshida estimator in the above, but we slightly modify this estimator for a technical reason. In \citet{CPV2011} the associated central limit theorem for that estimator has been shown, but it is restricted to the case when observation times are deterministic (or random but independent of the observed processes) and some of important sampling schemes in practice, like the Poisson sampling schemes, are excluded. In fact, the asymptotic variance given in their theorem has a quite complex form which depends on the special forms of the sampling times considered in that paper, so that the author guesses one cannot expect to extend this result to more general sampling involving the Poisson sampling schemes. For this reason, we first synchronize the sampling times partly. After that, we construct the pre-averaged Hayashi-Yoshida estimator from this new data. Then, we can compute the asymptotic variance due to Lemma \ref{key} in Section \ref{main}. This procedure is done in the same spirit as that discussed in Section 6.3 of \cite{Bibinger2012}.

In addition to the above problems, we also consider two kinds of endogeneity; one is the dependency between the microstructure noise and the diffusion processes, and the other is the dependency between the sampling times and the observed processes. The first one is motivated by an empirical analysis of \citet{HL2006} as well as some microstructure noise modeling in microeconomics such as \cite{GM1985}. Especially we will involve an asymptotically non-degenerate term which is correlated with the returns of the diffusion processes in the model of microstructure noise since \cite{HL2006} indicates the presence of such a structure. On the other hand, the second one is motivated by the recent studies on this topics in the absence of microstructure noise. See \cite{Fu2010b}, \cite{HJY2011}, \cite{HY2011}, \cite{LMRZZ2012} and \cite{PY2008} for details. \citet{RR2012} also considers these two types of endogeneity in a framework different from the model with additive observation noise, which they call the model with \textit{uncertainty zones}. In their framework, however, the observation errors are asymptotically degenerate. In this paper, we will show a central limit theorem for the estimation error of the proposed estimator in the situation explained in the above (see Theorem \ref{mainthm}). 

Usually a certain blocking technique is used in the proofs of the central limit theorems for pre-averaging estimators (see \cite{JLMPV2009}, \cite{PV2009SPA}, \cite{JPV2010} and \cite{CPV2011}). In this paper, however, we do not rely on such a technique but a technique used in \cite{HY2011} for the proof of the central limit theorem for the Hayashi-Yoshida estimator. This is based on Lemma \ref{lembasic} which tells us that in the first order the estimation error process of our estimator is asymptotically equivalent to the process $\mathbf{M}^n$ defined in Section \ref{secstable}, which has a structure similar to that of the estimation error process of the Hayashi-Yoshida estimator. This enables us to apply arguments that mimic those in \cite{HY2011}. The only thing different from \cite{HY2011} is the computation of the asymptotic variance process, but we can pass through this problem due to the modification explained in the above.

Lemma \ref{lembasic} has another important implication for the asymptotic theory of our estimator. That is, we can deduce a law of large number for our estimator, which can be regarded as a counterpart to Theorem 2.3 in \cite{HK2008}. This will be presented in Section \ref{consistency}.

The organization of this paper is the following. In Section \ref{setting} we introduce the mathematical model and explain the construction of our estimator. In Section \ref{main} the main result in this paper is stated. Section \ref{secstable} provides a brief sketch of the proof of the main result, while in Section \ref{topics} we deal with some topics related to statistical application of our estimator. Most of the proofs will be put in Section \ref{prooflembasic}-\ref{proofHKthm2.3}.


\section{The setting}\label{setting}


We start by introducing an appropriate stochastic basis on which our observation data is defined. Let $\mathcal{B}^{(0)}=(\Omega^{(0)},\mathcal{F}^{(0)},\mathbf{F}^{(0)}=(\mathcal{F}^{(0)}_t)_{t\in\mathbb{R}_+} ,P^{(0)})$ be a stochastic basis. For any $t\in\mathbb{R}_+$ we have a transition probability $Q_t(\omega^{(0)},\mathrm{d}z)$ from $(\Omega^{(0)},\mathcal{F}^{(0)}_t)$ into $\mathbb{R}^2$, which satisfies
\begin{equation}\label{centered}
\int z Q_t(\omega^{(0)},\mathrm{d}z)=0.
\end{equation}
We endow the space $\Omega^{(1)}=(\mathbb{R}^2)^{[0,\infty)}$ with the product Borel $\sigma$-field $\mathcal{F}^{(1)}$ and with the probability $Q(\omega^{(0)},\mathrm{d}\omega^{(1)})$ which is the product $\otimes_{t\in\mathbb{R}_+}Q_t(\omega^{(0)},\cdot)$. We also call $(\epsilon_t)_{t\in\mathbb{R}_+}$ the ``canonical process'' on $(\Omega^{(1)},\mathcal{F}^{(1)})$ and the filtaration $\mathcal{F}^{(1)}_t=\sigma(\epsilon_s;s\leq t)$. Then we consider the stochastic basis $\mathcal{B}=(\Omega,\mathcal{F},\mathbf{F}=(\mathcal{F}_t)_{t\in\mathbb{R}_+} ,P)$ defined as follows:
\begin{gather*}
\Omega=\Omega^{(0)}\times\Omega^{(1)},\qquad
\mathcal{F}=\mathcal{F}^{(0)}\otimes\mathcal{F}^{(1)},\qquad
\mathcal{F}_t=\cap_{s>t}\mathcal{F}^{(0)}_s\otimes\mathcal{F}^{(1)}_s,\\
P(\mathrm{d}\omega^{(0)},\mathrm{d}\omega^{(1)})=P^{(0)}(\mathrm{d}\omega^{(0)})Q(\omega^{(0)},\mathrm{d}\omega^{(1)}).
\end{gather*}
Any variable or process which is defined on either $\Omega^{(0)}$ or $\Omega^{(1)}$ can be considered in the usual way as a variable or a process on $\Omega$.

Next we introduce our observation data. Let $X$ and $Y$ be two continuous semimartingales on $\mathcal{B}^{(0)}$. Also, we have two sequences of $\mathbf{F}^{(0)}$-stopping times $(S^i)_{i\in\mathbb{Z}_+}$ and $(T^j)_{j\in\mathbb{Z}_+}$ that are increasing a.s., 
\begin{equation}\label{increace}
S^i\uparrow\infty\qquad \textrm{and}\qquad T^j\uparrow\infty.
\end{equation}
As a matter of convenience we set $S^{-1}=T^{-1}=0$. These stopping times implicitly depend on a parameter $n\in\mathbb{N}$, which represents the frequency of the observations. Denote by $(b_n)$ a sequence of positive numbers tending to 0 as $n\to\infty$. Let $\xi'$ be a constant satisfying $0<\xi'<1$. In this paper, we will always assume that
\begin{equation}\label{A4}
r_n(t):=\sup_{i\in\mathbb{Z}_+}(S^i\wedge t-S^{i-1}\wedge t)\vee\sup_{j\in\mathbb{Z}_+}(T^j\wedge t-T^{j-1}\wedge t)=o_p(b_n^{\xi'})
\end{equation}
as $n\to\infty$ for any $t\in\mathbb{R}_+$.

The processes $X$ and $Y$ are observed at the sampling times $(S^i)$ and $(T^j)$ with observation errors $(U^X_{S^i})_{i\in\mathbb{Z}_+}$ and $(U^Y_{T^j})_{j\in\mathbb{Z}_+}$ respectively. We assume that the observation errors have the following representations:
\begin{equation*}
U^X_{S^i}=b_n^{-1/2}(Z^X_{S^i}-Z^X_{S^{i-1}})+\epsilon^X_{S^i},\qquad
U^Y_{T^j}=b_n^{-1/2}(Z^Y_{T^j}-Z^Y_{T^{j-1}})+\epsilon^Y_{T^j}.
\end{equation*}
Here, $\epsilon_t=(\epsilon^X_t,\epsilon^Y_t)$ for each $t$, while $Z^X$ and $Z^Y$ are two continuous semimartingales on $\mathcal{B}^{(0)}$. We can take $Z^X=\phi^X X$ and $Z^Y=\phi^Y Y$ for some constants $\phi^X$ and $\phi^Y$, so that the observation errors can be correlated with the returns of the latent processes $X$ and $Y$. For this reason we will refer to $(b_n^{-1/2}(Z^X_{S^i}-Z^X_{S^{i-1}}))_{i\in\mathbb{Z}_+}$ and $(b_n^{-1/2}(Z^Y_{T^j}-Z^Y_{T^{j-1}}))_{j\in\mathbb{Z}_+}$ as the endogenous noise. The factor $b_n^{-1/2}$ is necessary for the endogenous noise not to degenerate asymptotically. Such a kind of noise appears in \cite{BNHLS2011}, \cite{ES2007}, \cite{HL2006}, \cite{KL2008}, \cite{KS2011} and \cite{NV2012}. After all, we have the observation data $\mathsf{X}=(\mathsf{X}_{S^i})_{i\in\mathbb{Z}_+}$ and $\mathsf{Y}=(\mathsf{Y}_{T^j})_{j\in\mathbb{Z}_+}$ of the form
\begin{equation*}
\mathsf{X}_{S^i}=X_{S^i}+U^X_{S^i},\qquad\mathsf{Y}_{T^j}=Y_{T^j}+U^Y_{T^j}.
\end{equation*}


Now we explain the construction of our estimator. First we introduce some notation. We choose a sequence $k_n$ of integers and a number $\theta\in(0,\infty)$ satisfying
\begin{equation}\label{window}
k_n=\theta b_n^{-1/2}+o(b_n^{-1/4})
\end{equation}
(for example $k_n=\lceil\theta b_n^{-1/2}\rceil$). We also choose a continuous function $g:[0,1]\rightarrow\mathbb{R}$ which is piecewise $C^1$ with a piecewise Lipschitz derivative $g'$ and satisfies
\begin{equation}\label{weight}
g(0)=g(1)=0,\qquad \psi_{HY}:=\int_0^1 g(x)\mathrm{d}x\neq 0
\end{equation}
(for example $g(x)=x\wedge(1-x)$). We associate the random intervals $I^i=[S^{i-1},S^i)$ and $J^j=[T^{j-1},T^j)$ with the sampling scheme $(S^i)$ and $(T^j)$ and refer to $\mathcal{I}=(I^i)_{i\in\mathbb{N}}$ and $\mathcal{J}=(J^j)_{j\in\mathbb{N}}$ as the sampling designs for $X$ and $Y$. We introduce the \textit{pre-averaging observation data} of $X$ and $Y$ based on the sampling designs $\mathcal{I}$ and $\mathcal{J}$ respectively as follows:
\begin{align*}
\overline{\mathsf{X}}(\mathcal{I})^i=\sum_{p=1}^{k_n-1}g\left (\frac{p}{k_n}\right)\left(\mathsf{X}_{S^{i+p}}-\mathsf{X}_{S^{i+p-1}}\right),\qquad
\overline{\mathsf{Y}}(\mathcal{J})^j=\sum_{q=1}^{k_n-1}g\left (\frac{q}{k_n}\right)\left(\mathsf{Y}_{T^{j+q}}-\mathsf{Y}_{T^{j+q-1}}\right),\qquad
i,j=0,1,\dots.
\end{align*}

The following quantity was introduced in Christensen et al. \cite{CKP2010} :
\begin{Def}[Pre-averaged Hayashi-Yoshida estimator]\label{Defphy}
The \textit{pre-averaged Hayashi-Yoshida estimator}, or \textit{pre-averaged HY estimator} of $\mathsf{X}$ and $\mathsf{Y}$ associated with sampling designs $\mathcal{I}$ and $\mathcal{J}$ is the process
\begin{equation*}
PHY(\mathsf{X},\mathsf{Y};\mathcal{I},\mathcal{J})^n_t
=\frac{1}{(\psi_{HY}k_n)^2}\sum_{\begin{subarray}{c}
i,j=0\\
S^{i+k_n}\vee T^{j+k_n}\leq t
\end{subarray}}^{\infty}\overline{\mathsf{X}}(\mathcal{I})^i\overline{\mathsf{Y}}(\mathcal{J})^j 1_{\{[S^i,S^{i+k_n})\cap[T^j,T^{j+k_n})\neq\emptyset\}},\qquad t\in\mathbb{R}_+.
\end{equation*}
\end{Def}

\begin{rmk*}
In order to improve the performance of the above estimator in finite samples, it will be efficient to replace the quantity $\psi_{HY}$ with $\frac{1}{k_n}\sum_{p=1}^{k_n-1}g(p/k_n)$ in the above definition. Such a kind of adjustments often appears in the literature on pre-averaging estimators.
\end{rmk*}


As mentioned in Section \ref{intro}, we modify the above pre-averaged HY estimator by applying an interpolation method similar to the refresh sampling method for the technical reason. The following notion was introduced to this area in \cite{BNHLS2011}:
\begin{Def}[Refresh time]
The first refresh time of sampling designs $\mathcal{I}$ and $\mathcal{J}$ is defined as $R^0=S^0\vee T^0$, and then subsequent refresh times as
\begin{align*}
R^k:=\min\{S^i|S^i>R^{k-1}\}\vee\min\{T^j|T^j>R^{k-1}\},\qquad k=1,2,\dots.
\end{align*}
\end{Def}

We introduce new sampling scheme by a kind of the next-tick interpolations to the refresh times. That is, we define $\widehat{S}^0:=S^0$, $\widehat{T}^0:=T^0$, and
\begin{align*}
\widehat{S}^k:=\min\{S^i|S^i>R^{k-1}\},\quad\widehat{T}^k:=\min\{T^j|T^j>R^{k-1}\},\qquad k=1,2,\dots.
\end{align*}
Then, we create new sampling designs as follows:
\begin{align*}
\widehat{I}^k:=[\widehat{S}^{k-1},\widehat{S}^k),\qquad\widehat{J}^k:=[\widehat{T}^{k-1},\widehat{T}^k),\qquad
\widehat{\mathcal{I}}:=(\widehat{I}^i)_{i\in\mathbb{N}},\qquad\widehat{\mathcal{J}}:=(\widehat{J}^j)_{j\in\mathbb{N}}.
\end{align*}
For the sampling designs $\widehat{\mathcal{I}}$ and $\widehat{\mathcal{J}}$ obtained in such a manner, we will consider the pre-averaged Hayashi-Yoshida estimator $\widehat{PHY}(\mathsf{X},\mathsf{Y})^n:=PHY(\mathsf{X},\mathsf{Y};\widehat{\mathcal{I}},\widehat{\mathcal{J}})^n$.

One of the advantage of working with the refresh time is described by the following proposition:
\begin{prop}\label{advantage}
The following statements are true.
\begin{enumerate}[\normalfont (a)]
\item $\widehat{S}^k\vee\widehat{T}^k=R^k$ for every $k$.
\item $(\widehat{S}^i<\widehat{T}^j)\Rightarrow(i\leq j)$ and $(\widehat{S}^i>\widehat{T}^j)\Rightarrow(i\geq j)$ for every $i,j$.
\end{enumerate}
\end{prop}

\begin{proof}
(a) Obvious.

\noindent (b) Since $\widehat{T}^j\leq R^j<\widehat{S}^{j+1}$, $(\widehat{S}^i<\widehat{T}^j)$ implies $\widehat{S}^{i}<\widehat{S}^{j+1}$, hence $i\leq j$. Consequently, we obtain the former statement. By symmetry we also obtain the latter statement.
\end{proof}


\section{Main results}\label{main}
We start with introducing some notation and conditions in order to state our main result. We write the canonical decompositions of $X$, $Y$, $Z^X$ and $Z^Y$ as follows:
\begin{equation}\label{CSM}
X=A^X+M^X,\qquad Y=A^Y+M^Y,\qquad Z^X=\underline{A}^X+\underline{M}^X,\qquad Z^Y=\underline{A}^Y+\underline{M}^Y.
\end{equation}
Here, $A^X$, $A^Y$, $\underline{A}^X$ and $\underline{A}^Y$ are continuous $\mathbf{F}^{(0)}$-adapted processes with locally finite variations, while $M^X$, $M^Y$, $\underline{M}^X$ and $\underline{M}^Y$ are continuous $\mathbf{F}^{(0)}$-local martingales. For each $i,j\in\mathbb{Z}_+$, let 
\begin{align*}
\bar{I}^i=[\widehat{S}^i,\widehat{S}^{i+k_n}),\qquad\bar{J}^j=[\widehat{T}^j,\widehat{T}^{j+k_n}),\qquad
\bar{K}^{ij}=1_{\{\bar{I}^i\cap\bar{J}^j\neq\emptyset\}}.
\end{align*}
Since $(\bar{I}^i\cap\bar{J}^j\neq\emptyset)\Rightarrow(|i-j|\leq k_n)$ by Proposition \ref{advantage}(b), we have
\begin{equation}\label{sumbarK}
\sum_{j=0}^\infty\bar{K}^{ij}\leq 2 k_n+1,\qquad\sum_{i=0}^\infty\bar{K}^{ij}\leq 2 k_n+1
\end{equation}
for any $i,j\in\mathbb{Z}_+$.

For two real-valued bounded measurable functions $\alpha,\beta$ on $\mathbb{R}$, set
\begin{align*}
c_{\alpha,\beta}(p,q)&:=\frac{1}{k_n^2}\sum_{i=(p-k_n+1)\vee1}^p\sum_{j=(q-k_n+1)\vee1}^q\alpha\left(\frac{p-i}{k_n}\right)\beta\left(\frac{q-j}{k_n}\right)\bar{K}^{i j},\\
\psi_{\alpha,\beta}(x)&=\int_0^1\int_{x+u-1}^{x+u+1}\alpha(u)\beta(v)\mathrm{d}v\mathrm{d}u.
\end{align*}
The following lemma, which is a counterpart of Lemma 7.2 of \citet{CPV2011}, is the key to calculation of the asymptotic variance of the estimation error of our estimator:
\begin{lem}\label{key}
Let $\alpha,\beta:\mathbb{R}\rightarrow\mathbb{R}$ be two piecewise Lipschitz functions satisfying $\alpha(x)=\beta(x)=0$ with $x\notin[0,1]$. Then we have
\begin{equation*}
\sup_{p,q:p,q\geq k_n}\left|c_{\alpha,\beta}(p,q)-\psi_{\alpha,\beta}\left(\frac{q-p}{k_n}\right)\right|=O_p\left(b_n^{1/2}\right)
\end{equation*}
as $n\rightarrow\infty$.
\end{lem}

\begin{proof}
For $p,q\geq k_n$, we can rewrite $c_{\alpha,\beta}(p,q)$ as
\begin{align*}
c_{\alpha,\beta}(p,q)=\frac{1}{k_n^2}\sum_{i,j=0}^{k_n-1}\alpha\left(\frac{i}{k_n}\right)\beta\left(\frac{j}{k_n}\right)\bar{K}^{p-i, q-j}.
\end{align*}
By definition, we have $\bar{K}^{p-i, q-j}=1_{\{\widehat{S}^{p-i}<\widehat{T}^{q-j+k_n}, \widehat{S}^{p-i+k_n}>\widehat{T}^{q-j}\}}$. Moreover, by Proposition \ref{advantage}(b) we have
\begin{align*}
(\widehat{S}^{p-i}<\widehat{T}^{q-j+k_n}, \widehat{S}^{p-i+k_n}> \widehat{T}^{q-j})&\Rightarrow (q-p+i-k_n\leq j\leq q-p+i+k_n),\\
(q-p+i-k_n\leq j\leq q-p+i+k_n)&\Rightarrow (\widehat{S}^{p-i}\leq\widehat{T}^{q-j+k_n}, \widehat{S}^{p-i+k_n}\geq \widehat{T}^{q-j}).
\end{align*}
Hence we obtain
\begin{align*}
c_{\alpha,\beta}(p,q)
=\frac{1}{k_n^2}\sum_{i=0}^{k_n-1}\alpha\left(\frac{i}{k_n}\right)
\sum_{j=[(q-p+i-k_n)\wedge (k_n-1)]\vee 0}^{[(q-p+i+k_n)\vee0]\wedge (k_n-1)}\beta\left(\frac{j}{k_n}\right)+O_p(b_n^{1/2})
\end{align*}
uniformly in $p,q$. Note that $q-p+i-k_n\leq q-p+i+k_n$ and $\alpha(x)=\beta(x)=0$ if $x\notin[0,1]$, we have
\begin{align*}
c_{\alpha,\beta}(p,q)
=\frac{1}{k_n^2}\sum_{i=0}^{k_n-1}\alpha\left(\frac{i}{k_n}\right)
\sum_{j=q-p+i-k_n}^{q-p+i+k_n}\beta\left(\frac{j}{k_n}\right)+O_p(b_n^{1/2})
\end{align*}
uniformly in $p,q$. Therefore, the piecewise Lipschitz continuity of $\alpha$ and $\beta$ completes the proof.
\end{proof} 

Next, let $N^n_t=\sum_{k=1}^{\infty}1_{\{R^k\leq t\}}$, $N^{n,1}_t=\sum_{k=1}^{\infty}1_{\{\widehat{S}^k\leq t\}}$ and $N^{n,2}_t=\sum_{k=1}^{\infty}1_{\{\widehat{T}^k\leq t\}}$ for each $t\in\mathbb{R}_+$ and 
\begin{align*}
\Gamma^k=[R^{k-1},R^k),\qquad \check{I}^k:=[\check{S}^k,\widehat{S}^k),\qquad\check{J}^k:=[\check{T}^k,\widehat{T}^k)
\end{align*}
for each $k\in\mathbb{N}$. Here, for each $t\in\mathbb{R}_+$ we write $\check{S}^k=\sup_{S^i<\widehat{S}^k}S^i$ and $\check{T}^k=\sup_{T^j<\widehat{T}^k}T^j$. Note that $\check{S}^k$ and $\check{T}^k$ may not be stopping times.  

Let $\xi$ be a positive constant satisfying $\frac{1}{2}<\xi<1$. Furthermore, let $\mathbf{H}^n=(\mathcal{H}^n_t)_{t\in\mathbb{R}_+}$ be a sequence of filtrations of $\mathcal{F}$ to which $N^n$, $N^{n,1}$ and $N^{n,2}$ are adapted, and for each $n$ and each $\rho\geq0$ we define the processes $\chi^n$, $G(\rho)^n$, $F(\rho)^{n,1}$, $F(\rho)^{n,2}$ and $F(1)^{n,1* 2}$ by
\begin{gather*}
\chi^n_{s}=P(\widehat{S}^k=\widehat{T}^k\big|\mathcal{H}_{R^{k-1}}^n),\qquad
G(\rho)^n_s=E\left[\left(b_n^{-1}|\Gamma^{k}|\right)^\rho\big|\mathcal{H}_{R^{k-1}}^n\right],\\
F(\rho)^{n,1}_{s}=E\left[\left(b_n^{-1}|\check{I}^{k}|\right)^\rho\big|\mathcal{H}_{\widehat{S}^{k-1}}^n\right],\qquad
F(\rho)^{n,2}_{s}=E\left[\left(b_n^{-1}|\check{J}^{k}|\right)^\rho\big|\mathcal{H}_{\widehat{T}^{k-1}}^n\right],\\
F(1)^{n,1*2}_{s}=b_n^{-1}E\left[|\check{I}^{k}\cap\check{J}^k|+|\check{I}^{k+1}\cap\check{J}^k|+|\check{I}^{k}\cap\check{J}^{k+1}|\big|\mathcal{H}_{R^{k-1}}^n\right]
\end{gather*}
when $s\in\Gamma^k$.

The following condition is necessary to compute the asymptotic variance of the estimation error of our estimator explicitly. For a sequence $(X^n)$ of c\`adl\`ag processes and a c\`adl\`ag process $X$, we write $X^n\xrightarrow{\text{Sk.p.}}X$ if $(X^n)$ converges to $X$ in probability for the Skorokhod topology.
\begin{enumerate}
\item[{[A1$'$]}]

(i) For each $n$, we have a c\`adl\`ag $\mathbf{H}^{n}$-adapted process $G^n$ and a random subset $\mathcal{N}^0_n$ of $\mathbb{N}$ such that $(\#\mathcal{N}^0_n)_{n\in\mathbb{N}}$ is tight, $G(1)^n_{R^{k-1}}=G^n_{R^{k-1}}$ for any $k\in\mathbb{N}-\mathcal{N}^0_n$, and there exists a c\`adl\`ag $\mathbf{F}^{(0)}$-adapted process $G$ satisfying that $G$ and $G_{-}$ do not vanish and that $G^n\xrightarrow{\text{Sk.p.}}G$ as $n\to\infty$.

(ii) There exists a constant $\rho\geq1/\xi'$ such that $\left(\sup_{0\leq s\leq t}G(\rho)^n_{s}\right)_{n\in\mathbb{N}}$ is tight for all $t>0$.

(iii) For each $n$, we have a c\`adl\`ag $\mathbf{H}^{n}$-adapted process $\chi^{\prime n}$ and a random subset $\mathcal{N}'_n$ of $\mathbb{N}$ such that $(\#\mathcal{N}'_n)_{n\in\mathbb{N}}$ is tight, $\chi^n_{R^{k-1}}=\chi^{\prime n}_{R^{k-1}}$ for any $k\in\mathbb{N}-\mathcal{N}'_n$, and there exists a c\`adl\`ag $\mathbf{F}^{(0)}$-adapted process $\chi$ such that $\chi^{\prime n}\xrightarrow{\text{Sk.p.}}\chi$ as $n\to\infty$.

(iv) For each $n$ and $l=1,2,1*2$, we have a c\`adl\`ag $\mathbf{H}^{n}$-adapted process $F^{n,l}$ and a random subset $\mathcal{N}^l_n$ of $\mathbb{N}$ such that $(\#\mathcal{N}^l_n)_{n\in\mathbb{N}}$ is tight, $F(1)^{n,l}_{R^{k-1}}=F^{n,l}_{R^{k-1}}$ for any $k\in\mathbb{N}-\mathcal{N}^l_n$, and there exists a c\`adl\`ag $\mathbf{F}^{(0)}$-adapted processes $F^l$ satisfying $F^{n,l}\xrightarrow{\text{Sk.p.}}F^l$ as $n\to\infty$.

(v) There exists a constant $\rho'\geq1/\xi'$ such that $\left(\sup_{0\leq s\leq t}F(\rho')^{n,l}_{s}\right)_{n\in\mathbb{N}}$ is tight for all $t>0$ and $l=1,2$.

\end{enumerate}

\begin{rmk}
A kind of conditions such as [A1$'$](i)-(ii) and [A1$'$](iv)-(v) appears in \cite{BNHLS2011}, \cite{HJY2011} and \cite{PY2008}. The condition [A1$'$](iii) is satisfied when $(S^i)=(T^j)$ (a synchronous case) with $\chi\equiv1$ or when $S^i\neq T^j$ for all $i,j\geq 1$ (a completely nonsynchronous case) with $\chi\equiv0$, for example.
\end{rmk}

Next, we introduce the following strong predictability condition for the sampling designs, which is an analog to the condition [A2] in \cite{HY2011}.
\begin{enumerate}
\item[{[A2]}] For every $n,i\in\mathbb{N}$, $S^i$ and $T^i$ are $\mathbf{G}^{(n)}$-stopping times, where $\mathbf{G}^{(n)}=(\mathcal{G}^{(n)}_t)_{t\in\mathbb{R}_+}$ is the filtration given by $\mathcal{G}^{(n)}_t=\mathcal{F}_{(t-b_n^{\xi-1/2})_+}^{(0)}$ for $t\in\mathbb{R}_+$.
\end{enumerate}

The following conditions are analogs to the conditions [A3] and [A4] in \cite{HY2011}: 
\begin{enumerate}
\item[[{A3]}] For each $V,W=X,Y,Z^X,Z^Y$, $[V,W]$ is absolutely continuous with a c\`adl\`ag derivative, and for the density process $f=[V,W]'$ there is a sequence $(\sigma_k)$ of $\mathbf{F}^{(0)}$-stopping times such that $\sigma_k\uparrow\infty$ as $k\to\infty$ and for every $k$ and any $\lambda>0$ we have a positive constant $C_{k,\lambda}$ satisfying
\begin{equation}\label{eqA3}
E\left[|f^{\sigma_k}_{\tau_1}-f^{\sigma_k}_{\tau_2}|^2\big|\mathcal{F}_{\tau_1\wedge\tau_2}\right]\leq C_{k,\lambda}E\left[|\tau_1-\tau_2|^{1-\lambda}\big|\mathcal{F}_{\tau_1\wedge\tau_2}\right]
\end{equation}
for any bounded $\mathbf{F}^{(0)}$-stopping times $\tau_1$ and $\tau_2$, and $f$ is adapted to $\mathbf{H}^n$.
\item[{[A4]}] $\xi\vee\frac{9}{10}<\xi'$ and $(\ref{A4})$ holds for every $t\in\mathbb{R}_+$.
\end{enumerate}

The following conditions, which are analogs to the conditions [A5] and [A6] in \cite{HY2011}, are necessary to deal with the drift parts. For a (random) interval $I$ and a time $t$, we write $I(t)=I\cap[0,t)$.
\begin{enumerate}
\item[[{A5]}] $A^{X}$, $A^{Y}$, $\underline{A}^X$ and $\underline{A}^Y$ are absolutely continuous with c\`adl\`ag derivatives, and there is a sequence $(\sigma_k)$ of $\mathbf{F}^{(0)}$-stopping times such that $\sigma_k\uparrow\infty$ as $k\to\infty$ and for every $k$ we have a positive constant $C_{k}$ and $\lambda_k\in(0,3/4)$ satisfying
\begin{equation}\label{eqA5}
E\left[|f^{\sigma_k}_t-f^{\sigma_k}_{\tau}|^2\big|\mathcal{F}_{\tau\wedge t}\right]\leq C_{k}E\left[|t-\tau|^{1-\lambda_k}\big|\mathcal{F}_{\tau\wedge t}\right]
\end{equation}
for every $t>0$ and any bounded $\mathbf{F}^{(0)}$-stopping time $\tau$, for the density processes $f=(A^X)'$, $(A^{Y})'$, $(\underline{A}^X)'$ and $(\underline{A}^Y)'$.
\item[{[A6]}] For each $t\in\mathbb{R}_+$, $b_n^{-1}H_n(t)=O_p(1)$ as $n\to\infty$, where $H_n(t)=\sum_{k=1}^{\infty}|\Gamma^k(t)|^2$.
\end{enumerate}

The following condition is a regularity condition for the exogenous noise process:
\begin{enumerate}
\item[{[N]}] $(\int |z|^8Q_t(\mathrm{d}z))_{t\in\mathbb{R}_+}$ is a locally bounded process, and the covariance matrix process
\begin{equation}\label{defPsi}
\Psi_t(\omega^{(0)})=\int zz^*Q_t(\omega^{(0)},\mathrm{d}z).
\end{equation}
is c\`adl\`ag and quasi-left continuous. Furthermore, there is a sequence $(\sigma^k)$ of $\mathbf{F}^{(0)}$-stopping times such that $\sigma^k\uparrow\infty$ as $k\to\infty$ and for every $k$ and any $\lambda>0$ we have a positive constant $C_{k,\lambda}$ satisfying
\begin{equation}\label{eqN}
E\left[|\Psi^{ij}_{\sigma^k\wedge t}-\Psi^{ij}_{\sigma^k\wedge (t-h)_+}|^2\big|\mathcal{F}_{(t-h)_+}\right]\leq C_{k,\lambda} h^{1-\lambda}
\end{equation}
for every $i,j\in\{1,2\}$ and every $t,h>0$.
\end{enumerate}

\begin{rmk}
The inequalities $(\ref{eqA3})$, $(\ref{eqA5})$ and $(\ref{eqN})$ are satisfied when $w(f;h,t)=O_p(h^{\frac{1}{2}-\lambda})$ as $h\to\infty$ for every $t,\lambda\in(0,\infty)$, for example. Here, for a real-valued function $x$ on $\mathbb{R}_+$, the \textit{modulus of continuity} on $[0,T]$ is denoted by $w(x;\delta,T)=\sup\{|x(t)-x(s)|;s,t\in[0,T],|s-t|\leq\delta\}$ for $T,\delta>0$. This is the original condition in \cite{HY2011}. Another such example is the case that there exist an $\mathbf{F}^{(0)}$-adapted process $B$ with a locally integrable variation and a locally square-integrable martingale $L$ such that $f=B+L$ and both of the predictable compensator of the variation process of $B$ and predictable quadratic variation of $L$ are absolutely continuous with locally bounded derivatives. This type of condition is familiar in the context of the estimation of volatility-type quantities; see \cite{HJY2011} and \cite{JPV2010} for instance. Furthermore, in both of the cases $f$ is c\`adl\`ag and quasi-left continuous.
\end{rmk}

We extend the functions $g$ and $g'$ to the whole real line by setting $g(x)=g'(x)=0$ for $x\notin[0,1]$. Then we put
\begin{gather*}
\kappa:=\int_{-2}^{2}\psi_{g,g}(x)^2\mathrm{d}x,\qquad
\widetilde{\kappa}:=\int_{-2}^{2}\psi_{g',g'}(x)^2\mathrm{d}x,\qquad
\overline{\kappa}:=\int_{-2}^{2}\psi_{g,g'}(x)^2\mathrm{d}x.
\end{gather*}

We denote by $\mathbb{D}(\mathbb{R}_+)$ the space of c\`adl\`ag functions on $\mathbb{R}_+$ equipped with the Skorokhod topology. A sequence of random elements $X^n$ defined on a probability space $(\Omega,\mathcal{F},P)$ is said to \textit{converge stably in law} to a random element $X$ defined on an appropriate extension $(\tilde{\Omega},\tilde{\mathcal{F}} ,\tilde{P})$ of $(\Omega,\mathcal{F},P)$ if $E[Yg(X^n)]\rightarrow E[Yg(X)]$ for any $\mathcal{F}$-measurable and bounded random variable $Y$ and any bounded and continuous function $g$. We then write $X^n\rightarrow^{d_s}X$. A sequence $(X^n)$ of stochastic processes is said to converge to a process $X$ \textit{uniformly on compacts in probability} (abbreviated \textit{ucp}) if, for each $t>0$, $\sup_{0\leq s\leq t}|X^n_s-X_s|\rightarrow^p0$ as $n\rightarrow\infty$. We then write $X^n\xrightarrow{ucp}X$. 

Now we are ready to state our main result.
\begin{thm}\label{mainthm}

$(\mathrm{a})$ Suppose $[\mathrm{A}1'](\mathrm{i})$-$(\mathrm{iii})$, $[\mathrm{A}2]$-$[\mathrm{A}6]$ and $[\mathrm{N}]$ are satisfied. Suppose also $Z^X=Z^Y=0$. Then
\begin{align*}
b_n^{-1/4}\{\widehat{PHY}(\mathsf{X},\mathsf{Y})^n-[X,Y]\}\to^{d_s}\int_0^\cdot w_s\mathrm{d}\widetilde{W}_s\qquad\mathrm{in}\ \mathbb{D}(\mathbb{R}_+)
\end{align*}
as $n\to\infty$, where $\tilde{W}$ is a one-dimensional standard Wiener process (defined on an extension of $\mathcal{B}$) independent of $\mathcal{F}$ and $w$ is given by
\begin{align}
w_s^2=\psi_{HY}^{-4}[&\theta\kappa\{[X]'_s[Y]'_s+([X,Y]'_s)^2\}G_s
+\theta^{-3}\widetilde{\kappa}\{\Psi^{11}_s\Psi^{22}_s+\left(\Psi^{12}_s\chi_s\right)^2\}G_s^{-1}\nonumber\\
&+\theta^{-1}\overline{\kappa}\{[X]'_s\Psi^{22}_s+[Y]'_s\Psi^{11}_s+2[X,Y]'_s\Psi^{12}_s\chi_s\}].\label{avar}
\end{align}

\begin{enumerate}
\item[$(\mathrm{b})$] Suppose $[\mathrm{A}1']$, $[\mathrm{A}2]$-$[\mathrm{A}6]$ and $[\mathrm{N}]$ are satisfied. Then
\begin{align*}
b_n^{-1/4}\{\widehat{PHY}(\mathsf{X},\mathsf{Y})^n-[X,Y]\}\to^{d_s}\int_0^\cdot w_s\mathrm{d}\widetilde{W}_s\qquad\mathrm{in}\ \mathbb{D}(\mathbb{R}_+)
\end{align*}
as $n\to\infty$, where $\tilde{W}$ is as in the above and $w$ is given by
\begin{align}
w_s^2=\psi_{HY}^{-4}\bigg[&\theta\kappa\left\{[X]'_s[Y]'_s+([X,Y]'_s)^2\right\}G_s
+\theta^{-3}\widetilde{\kappa}\left\{\overline{\Psi}^{11}_s\overline{\Psi}^{22}_s+\left(\overline{\Psi}^{12}_s\right)^2\right\}G_s^{-1}\nonumber\\
&+\theta^{-1}\overline{\kappa}\left\{[X]'_s\overline{\Psi}^{22}_s+[Y]'_s\overline{\Psi}^{11}_s+2[X,Y]'_s\overline{\Psi}^{12}_s-\left([Z^X,Y]'_s F^1_s-[X,Z^Y]'_s F^2_s\right)^2 G_s^{-1}\right\}\Bigg],\label{avarend}
\end{align}
where
\begin{align*}
\overline{\Psi}^{11}_s=\Psi^{11}_s+[Z^X]'_s F^1_s,\qquad
\overline{\Psi}^{22}_s=\Psi^{22}_s+[Z^Y]'_s F^2_s,\qquad
\overline{\Psi}^{12}_s=\Psi^{12}_s\chi_s+[Z^X,Z^Y]'_s F^{1* 2}_s.
\end{align*}
\end{enumerate} 
\end{thm}

A sketch of the proof is given in the next section.


\section{Stable convergence of the estimation error}\label{secstable}
In this section we briefly sketch the proof of the main theorem. First we introduce some notation. For processes $V$ and $W$, $V\bullet W$ denotes the integral (either stochastic or ordinary) of $V$ with respect to $W$. For any semimartingale $V$ and any (random) interval $I$, we define the processes $V(I)_t$ and $I_t$ by $V(I)_t=\int_0^t1_I(s-)\mathrm{d}V_s$ and $I_t=1_I(t)$ respectively. We denote by $\Phi$ the set of all real-valued piecewise Lipschitz functions $\alpha$ on $\mathbb{R}$ satisfying $\alpha(x)=0$ for any $x\notin[0,1]$. For a function $\alpha$ on $\mathbb{R}$ we write $\alpha^n_p=\alpha(p/k_n)$ for each $n\in\mathbb{N}$ and $p\in\mathbb{Z}$. For any semimartingale $V$, any sampling design $\mathcal{D}=(D^i)_{i\in\mathbb{N}}$ and any $\alpha\in\Phi$, we define the process $\bar{V}(\mathcal{D})^i_t$ for each $i\in\mathbb{N}$ by
\begin{align*}
\bar{V}_\alpha(\mathcal{D})^i_t=\sum_{p=0}^{k_n-1}\alpha^n_p V(D^{i+p})_t.
\end{align*}


We introduce the following auxiliary regularity conditions:
\begin{enumerate}
\item[{[C1]}]  $A^X$, $A^Y$, $\underline{A}^X$, $\underline{A}^Y$, and $[V,W]$ for $V,W=X,Y,Z^X,Z^Y$ are absolutely continuous with locally bounded derivatives.
\item[{[C2]}] $(\int |z|^2Q_t(\mathrm{d}z))_{t\in\mathbb{R}_+}$ is a locally bounded process.
\item[{[C3]}] $b_nN^n_t=O_p(1)$ as $n\to\infty$ for every $t$.
\end{enumerate}
We define $\Psi$ by $(\ref{defPsi})$ whenever we have [C2].

 
Let
\begin{align*}
\mathfrak{E}^X_t=-\frac{1}{k_n}\sum_{p=1}^{\infty}\epsilon^X_{\widehat{S}^p}1_{\{\widehat{S}^p\leq t\}},\qquad
\mathfrak{E}^Y_t=-\frac{1}{k_n}\sum_{q=1}^{\infty}\epsilon^Y_{\widehat{T}^q}1_{\{\widehat{T}^q\leq t\}}.
\end{align*}
$\mathfrak{E}^X$ and $\mathfrak{E}^Y$ are obviously purely discontinuous locally square-integrable martingales on $\mathcal{B}$ if [C2] holds (note that both $(\widehat{S}^i)$ and $(\widehat{T}^j)$ are $\mathbf{F}^{(0)}$-stopping times). Furthermore, if $\Psi$ is c\`adl\`ag, quasi-left continuous and both $(S^i)$ and $(T^j)$ are $\mathbf{F}^{(0)}$-predictable times, then we have
\begin{gather*}
\langle\mathfrak{E}^X\rangle_t=\frac{1}{k_n^2}\sum_{p=1}^{\infty}\Psi^{11}_{\widehat{S}^p}1_{\{\widehat{S}^p\leq t\}},\qquad
\langle\mathfrak{E}^Y\rangle_t=\frac{1}{k_n^2}\sum_{q=1}^{\infty}\Psi^{22}_{\widehat{T}^q}1_{\{\widehat{T}^q\leq t\}},\qquad
\langle\mathfrak{E}^X,\mathfrak{E}^Y\rangle_t=\frac{1}{k_n^2}\sum_{p,q=1}^{\infty}\Psi^{12}_{\widehat{S}^p}1_{\{\widehat{S}^p=\widehat{T}^q\leq t\}}.
\end{gather*}


Though $\check{S}^k$ and $\check{T}^k$ may not be stopping times, we have the following result:
\begin{lem}\label{check}
The random variables $\check{I}^k_t$ and $\check{J}^k_t$ are $\mathbf{F}^{(0)}_t$-measurable for every $k,t$.
\end{lem}

\begin{proof}
Since $\{\check{I}^k_t=1\}=\{\check{S}^k\leq t<\widehat{S}^k\}=\bigcap_i[\{S^i\leq t<\widehat{S}^k\}\cup\{t<\widehat{S}^k\leq S^i\}]$, we obtain $\{\check{I}^k_t=1\}\in\mathcal{F}^{(0)}_t$ and thus $\check{I}^k_t$ is $\mathcal{F}^{(0)}_t$-measurable. Similarly we can show that $\check{J}^k_t$ is $\mathcal{F}^{(0)}_t$-measurable.
\end{proof}

Due to the above lemma, both of the processes $\mathfrak{I}_t:=\sum_{p=1}^\infty\check{I}^p_t$ and $\mathfrak{J}_t:=\sum_{q=1}^\infty\check{J}^q_t$ are $\mathbf{F}^{(0)}$-adapted. Therefore, we can define the following processes:
\begin{align*}
\mathfrak{Z}^X_t=-\mathfrak{I}_-\bullet Z^X_t,\qquad
\mathfrak{Z}^Y_t=-\mathfrak{J}_-\bullet Z^Y_t.
\end{align*}
Then we set
\begin{align*}
\mathfrak{U}^X=\mathfrak{E}^X+(k_n\sqrt{b_n})^{-1}\mathfrak{Z}^X,\qquad
\mathfrak{U}^Y=\mathfrak{E}^Y+(k_n\sqrt{b_n})^{-1}\mathfrak{Z}^Y.
\end{align*} 


For any semimartingales $V,W$ and any $\alpha,\beta\in\Phi$, set
\begin{align*}
\bar{L}_{\alpha,\beta}(V,W)^{ij}=\bar{V}_\alpha(\widehat{\mathcal{I}}^i)_-\bullet \bar{W}_\beta(\widehat{\mathcal{J}}^j)+\bar{W}_\beta(\widehat{\mathcal{J}}^j)_-\bullet \bar{V}_\alpha(\widehat{\mathcal{I}}^i)
\end{align*}
for each $i,j\in\mathbb{N}$. We define the process $\mathbf{M}(k)^n_t\ (k=1,2,3,4)$ by
\begin{align*}
&\mathbf{M}(1)^n_t=\frac{1}{(\psi_{HY}k_n)^2}\sum_{i,j=1}^{\infty}\bar{K}^{ij}_t\bar{L}_{g,g}(X,Y)^{ij}_t,\qquad
\mathbf{M}(2)^n_t=\frac{1}{(\psi_{HY}k_n)^2}\sum_{i,j=1}^{\infty}\bar{K}^{ij}_t\bar{L}_{g',g'}(\mathfrak{U}^X,\mathfrak{U}^Y)^{ij}_t,\\
&\mathbf{M}(3)^n_t=\frac{1}{(\psi_{HY}k_n)^2}\sum_{i,j=1}^{\infty}\bar{K}^{ij}_t\bar{L}_{g,g'}(X,\mathfrak{U}^Y)^{ij}_t,\qquad
\mathbf{M}(4)^n_t=\frac{1}{(\psi_{HY}k_n)^2}\sum_{i,j=1}^{\infty}\bar{K}^{ij}_t\bar{L}_{g',g}(\mathfrak{U}^X,Y)^{ij}_t,
\end{align*}
where $\bar{K}^{ij}_t=1_{\{\bar{I}^i(t)\cap\bar{J}^j(t)\neq\emptyset\}}$, and set $\mathbf{M}^n_t=\sum_{k=1}^4\mathbf{M}(k)^n_t$. Then we have the following lemma.


\begin{lem}\label{lembasic}
Suppose that $(\ref{A4})$ and $[\mathrm{C}1]$-$[\mathrm{C}3]$ are satisfied. Then
\begin{align*}
b_n^{-\gamma}\left\{ \widehat{PHY}(\mathsf{X}, \mathsf{Y})^n-[X,Y]-\mathbf{M}^n\right\}\xrightarrow{ucp}0
\end{align*}
as $n\rightarrow\infty$ for any $\gamma<\xi'-1/2$.
\end{lem}

We give a proof of Lemma \ref{lembasic} in Section \ref{prooflembasic}. The above lemma implies that we may consider $\mathbf{M}^n$ instead of the estimation error of our estimator as far as $\xi'>3/4$. 


\begin{lem}\label{HYlem3.1}
Let $V,W$ be two semimartingales and $g,h$ are two real functions on $\mathbb{R}$. For any $i,j\in\mathbb{Z}_+$ and any $t\in\mathbb{R}_+$, $\bar{K}^{ij}_t\bar{L}_{g,h}(V,W)_t=\bar{K}^{ij}_-\bullet\bar{L}_{g,h}(V,W)_t$.
\end{lem}

\begin{proof}
By integration by parts we have $$\bar{K}^{ij}_t\bar{L}_{g,h}(V,W)_t=\bar{K}^{ij}_-\bullet\bar{L}_{g,h}(V,W)_t+\bar{L}_{g,h}(V,W)_-\bullet\bar{K}^{ij}_t+[\bar{K}^{ij}, \bar{L}_{g,h}(V,W)]_t,$$ hence it is sufficient to show that $\bar{L}_{g,h}(V,W)_-\bullet\bar{K}^{ij}_t=[\bar{K}^{ij}, \bar{L}_{g,h}(V,W)]_t=0$. $\bar{K}^{ij}_t$ is a step function starting from 0 at $t=0$ and jumps to $+1$ at $t=R^\vee(i,j)$ when $\bar{I}^i\cap\bar{J}^j\neq\emptyset$. So, $\bar{L}_{g,h}(V,W)_-\bullet\bar{K}^{ij}_t=\bar{L}_{g,h}(V,W)_{R^\vee(i,j)\wedge t-}\bar{K}^{ij}_t$ and $[\bar{K}^{ij}, \bar{L}_{g,h}(V,W)]_t=\bar{K}^{ij}_t \Delta\bar{L}_{g,h}(V,W)_{R^\vee(i,j)\wedge t}$. However, $\bar{L}_{g,h}(V,W)_t=0$ for $t\leq R^\vee(i,j)$ by its definition.
\end{proof}


We momentarily assume that $X$, $Y$, $Z^X$ and $Z^Y$ are continuous local martingales. Lemma \ref{HYlem3.1} implies that $\mathbf{M}^n$ is a locally square-integrable martingale. Therefore, we can define the quantities $\mathfrak{V}_t^n:=\langle\mathbf{M}^n\rangle_t$ and $\mathfrak{V}_{N,t}^n:=\langle \mathbf{M}^n,N\rangle_t$ for a locally square-integrable martingale $N$. Then we consider the following conditions.
\begin{enumerate}
\item[{[A$1^*$]}] There exists an $\mathbf{F}$-adapted, nondecreasing, continuous process $(V_t)_{t\in\mathbb{R}_+}$ such that $b_n^{-1/2}\mathfrak{V}^n_t\to^pV_t$ as $n\to\infty$ for every $t$.
\item[{[B1]}] $b_n^{-1/4}\mathfrak{V}_{N,t}^n\to^p0$ as $n\to\infty$ for every $t$ and any $N\in\{X,Y,Z^X,Z^Y\}$.
\end{enumerate}

The next two lemmas imply that the above two conditions are sufficient for our stable convergence problem. They are proved in Section \ref{prooforthogonal} and \ref{proofLindeberg}, respectively.


\begin{lem}\label{orthogonal}
Suppose that $[\mathrm{A}1^*]$, $[\mathrm{A}4]$ and $[\mathrm{C}1]$-$[\mathrm{C}3]$ are satisfied. Then for any square-integrable martingale $N$ orthogonal to $(X,Y,Z^X,Z^Y)$ we have $b_n^{-1/4}\langle \mathbf{M}^n,N\rangle_t\to^p0$ as $n\to\infty$ for every $t$.
\end{lem}


\begin{lem}\label{Lindeberg}
Suppose that $[\mathrm{A}4]$, $[\mathrm{C}1]$-$[\mathrm{C}3]$ and $[\mathrm{N}]$ are satisfied. Then
$\sum_{s:0\leq s\leq t}|b_n^{-1/4}\Delta\mathbf{M}^n_s|^4\to^p0$
as $n\to\infty$ for any $t>0$.
\end{lem}


We consider the following conditions:
\begin{enumerate}
\item[{[W]}] There exists an $\mathbf{F}$-predictable process $w$ such that $V_{\cdot}=\int_0^{\cdot}w_s^2\mathrm{d}s$.
\item[{[SC]}] $b_n^{-1/4}\mathbf{M}^n\to^{d_s}\mathbf{M}$ in $\mathbb{D}(\mathbb{R}_+)$ as $n\rightarrow\infty$, where $\mathbf{M}=\int_0^{\cdot}w_s\mathrm{d}\tilde{W}_s$, $w$ is a some predictable process, and $\tilde{W}$ is a one-dimensional standard Wiener process (defined on an extension of $\mathcal{B}$) independent of $\mathcal{F}$.
\end{enumerate}

\begin{prop}\label{jacod}
Suppose that $[\mathrm{A}1^*]$, $[\mathrm{B}1]$, $[\mathrm{A}4]$, $[\mathrm{C}1]$-$[\mathrm{C}3]$, $[\mathrm{N}]$ and $[\mathrm{W}]$ are fulfilled. Then $[\mathrm{SC}]$ holds.
\end{prop}

\begin{proof}
We apply Theorem 2-2 of \cite{J1997}. Eq.~(2.8), (2.9), (2.10) and (2.12) in \cite{J1997} are satisfied with $B=0, F=V$ and $G=0$ due to the assumptions and Lemma \ref{orthogonal}. Moreover, Lemma \ref{Lindeberg} yields Eq.~(2.26) in \cite{J1997} due to the definition of the compensator of a random measure. Consequently, we complete the proof.
\end{proof} 


The left problem is to check the conditions [A1$^*$], [B1] and [W]. Following \cite{HY2011}, we introduce some notation and conditions. For any locally square-integrable martingales $M$, $N$, $M'$, $N'$ and any $\alpha,\beta,\alpha',\beta'\in\Phi$, let
\begin{align*}
&V^{iji'j'}_{\alpha,\beta;\alpha',\beta'}(M,N;M',N')_t\\
:=&\langle\bar{M}_{\alpha}(\widehat{\mathcal{I}})^i,\bar{M}'_{\alpha'}(\widehat{\mathcal{I}})^{i'}\rangle_t\langle\bar{N}_\beta(\widehat{\mathcal{J}})^j,\bar{N}'_{\beta'}(\widehat{\mathcal{J}})^{j'}\rangle_t+\langle\bar{M}_\alpha(\widehat{\mathcal{I}})^i,\bar{N}'_{\beta'}(\widehat{\mathcal{J}})^{j'}\rangle_t\langle\bar{M}'_{\alpha'}(\widehat{\mathcal{I}})^{i'},\bar{N}_\beta(\widehat{\mathcal{J}})^{j}\rangle_t.
\end{align*}
Then we introduce the following condition:
\begin{enumerate}
\item[{[B2]}] For any $M,M'\in\{X,\mathfrak{E}^X,\mathfrak{Z}^X\}$, any $N,N'\in\{Y,\mathfrak{E}^Y,\mathfrak{Z}^Y\}$ and any $\alpha,\beta,\alpha',\beta'\in\Phi$,
\begin{align*}
&b_n^{-1/2}\sum_{i,j,i',j'}(\bar{K}^{ij}_-\bar{K}^{i'j'}_-)\bullet\langle\bar{L}_{\alpha,\beta}^{ij}(M,N),\bar{L}_{\alpha',\beta'}^{i'j'}(M',N')\rangle_t\\
=&b_n^{-1/2}\sum_{i,j,i',j'}(\bar{K}^{ij}_-\bar{K}^{i'j'}_-)\bullet V^{iji'j'}_{\alpha,\beta;\alpha',\beta'}(M,N;M',N')_t+o_p\left( k_n^4\right)
\end{align*}
as $n\to\infty$ for every $t\in\mathbb{R}_+$.
\end{enumerate}

Let
\begin{align*}
\bar{V}^{n,1}_t
=&\frac{1}{\psi_{HY}^4}\sum_{p,q}\psi_{g,g}\left(\frac{q-p}{k_n}\right)^2\left\{[X](\widehat{I}^p)_t[Y](\widehat{J}^q)_t+[X,Y](\widehat{I}^p)_t[X,Y](\widehat{J}^q)_t\right\},\\
\bar{V}^{n,2}_t
=&\frac{1}{(\psi_{HY}k_n)^4}\sum_{p,q}\psi_{g',g'}\left(\frac{q-p}{k_n}\right)^2\Big\{\left(\Psi^{11}_{\widehat{S}^p}+b_n^{-1}[Z^X](\check{I}^p)_t\right)\left(\Psi^{22}_{\widehat{T}^q}+b_n^{-1}[Z^Y](\check{J}^q)_t\right)\\
&+\left(\Psi^{12}_{R^p}1_{\{\widehat{S}^p=\widehat{T}^p\leq t\}}+b_n^{-1}[\mathfrak{Z}^X,\mathfrak{Z}^Y](\widehat{I}^p)_t\right)\left(\Psi^{12}_{R^q}1_{\{\widehat{S}^q=\widehat{T}^q\leq t\}}+b_n^{-1}[\mathfrak{Z}^X,\mathfrak{Z}^Y](\widehat{J}^q)_t\right)\Big\},\\
\bar{V}^{n,3}_t
=&\frac{1}{\psi_{HY}^4 k_n^2}\sum_{p,q}\psi_{g,g'}\left(\frac{q-p}{k_n}\right)^2\Big\{[X](\widehat{I}^p)_t\left(\Psi^{22}_{\widehat{T}^q}1_{\{\widehat{T}^q\leq t\}}+b_n^{-1}[Z^Y](\check{J}^q)_t\right)-b_n^{-1}[X,Z^Y](\check{J}^p)_t[X,Z^Y](\check{J}^q)_t\Big\},\\
\bar{V}^{n,4}_t
=&\frac{1}{\psi_{HY}^4 k_n^2}\sum_{p,q}\psi_{g,g'}\left(\frac{q-p}{k_n}\right)^2\Big\{[Y](\widehat{J}^q)_t\left(\Psi^{11}_{\widehat{S}^p}1_{\{\widehat{S}^p\leq t\}}+b_n^{-1}[Z^X](\check{I}^p)_t\right)-b_n^{-1}[Y,Z^X](\check{I}^p)_t[Y,Z^X](\check{I}^q)_t\Big\}
\end{align*}
and
\begin{align*}
\bar{V}^{n,12}_t
=&\frac{b_n^{-1}}{\psi_{HY}^4 k_n^2}\sum_{p,q}\psi_{g',g'}\left(\frac{q-p}{k_n}\right)\psi_{g,g}\left(\frac{q-p}{k_n}\right)\{[Z^X,X](\check{I}^p)_t[Z^Y,Y](\check{J}^q)_t+[Z^X,Y](\check{I}^p)_t[X,Z^Y](\check{J}^q)_t\},\\
\bar{V}^{n,34}_t
=&\frac{1}{\psi_{HY}^4 k_n^2}\sum_{p,q}\psi_{g,g'}\left(\frac{q-p}{k_n}\right)^2\Big\{-b_n^{-1}[X,Z^X](\check{I}^p)_t[Z^Y,Y](\check{J}^q)_t\\
&\hphantom{\frac{1}{\psi_{HY}^4 k_n^2}\sum_{p,q}\psi_{g,g'}\left(\frac{q-p}{k_n}\right)^2\Big\{}+[X,Y](\widehat{I}^p)_t\left(\Psi^{12}_{R^q}1_{\{\widehat{S}^q=\widehat{T}^q\leq t\}}+b_n^{-1}[\mathfrak{Z}^X,\mathfrak{Z}^Y](\widehat{J}^p)_t)\right)\Big\},
\end{align*}
and set $\bar{V}^n_t=\sum_{l=1}^4\bar{V}^{n,l}_t+2(\bar{V}^{n,12}_t+\bar{V}^{n,34}_t)$. Then we have the following proposition, which enables us to work with $\bar{V}^n$, a more tractable process than $\mathfrak{V}^n$. The proof is given in Section \ref{proofHYprop3.2}.
\begin{prop}\label{HYprop3.2}
Suppose that $[\mathrm{A}4]$, $[\mathrm{C}1]$-$[\mathrm{C}3]$ and $[\mathrm{B}2]$ hold. Suppose also that $\Psi$ is c\`adl\`ag and quasi-left continuous. Then $\mathfrak{V}^n_t=\bar{V}^n_t+o_p(b_n^{1/2})$ as $n\to\infty$ for all $t\in\mathbb{R}_+$.
\end{prop}


We modify [A1$^*$].
\begin{enumerate}
\item[{[A1]}] There exists an $\mathbf{F}$-adapted, nondecreasing, continuous process $(V_t)_{t\in\mathbb{R}_+}$ such that $b_n^{-1/2}\bar{V}^n_t\to V_t$ as $n\to\infty$ for every $t$.
\end{enumerate}

By Proposition \ref{HYprop3.2}, we can rephrase Proposition \ref{jacod} as follows.
\begin{prop}\label{jacod2}
Suppose that $[\mathrm{A}1]$, $[\mathrm{A}4]$, $[\mathrm{B}1]$, $[\mathrm{B}2]$, $[\mathrm{C}1]$-$[\mathrm{C}3]$, $[\mathrm{N}]$ and $[\mathrm{W}]$ for $V$ in $[\mathrm{A}1]$ are satisfied. Then $[\mathrm{SC}]$ holds.
\end{prop}



The condition [W] for $V$ in [A1] can be checked by the following lemma, which is proved in Section \ref{proofHYlem4.1}.
\begin{lem}\label{HYlem4.1}
Suppose that $[\mathrm{A}1'](\mathrm{i})$-$(\mathrm{iii})$, $[\mathrm{A}3]$, $[\mathrm{A}4]$ and $[\mathrm{C}2]$ hold. Suppose also that $\Psi$ is c\`adl\`ag and adapted to $\mathbf{H}^n$ for every $n$. Then
\begin{enumerate}[\normalfont (a)]
\item $b_n^{-1/2}\sum_{p,q=1}^{\infty}\psi_{g,g}\left(\frac{q-p}{k_n}\right)^2\langle X\rangle(\widehat{I}^p)_t\langle Y\rangle(\widehat{J}^q)_t\to^p\theta\kappa\int_0^t\langle X\rangle'_s\langle Y\rangle'_sG_s\mathrm{d}s$,
\item $b_n^{-1/2}\sum_{p,q=1}^{\infty}\psi_{g,g}\left(\frac{q-p}{k_n}\right)^2\langle X,Y\rangle(\widehat{I}^p)_t\langle X, Y\rangle(\widehat{J}^q)_t\to^p\theta\kappa\int_0^t(\langle X,Y\rangle'_s)^2G_s\mathrm{d}s$,
\item $b_n^{-1/2}\frac{1}{k_n^4}\sum_{p,q=1}^{\infty}\psi_{g',g'}\left(\frac{q-p}{k_n}\right)^2\Psi^{11}_{\widehat{S}^p}\Psi^{22}_{\widehat{T}^q}1_{\{\widehat{S}^p\vee\widehat{T}^q\leq t\}}\to^p
\theta^{-3}\widetilde{\kappa}\int_0^t\Psi^{11}_s\Psi^{22}_s G_s^{-1}\mathrm{d}s$,
\item $b_n^{-1/2}\frac{1}{k_n^4}\sum_{p,q=1}^{\infty}\psi_{g',g'}\left(\frac{q-p}{k_n}\right)^2\Psi^{12}_{\widehat{S}^p}1_{\{\widehat{S}^p=\widehat{T}^p\leq t\}}\Psi^{21}_{\widehat{T}^q}1_{\{\widehat{S}^q=\widehat{T}^q\leq t\}}
\to^p\theta^{-3}\widetilde{\kappa}\int_0^t\left(\Psi^{12}_s\chi_s\right)^2 G_s^{-1}\mathrm{d}s$,
\item $b_n^{-1/2}\frac{1}{k_n^2}\sum_{p,q=1}^{\infty}\psi_{g,g'}\left(\frac{q-p}{k_n}\right)^2\langle X\rangle(\widehat{I}^p)_t\Psi^{22}_{\widehat{T}^q}1_{\{\widehat{T}^q\leq t\}}\to^p\theta^{-1}\overline{\kappa}\int_0^t\langle X\rangle'_s\Psi^{22}_s\mathrm{d}s$,
\item $b_n^{-1/2}\frac{1}{k_n^2}\sum_{p,q=1}^{\infty}\psi_{g,g'}\left(\frac{p-q}{k_n}\right)^2\langle Y\rangle(\widehat{J}^q)_t\Psi^{11}_{\widehat{S}^p}1_{\{\widehat{S}^p\leq t\}}\to^p\theta^{-1}\overline{\kappa}\int_0^t\langle Y\rangle'_s\Psi^{11}_s\mathrm{d}s$
\item $b_n^{-1/2}\frac{1}{k_n^2}\sum_{p,q=1}^{\infty}\psi_{g,g'}\left(\frac{q-p}{k_n}\right)^2\langle X,Y\rangle(\widehat{I}^p)_t\Psi^{12}_{R^q}1_{\{\widehat{S}^q=\widehat{T}^q\leq t\}}\to^p\theta^{-1}\overline{\kappa}\int_0^t\langle X,Y\rangle'_s\Psi^{12}_s\chi_s\mathrm{d}s$
\end{enumerate}
as $n\to\infty$ for every $t$. Furthermore, if $[\mathrm{A}1'](\mathrm{iv})$-$(\mathrm{v})$ hold in addition to the above assumption, then
\begin{enumerate}[\normalfont (a)]\setcounter{enumi}{7}
\item $b_n^{-1/2}\frac{b_n^{-2}}{k_n^4}\sum_{p,q=1}^{\infty}\psi_{g',g'}\left(\frac{q-p}{k_n}\right)^2[Z^X](\check{I}^p)_t [Z^Y](\check{J}^q)_t\to^p
\theta^{-3}\widetilde{\kappa}\int_0^t[Z^X]'_s[Z^Y]'_s F^1_s F^2_s G_s^{-1}\mathrm{d}s$,
\item $b_n^{-1/2}\frac{b_n^{-2}}{k_n^4}\sum_{p,q=1}^{\infty}\psi_{g',g'}\left(\frac{q-p}{k_n}\right)^2[\mathfrak{Z}^X,\mathfrak{Z}^Y](\widehat{I}^p)_t [\mathfrak{Z}^X,\mathfrak{Z}^Y](\widehat{J}^q)_t\to^p
\theta^{-3}\widetilde{\kappa}\int_0^t([Z^X,Z^Y]'_s F^{1* 2}_s)^2 G_s^{-1}\mathrm{d}s$,
\item $b_n^{-1/2}\frac{b_n^{-1}}{k_n^4}\sum_{p,q=1}^{\infty}\psi_{g',g'}\left(\frac{q-p}{k_n}\right)^2\Psi^{11}_{\widehat{S}^p}1_{\{\widehat{S}^p\leq t\}}[Z^Y](\check{J}^q)_t\to^p
\theta^{-3}\widetilde{\kappa}\int_0^t\Psi^{11}_s[Z^Y]'_s F^{2}_s G_s^{-1}\mathrm{d}s$,
\item $b_n^{-1/2}\frac{b_n^{-1}}{k_n^4}\sum_{p,q=1}^{\infty}\psi_{g',g'}\left(\frac{q-p}{k_n}\right)^2[Z^X](\check{I}^p)_t\Psi^{22}_{\widehat{T}^q}1_{\{\widehat{T}^q\leq t\}}\to^p
\theta^{-3}\widetilde{\kappa}\int_0^t[Z^X]'_s\Psi^{22}_s F^{1}_s G_s^{-1}\mathrm{d}s$,
\item $b_n^{-1/2}\frac{b_n^{-1}}{k_n^2}\sum_{p,q=1}^{\infty}\psi_{g,g'}\left(\frac{q-p}{k_n}\right)^2[X](\widehat{I}^p)_t[Z^Y](\check{J}^q)_t\to^p
\theta^{-1}\overline{\kappa}\int_0^t[X]'_s[Z^Y]'_s F^{2}_s\mathrm{d}s$,
\item $b_n^{-1/2}\frac{b_n^{-1}}{k_n^2}\sum_{p,q=1}^{\infty}\psi_{g,g'}\left(\frac{q-p}{k_n}\right)^2[Z^X](\check{I}^p)_t[Y](\widehat{J}^q)_t\to^p
\theta^{-1}\overline{\kappa}\int_0^t[Z^X]'_s[Y]'_s F^{1}_s\mathrm{d}s$,
\item $b_n^{-1/2}\frac{b_n^{-1}}{k_n^2}\sum_{p,q=1}^{\infty}\psi_{g,g'}\left(\frac{q-p}{k_n}\right)^2[X,Z^Y](\check{J}^p)_t[X,Z^Y](\check{J}^q)_t\to^p
\theta^{-1}\overline{\kappa}\int_0^t([X,Z^Y]'_s F^{2}_s)^2 G^{-1}_s\mathrm{d}s$,
\item $b_n^{-1/2}\frac{b_n^{-1}}{k_n^2}\sum_{p,q=1}^{\infty}\psi_{g,g'}\left(\frac{q-p}{k_n}\right)^2[Z^X,Y](\check{I}^p)_t[Z^X,Y](\check{I}^q)_t\to^p
\theta^{-1}\overline{\kappa}\int_0^t([Z^X,Y]'_s F^{1}_s)^2 G^{-1}_s\mathrm{d}s$,
\item {\small $b_n^{-1/2}\frac{b_n^{-1}}{k_n^2}\sum_{p,q=1}^{\infty}\psi_{g',g'}\left(\frac{q-p}{k_n}\right)\psi_{g,g}\left(\frac{q-p}{k_n}\right)[Z^X,X](\check{I}^p)_t[Z^Y,Y](\check{J}^q)_t
\to^p\theta^{-1}\overline{\kappa}\int_0^t[Z^X,X]'_s [Z^Y,Y]'_s F^{1}_s F^2_s G^{-1}_s\mathrm{d}s$},
\item {\small $b_n^{-1/2}\frac{b_n^{-1}}{k_n^2}\sum_{p,q=1}^{\infty}\psi_{g',g'}\left(\frac{q-p}{k_n}\right)\psi_{g,g}\left(\frac{q-p}{k_n}\right)[Z^X,Y](\check{I}^p)_t[X,Z^Y](\check{J}^q)_t
\to^p\theta^{-1}\overline{\kappa}\int_0^t[Z^X,Y]'_s [X,Z^Y]'_s F^{1}_s F^2_s G^{-1}_s\mathrm{d}s$},
\item $b_n^{-1/2}\frac{b_n^{-1}}{k_n^2}\sum_{p,q=1}^{\infty}\psi_{g,g'}\left(\frac{q-p}{k_n}\right)[X,Z^X](\check{I}^p)_t[Z^Y,Y](\check{J}^q)_t\to^p
\theta^{-1}\overline{\kappa}\int_0^t[X,Z^X]'_s [Y,Z^Y]'_s F^{1}_s F^2_s G^{-1}_s\mathrm{d}s$,
\item $b_n^{-1/2}\frac{b_n^{-1}}{k_n^2}\sum_{p,q=1}^{\infty}\psi_{g,g'}\left(\frac{q-p}{k_n}\right)^2[X,Y](\widehat{I}^p)_t[\mathfrak{Z}^X,\mathfrak{Z}^Y](\widehat{J}^q)_t\to^p
\theta^{-1}\overline{\kappa}\int_0^t[X,Y]'_s [Z^X,Z^Y]'_s F^{1* 2}_s\mathrm{d}s$
\end{enumerate}
as $n\to\infty$ for every $t$.
\end{lem} 



The following proposition, which is an analog to Proposition 5.1 in \cite{HY2011}, gives a sufficient condition for [B2]. The proof is given in Section \ref{proofHYprop5.1}.
\begin{prop}\label{HYprop5.1}
$[\mathrm{B}2]$ holds true under $[\mathrm{A}2]$-$[\mathrm{A}4]$, $[\mathrm{N}]$ and $[\mathrm{C}3]$.
\end{prop}



It still remains to check the asymptotic orthogonality condition [B1]. However, it will be shown that it is the same kind of task as solving [B2]. This phenomenon is also seen in \cite{HY2011}.
\begin{thm}\label{HYthm6.1}
Suppose that $X$, $Y$, $Z^X$ and $Z^Y$ are continuous semimartingales given by $(\ref{CSM})$.
\begin{enumerate}[\normalfont (a)]
\item If $[\mathrm{A}1]$-$[\mathrm{A}6]$, $[\mathrm{N}]$, $[\mathrm{C}3]$ and $[\mathrm{W}]$ are satisfied, then $[\mathrm{SC}]$ holds.
\item If  $Z^X=Z^Y=0$, $[\mathrm{A}1'](\mathrm{i})$-$(\mathrm{iii})$, $[\mathrm{A}2]$-$[\mathrm{A}6]$ and $[\mathrm{N}]$ are satisfied, then $[\mathrm{SC}]$ holds for $w$ given by $(\ref{avar})$.
\item If $[\mathrm{A}1']$, $[\mathrm{A}2]$-$[\mathrm{A}6]$ and $[\mathrm{N}]$ are satisfied, then $[\mathrm{SC}]$ holds for $w$ given by $(\ref{avarend})$.
\end{enumerate}
\end{thm}


It is worthy of remark that neither [A5] nor [A6] is necessary for local martingales as seen in \cite{HY2011}.
\begin{thm}\label{HYthm6.2}
Suppose that $X$, $Y$, $Z^X$ and $Z^Y$ are continuous local martingales.
\begin{enumerate}[\normalfont (a)]
\item If $[\mathrm{A}1]$-$[\mathrm{A}4]$, $[\mathrm{N}]$, $[\mathrm{C}3]$ and $[\mathrm{W}]$ are satisfied, then $[\mathrm{SC}]$ holds.
\item If  $Z^X=Z^Y=0$, $[\mathrm{A}1'](\mathrm{i})$-$(\mathrm{iii})$, $[\mathrm{A}2]$-$[\mathrm{A}4]$ and $[\mathrm{N}]$ are satisfied, then $[\mathrm{SC}]$ holds for $w$ given by $(\ref{avar})$.
\item If $[\mathrm{A}1']$, $[\mathrm{A}2]$-$[\mathrm{A}4]$ and $[\mathrm{N}]$ are satisfied, then $[\mathrm{SC}]$ holds for $w$ given by $(\ref{avarend})$.
\end{enumerate}
\end{thm}

Theorem \ref{HYthm6.1} and \ref{HYthm6.2} are proved in Section \ref{drift}.

\begin{proof}[\upshape{\bfseries{Proof of Theorem \ref{mainthm}}}]
The desired result follows from Lemma \ref{lembasic} and Theorem \ref{HYthm6.1}(b) and (c).
\end{proof}


\section{Some related topics for statistical application}\label{topics}

\subsection{Consistency}\label{consistency}

In order to obtain our main theorem, we need to impose a kind of predictability such as [A2] on the sampling scheme. In fact, it is still an active research area to seek asymptotic theories of estimators for volatility-type quantities when sampling scheme is random and depends on observed processes even if neither nonsynchronicity nor microstructure noise is present; See \cite{Fu2010b} and \cite{LMRZZ2012} for instance. Such a situation, however, dramatically changes when we restrict our attention to the consistency of the estimators. As is well known, the classical realized covariance is a consistent estimator for the integrated covariance whenever the sampling scheme consists of stopping times and the mesh size of sampling times tends to 0. Furthermore, \citet{HK2008} verified such a result for the Hayashi-Yoshida estimator in the presence of the nonsynchronicity of the sampling scheme. The following theorem, which is a by-product of Lemma \ref{lembasic}, tells us such a result is still valid for our estimator:
\begin{thm}\label{HKthm2.3}
Suppose $(\ref{A4})$ and $[\mathrm{C}1]$-$[\mathrm{C}3]$ are satisfied. Then $\widehat{PHY}(\mathsf{X},\mathsf{Y})^n\xrightarrow{ucp}[X,Y]$ as $n\to\infty$, provided that $\xi'>1/2$.
\end{thm}

The proof is given in Section \ref{proofHKthm2.3}.

\subsection{Poisson sampling with a random change point}

As an illustrative example of sampling scheme satisfying the conditions [A1$'$], [A2], [A4] and [A6], we shall discuss a Poisson sampling with a random change point, which was also discussed in \cite{HY2011}.

First we construct the stochastic basis $\mathcal{B}^{(0)}$ which is appropriate for the present situation. Let $(\Omega',\mathcal{F}',(\mathcal{F}'_t)$, $P')$ be a stochastic basis, and suppose that the semimartingales $X$, $Y$, $Z^X$ and $Z^Y$ are defined on this basis. Suppose also that $\Psi$ is $(\mathcal{F}'_t)$-adapted. Furthermore, on an auxiliary probability space $(\Omega'',\mathcal{F}'',P'')$, there are mutually independent standard Poisson processes $(\underline{N}^k_t)$, $(\overline{N}^k_t)$ $(k=1,2)$. Then we construct $\mathcal{B}^{(0)}=(\Omega^{(0)},\mathcal{F}^{(0)},\mathbf{F}^{(0)}=(\mathcal{F}^{(0)}_t)_{t\in\mathbb{R}_+} ,P^{(0)})$ by
\begin{align*}
\Omega^{(0)}=\Omega'\times\Omega'',\qquad
\mathcal{F}^{(0)}=\mathcal{F}'\otimes\mathcal{F}'',\qquad
\mathcal{F}^{(0)}_t=\mathcal{F}'_t\otimes\mathcal{F}'',\qquad
P^{(0)}=P'\times P''.
\end{align*}

Next we construct our sampling schemes. For each $k=1,2$, let $\underline{p}^k,\overline{p}^k\in(0,\infty)$ and let $\tau^k$ be an $(\mathcal{F}'_t)$-stopping time. Define $(\underline{S}^i)$ and $(\overline{S}^i)$ each as the arrival times of the point processes $\underline{N}^{n,1}=(\underline{N}^1_{n\underline{p}^1t})$ and $\overline{N}^{n,1}=(\overline{N}^1_{n\overline{p}^1t})$ respectively. Let $\eta\in(0,\frac{1}{2})$ and set $\tau^1_n=\tau^1+n^{-\eta}$. Then, we define $(S^i)$ sequentially by $S^0=0$ and
\begin{align*}
S^i=\inf_{l,m\in\mathbb{N}}\left\{\underline{S}^l_{\{S^{i-1}<\underline{S}^l<\tau^1_n\}},(\tau^1_n+\overline{S}^m)_{\{S^{i-1}<\tau_n^1+\overline{S}^m\}}\right\},\qquad i=1,2,\dots.
\end{align*}
Here, for a stopping time $T$ with respect to filtration $(\mathcal{F}_t)$ and a set $A\in\mathcal{F}_T$, we define $T_A$ by $T_A(\omega)=T(\omega)$ if $\omega\in A$; $T_A(\omega)=\infty$ otherwise. $(T^j)$ is defined in the same way using $\underline{N}^{n,2}=(\underline{N}^2_{n\underline{p}^2t})$, $\overline{N}^{n,2}=(\overline{N}^2_{n\overline{p}^2t})$ and $\tau^2_n=\tau^2+n^{-\eta}$ instead of $\underline{N}^{n,1}$, $\overline{N}^{n,1}$ and $\tau^1_n$ respectively.

Now we verify the conditions [A1$'$], [A2], [A4] and [A6]. First, [A6] is obviously satisfied. Next, [A2] can be verified with $\xi=\eta+1/2$ in a similar manner to the proof of Lemma 8.1 of \cite{HY2011}. Moreover, since $r_n(t)=O_p(\log n/n)$ as $n\to\infty$ for any $t>0$ by Corollary 1 in \cite{RT1973}, $(\ref{A4})$ holds for any $\xi'\in(0,1)$, hence [A4] holds true. Finally, let $\mathbf{H}^n$ be the filtration generated by the $\sigma$-field $\mathcal{F}'$ and the processes $N^{n,1}$ and $N^{n,2}$. Then [A1$'$] is verified by the following proposition. 
\begin{prop}
We have $[\mathrm{A}1']$ with $\chi\equiv0$ and
\begin{align}
&G_s=\left(\frac{1}{\underline{p}^1}+\frac{1}{\underline{p}^2}-\frac{1}{\underline{p}^1+\underline{p}^2}\right)1_{\{s<\tau^1\wedge\tau^2\}}
+\left(\frac{1}{\overline{p}^1}+\frac{1}{\underline{p}^2}-\frac{1}{\overline{p}^1+\underline{p}^2}\right)1_{\{\tau^1\leq s<\tau^2\}}+\nonumber\\
&\hphantom{G_s=(}\left(\frac{1}{\underline{p}^1}+\frac{1}{\overline{p}^2}-\frac{1}{\underline{p}^1+\overline{p}^2}\right)1_{\{\tau^2\leq s<\tau^1\}}
+\left(\frac{1}{\overline{p}^1}+\frac{1}{\overline{p}^2}-\frac{1}{\overline{p}^1+\overline{p}^2}\right)1_{\{\tau^1\vee\tau^2\leq s\}},\label{poissonG}\\
&\left. \begin{array}{l}\displaystyle
F^1_s=\frac{1}{\underline{p}^1}1_{\{s<\tau^1\}}+\frac{1}{\overline{p}^1}1_{\{\tau^1\leq s\}},\qquad\qquad
F^2_s=\frac{1}{\underline{p}^2}1_{\{s<\tau^2\}}+\frac{1}{\overline{p}^2}1_{\{\tau^2\leq s\}},\\ 
\displaystyle F^{1*2}_s=\frac{2}{\underline{p}^1+\underline{p}^2}1_{\{s<\tau^1\wedge\tau^2\}}
+\frac{2}{\overline{p}^1+\underline{p}^2}1_{\{\tau^1\leq s<\tau^2\}}
+\frac{2}{\underline{p}^1+\overline{p}^2}1_{\{\tau^2\leq s<\tau^1\}}
+\frac{2}{\overline{p}^1+\overline{p}^2}1_{\{\tau^1\vee\tau^2\leq s\}}.
\end{array}\right\}\label{poissonF}
\end{align}
\end{prop}

\begin{proof}
First, it is evident that [A1$'$](ii)-(iii) and (v) with $\chi\equiv0$ hold true.

Next, let $\lambda^1,\lambda^2\in(0,\infty)$, and consider a Poisson process $(\widetilde{N}_t)$ with intensity $\lambda^1+\lambda^2$. Moreover, let $(\eta_k)_{k\in\mathbb{N}}$ be an i.i.d. random variables independent of $\widetilde{N}$ with $P(\eta_1=1)=1-P(\eta_1=0)=\lambda^1/(\lambda^1+\lambda^2)$, and set $\widetilde{N}^1_t=\sum_{k=1}^{\widetilde{N}^n_t}\eta_k$ and $\widetilde{N}^2_t=\widetilde{N}_t-\widetilde{N}^1_t$. Then, a short calculation shows that for each $k=1,2$ $\widetilde{N}^k$ is a Poisson process with intensity $\lambda^k$. Furthermore, Theorem 2 in \cite{CA1968} implies that $\widetilde{N}^1$ and $\widetilde{N}^2$ are independent. Using this fact, we can show that $G(1)^n_{R^{k-1}}=G^n_{R^{k-1}}$ for any $k\in\mathbb{N}-\mathcal{N}^0_n$. Here, $G^n$ denotes the process defined by the right hand of $(\ref{poissonG})$ with $\tau^1$ and $\tau^2$ replaced by $\tau^1_n$ and $\tau^2_n$, and $\mathcal{N}^0_n=\{k|R^{k-1}\leq\tau^1_n<R^k\}\cup\{k|R^{k-1}\leq\tau^2_n<R^k\}$. Since $\#\mathcal{N}^0_n\leq2$ and $G^n\xrightarrow{\text{Sk.p.}}G$ as $n\to\infty$, we conclude that [A1$'$](i) with $(\ref{poissonG})$ holds true.

Finally, if $R^k=\widehat{S}^k$, we have $\check{S}^{k+1}=R^k$ and $\widehat{T}^k\leq\check{T}^{k+1}<R^k$, hence we have $\check{I}^{k+1}\cap\check{J}^k=\emptyset$ and $\check{S}^k\vee\check{T}^{k+1}<\widehat{S}^k\wedge\widehat{T}^{k+1}$. Similarly we have $\check{I}^{k}\cap\check{J}^{k+1}=\emptyset$ and $\check{S}^{k+1}\vee\check{T}^{k}<\widehat{S}^{k+1}\wedge\widehat{T}^{k}$ if $R^k=\widehat{T}^k$. Combining these facts with an argument similar to the above, we can also show that [A1$'$](iv) with $(\ref{poissonF})$ holds true, hence we complete the proof.
\end{proof}

\subsection{A round-off error model}

In this subsection we illustrate an example of the microstructure noise model involving rounding effects. It is a version of Example 2 in \cite{JLMPV2009}. The round off error is known as one of the sources of the Epps effect; see \cite{MSG2010}.

In the remainder of this subsection, we assume that $Z^X=Z^Y=0$. Suppose that the observation data is given as follows:
\begin{equation}\label{round}
\mathsf{X}_{S^i}=\gamma^X\lceil (X_{S^i}+u^X_i)/\gamma^X\rfloor,\qquad
\mathsf{Y}_{T^j}=\gamma^Y\lceil (Y_{T^j}+u^Y_j)/\gamma^Y\rfloor.
\end{equation}
Here, $\gamma^X,\gamma^Y>0$, $(u^X_i)$ and $(u^Y_j)$ are mutually independent i.i.d. sequences of random variables independent of $X$ and $Y$, and for a real number $x$ we denote by $\lceil x\rfloor$ the unique integer $a$ such that $a-1/2\leq x<a+1/2$. Suppose that $u^X_i$ and $u^Y_j$ each are uniform over $[-\gamma^X/2,\gamma^X/2]$ and $[-\gamma^Y/2,\gamma^Y/2]$ respectively. Then, this model can be accommodated to our framework in the following way: For a real number $x$, let $\mu^X_x$ be the Bernoulli distribution taking values $\gamma^X(\delta^X_x+\mathrm{sign}(-\delta^X_x))$ and $\gamma^X\delta^X_x$ with probabilities $|\delta^X_x|$ and $1-|\delta^X_x|$, where $\delta^X_x=\lceil x/\gamma^X\rfloor-x/\gamma^X$ and $\mathrm{sign}(a)$ is equal to 1 if $a\geq 0$ and $-1$ otherwise. Similarly we define $\mu^Y_x$ with replacing $X$ by $Y$. After that, we define $Q_t(\omega^{(0)},\mathrm{d}x\mathrm{d}y)=\mu^X_{X_t(\omega^{(0)})}(\mathrm{d}x)\mu^Y_{Y_t(\omega^{(0)})}(\mathrm{d}y)$. We can easily confirm $(\ref{centered})$ and $(\ref{round})$. Moreover, since the function $x\mapsto |\lceil x\rfloor-x|$ is Lipschitz, the condition [N] holds if we have [C1].

\section{Proof of Lemma \ref{lembasic}}\label{prooflembasic}


Throughout this section, we fix a constant $\gamma$ such that $\gamma<\xi'-1/2$. First note that for the proof we can use a localization procedure, and which allows us to systematically replace the conditions [C1]-[C3] by the following strengthened versions:
\begin{enumerate}
\item[{[SC1]}] [C1] holds, and $(A^X)'$, $(A^Y)'$, $(\underline{A}^X)'$, $(\underline{A}^Y)'$ and $[V,W]'$ for each $V,W=X,Y,Z^X,Z^Y$ are bounded.
\item[{[SC2]}] $(\int |z|^2Q_t(\mathrm{d}z))_{t\in\mathbb{R}_+}$ is a bounded process.
\item[{[SC3]}] There is a positive constant $K$ such that $b_nN_n(t)\leq K$ for all $n$ and $t$.
\end{enumerate}

We write $\bar{r}_n=b_n^{\xi'}$. Next, let $\upsilon_n=\inf\{t|r_n(t)>\bar{r}_n\}$, and define a sequence $(\widetilde{S}^i)_{i\in\mathbb{Z}^+}$ sequentially by
\begin{equation*}
\widetilde{S}^i=
\left\{\begin{array}{ll}
S^i & \textrm{if $S^i<\upsilon_n$},\\
\widetilde{S}^{i-1}+\bar{r}_n & \textrm{otherwise}.
\end{array}\right.
\end{equation*}
Then, $(\widetilde{S}^i)$ is obviously a sequence of $\mathbf{F}^{(0)}$-stopping times satisfying $(\ref{increace})$ and $\sup_{i\in\mathbb{N}}(\widetilde{S}^i-\widetilde{S}^{i-1})\leq\bar{r}_n$. Furthermore, for any $t>0$ we have $P(\bigcap_i\{\widetilde{S}^i\wedge t\neq S^i\wedge t\})\leq P(\upsilon_n<t)\to0$ as $n\to\infty$ by $(\ref{A4})$. By replacing $(S^i)$ with $(T^j)$, we can construct a sequence $(\widetilde{T}^j)$ in a similar manner. This argument implies that we may also assume that
\begin{equation}\label{SA4}
\sup_{t\in\mathbb{R}_+}r_n(t)\leq\bar{r}_n
\end{equation}
by an appropriate localization procedure.

Set $\Delta(g)^n_p=g^n_{p+1}-g^n_p$ for every $n,p$. For a process $V=(V_t)_{t\in\mathbb{R}_+}$, let
\begin{align*}
\widetilde{V}_g(\widehat{\mathcal{I}})^i_t=\sum_{p=0}^{k_n-1}k_n\Delta(g)^n_p V(\widehat{I}^{i+p})_t,\qquad
\widetilde{V}_g(\widehat{\mathcal{J}})^j_t=\sum_{q=0}^{k_n-1}k_n\Delta(g)^n_q V(\widehat{J}^{j+q})_t
\end{align*}
for each $t\in\mathbb{R}_+$ and $i,j\in\mathbb{N}$.


\begin{lem}\label{modulus}
Suppose $A^X$, $A^Y$, $[X]$, $[Y]$, $\underline{A}^X$, $\underline{A}^Y$, $[Z^X]$ and $[Z^Y]$ are absolutely continuous with locally bounded derivatives. Suppose also $(\ref{SA4})$ holds. Then a.s. we have
\begin{align}
\limsup_{n\rightarrow\infty}\sup_{i\in\mathbb{N}}
\frac{|\bar{X}_g(\widehat{\mathcal{I}})^i_t|}{\sqrt{2 k_n\bar{r}_n\log\frac{1}{\bar{r}_n}}}\leq \| g\|_\infty\sup_{0\leq s\leq t}|[X]'_s|,\qquad
\limsup_{n\rightarrow\infty}\sup_{i\in\mathbb{N}}
\frac{|\bar{\mathfrak{Z}}^X_{g'}(\widehat{\mathcal{I}})^i_t|+|\widetilde{\mathfrak{Z}}^X_g(\widehat{\mathcal{I}})^i_t|}{\sqrt{2 k_n\bar{r}_n\log\frac{1}{\bar{r}_n}}}\leq L\sup_{0\leq s\leq t}|[Z^X]'_s|,\label{modulus1}\\
\limsup_{n\rightarrow\infty}\sup_{j\in\mathbb{N}}
\frac{|\bar{Y}_g(\widehat{\mathcal{J}})^j_t|}{\sqrt{2 k_n\bar{r}_n\log\frac{1}{\bar{r}_n}}}\leq \| g\|_\infty\sup_{0\leq s\leq t}|[Y]'_s|,\qquad
\limsup_{n\rightarrow\infty}\sup_{j\in\mathbb{N}}
\frac{|\bar{\mathfrak{Z}}^Y_{g'}(\widehat{\mathcal{J}})^j_t|+|\widetilde{\mathfrak{Z}}^Y_g(\widehat{\mathcal{J}})^j_t|}{\sqrt{2 k_n\bar{r}_n\log\frac{1}{\bar{r}_n}}}\leq L\sup_{0\leq s\leq t}|[Z^Y]'_s|\label{modulus2}
\end{align}
for any $t>0$, where $L$ is a positive constant which only depends on $g$.
\end{lem}

\begin{proof}
Combining a representation of a continuous local martingale with Brownian motion and L\'{e}vy's theorem on the uniform modulus of continuity of Brownian motion, we obtain
\begin{align*}
\limsup_{\delta\rightarrow +0}\sup_{
\begin{subarray}{c}
s,u\in[0,t]\\
|s-u|\leq\delta
\end{subarray}}
\frac{|X_s-X_u|}{\sqrt{2\delta\log\frac{1}{\delta}}}\leq \sup_{0\leq s\leq t}|[X]'_s|,\qquad
\limsup_{\delta\rightarrow +0}\sup_{
\begin{subarray}{c}
s,u\in[0,t]\\
|s-u|\leq\delta
\end{subarray}}
\frac{|\mathfrak{Z}^X_s-\mathfrak{Z}^X_u|}{\sqrt{2\delta\log\frac{1}{\delta}}}\leq \sup_{0\leq s\leq t}|[Z^X]'_s|.
\end{align*}
Since $\bar{X}_g(\widehat{\mathcal{I}})^i_t=-\sum_{p=0}^{k_n-1}\Delta(g)^n_p(X_{\widehat{S}^{i+p}\wedge t}-X_{\widehat{S}^{i}\wedge t})$ and $|\Delta(g)^n_p|\leq\frac{1}{k_n}\| g\|_\infty$, we obtain the first inequality in $(\ref{modulus1})$. On the other hand, since Abel's partial summation formula yields $\bar{\mathfrak{Z}}^X_{g'}(\widehat{\mathcal{I}})^i_t=\sum_{p=0}^{k_n-1}k_n\{(g')^n_{p}-(g')^n_{p+1}\}(\mathfrak{Z}^X_{\widehat{S}^{i+p}\wedge t}-\mathfrak{Z}^X_{\widehat{S}^{i-1}\wedge t})$ and  $\widetilde{\mathfrak{Z}}^X_g(\widehat{\mathcal{I}})^i_t=\sum_{p=0}^{k_n-1}k_n\{\Delta(g)^n_{p}-\Delta(g)^n_{p+1}\}(\mathfrak{Z}^X_{\widehat{S}^{i+p}\wedge t}-\mathfrak{Z}^X_{\widehat{S}^{i-1}\wedge t})$, and $\Delta(g)^n_{p+1}-\Delta(g)^n_p=-\int_{p/k_n}^{(p+1)/k_n}\{g'(x+1/k_n)-g'(x)\}\mathrm{d}x$, the piecewise Lipschitz continuity of $g'$ implies the second inequality in $(\ref{modulus1})$.

By symmetry we also obtain $(\ref{modulus2})$.
\end{proof}

We can strengthen Lemma \ref{modulus} by a localization if we assume that $(\ref{SA4})$ and [SC1] hold, so that in the remainder of this section we always assume that we have a positive constant $K$ and a positive integer $n_0$ such that
\begin{align}
\sup_{i\in\mathbb{N}}
\frac{|\bar{X}_g(\widehat{\mathcal{I}})^i_t(\omega)|+|\bar{\mathfrak{Z}}^X_{g'}(\widehat{\mathcal{I}})^i_t(\omega)|+|\widetilde{\mathfrak{Z}}^X_g(\widehat{\mathcal{I}})^i_t(\omega)|}{\sqrt{2 k_n\bar{r}_n|\log b_n|}}
+\sup_{j\in\mathbb{N}}
\frac{|\bar{Y}_g(\widehat{\mathcal{J}})^j_t(\omega)|+|\bar{\mathfrak{Z}}^Y_{g'}(\widehat{\mathcal{J}})^j_t(\omega)|+|\widetilde{\mathfrak{Z}}^Y_g(\widehat{\mathcal{J}})^j_t(\omega)|}{\sqrt{2 k_n \bar{r}_n|\log b_n|}}
\leq K\label{absmod2}
\end{align}
for all $t>0$ and $\omega\in\Omega$ if $n\geq n_0$. Moreover, we only consider sufficiently large $n$ such that $n\geq n_0$.


Let
\begin{align*}
\mathbf{I}_t:=&\frac{1}{(\psi_{HY}k_n)^2}\sum_{i,j=1}^{\infty}\bar{X}_g(\widehat{\mathcal{I}})^i_t\bar{Y}_g(\widehat{\mathcal{J}})^j_t\bar{K}^{ij}_t,&
\mathbf{II}_t:=&\frac{1}{(\psi_{HY}k_n)^2}\sum_{i,j=1}^{\infty}\widetilde{\mathfrak{U}}^X_g(\widehat{\mathcal{I}})^i_t\widetilde{\mathfrak{U}}^Y_g(\widehat{\mathcal{J}})^j_t\bar{K}^{ij}_t,\\
\mathbf{III}_t:=&\frac{1}{(\psi_{HY}k_n)^2}\sum_{i,j=1}^{\infty}\bar{X}_g(\widehat{\mathcal{I}})^i_t\widetilde{\mathfrak{U}}^Y_g(\widehat{\mathcal{J}})^j_t\bar{K}^{ij}_t,&
\mathbf{IV}_t:=&\frac{1}{(\psi_{HY}k_n)^2}\sum_{i,j=1}^{\infty}\widetilde{\mathfrak{U}}^X_g(\widehat{\mathcal{I}})^i_t\bar{Y}_g(\widehat{\mathcal{J}})^j_t\bar{K}^{ij}_t.
\end{align*}

The following lemma tells us that the edge effects are negligible. Throughout the discussions, for (random) sequences $(x_n)$ and $(y_n)$, $x_n\lesssim y_n$ means that there exists a (non-random) constant $C\in[0,\infty)$ such that $x_n\leq Cy_n$ for large $n$. We denote by $E_0$ a conditional expectation given $\mathcal{F}^{(0)}$, i.e. $E_0[\cdot]:=E[\cdot|\mathcal{F}^{(0)}]$.
\begin{lem}\label{basicrep}
Suppose that $[\mathrm{SC}1]$-$[\mathrm{SC}3]$ and $(\ref{SA4})$ are satisfied. Then we have $b_n^{-\gamma}\{\widehat{PHY}(\mathsf{X},\mathsf{Y})^n-(\mathbf{I}+\mathbf{II}+\mathbf{III}+\mathbf{IV})\}\xrightarrow{ucp}0$ as $n\to\infty$.
\end{lem}

\begin{proof}
We can rewrite $\widehat{PHY}(\mathsf{X},\mathsf{Y})^n_t$ and $\mathbf{I}_t+\mathbf{II}_t+\mathbf{III}_t+\mathbf{IV}_t$ as
\begin{align*}
&\widehat{PHY}(\mathsf{X},\mathsf{Y})^n_t\\
=&\frac{1}{(\psi_{HY}k_n)^2}\Biggl[\sum_{i,j=1}^{\infty}\bar{\mathsf{X}}_g(\widehat{\mathcal{I}})^i_t\bar{\mathsf{Y}}_g(\widehat{\mathcal{J}})^j_t\bar{K}^{ij}_t 1_{\{\widehat{S}^{i+k_n}\vee\widehat{T}^{j+k_n}\leq t\}}
+\sum_{i,j:i=0\textrm{ or }j=0}\bar{\mathsf{X}}(\widehat{\mathcal{I}})^i\bar{\mathsf{Y}}(\widehat{\mathcal{J}})^j\bar{K}^{ij} 1_{\{\widehat{S}^{i+k_n}\vee\widehat{T}^{j+k_n}\leq t\}}\Biggr]
\end{align*}
and
\begin{align*}
\mathbf{I}_t+\mathbf{II}_t+\mathbf{III}_t+\mathbf{IV}_t
=\frac{1}{(\psi_{HY}k_n)^2}\sum_{i,j=1}^{\infty}\bar{\mathsf{X}}_g(\widehat{\mathcal{I}})^i_t\bar{\mathsf{Y}}_g(\widehat{\mathcal{J}})^j_t\bar{K}^{ij}_t 1_{\{\widehat{S}^{i}\vee\widehat{T}^{j}\leq t\}},
\end{align*}
where $\bar{\mathsf{X}}_g(\widehat{\mathcal{I}})^i_t=\bar{X}_g(\widehat{\mathcal{I}})^i_t+\widetilde{\mathfrak{U}}^X_g(\widehat{\mathcal{I}})^i_t$ and $\bar{\mathsf{Y}}_g(\widehat{\mathcal{J}})^j_t=\bar{Y}_g(\widehat{\mathcal{J}})^j_t+\widetilde{\mathfrak{U}}^Y_g(\widehat{\mathcal{J}})^j_t$. hence we can decompose the target quantity as
\begin{align*}
&(\mathbf{I}_t+\mathbf{II}_t+\mathbf{III}_t+\mathbf{IV}_t)-\widehat{PHY}(\mathsf{X},\mathsf{Y})^n_t\\
=&\frac{1}{(\psi_{HY}k_n)^2}\Biggl[\sum_{i,j=1}^{\infty}\bar{\mathsf{X}}_g(\widehat{\mathcal{I}})^i_t\bar{\mathsf{Y}}_g(\widehat{\mathcal{J}})^j_t\bar{K}^{ij}_t 1_{\{\widehat{S}^{i}\vee\widehat{T}^{j}\leq t<\widehat{S}^{i+k_n}\vee\widehat{T}^{j+k_n}\}}-\sum_{i=1}^{\infty}\bar{\mathsf{X}}_g(\widehat{\mathcal{I}})^i_t\bar{\mathsf{Y}}(\widehat{\mathcal{J}})^0\bar{K}^{i0}_t 1_{\{\widehat{S}^{i+k_n}\vee\widehat{T}^{k_n}\leq t\}}\\
&\hphantom{\frac{1}{(\psi_{HY}k_n)^2}\Biggl[}-\sum_{j=1}^{\infty}\bar{\mathsf{X}}(\widehat{\mathcal{I}})^0\bar{\mathsf{Y}}_g(\widehat{\mathcal{J}})^j_t\bar{K}^{0j}_t 1_{\{\widehat{S}^{k_n}\vee\widehat{T}^{j+k_n}\leq t\}}-\bar{\mathsf{X}}(\widehat{\mathcal{I}})^0\bar{\mathsf{Y}}(\widehat{\mathcal{J}})^0 1_{\{\widehat{S}^{k_n}\vee\widehat{T}^{k_n}\leq t\}}\Biggr]\\
=:&\mathbf{A}_{1,t}+\mathbf{A}_{2,t}+\mathbf{A}_{3,t}++\mathbf{A}_{4,t}.
\end{align*}

Consider $\mathbf{A}_{1,t}$. The Schwarz inequality yields
\begin{align*}
&E_0\left[\sup_{0\leq s\leq t}|\mathbf{A}_{1,s}|\right]\\
\leq&\frac{1}{(\psi_{HY}k_n)^2}\sup_{0\leq s\leq t}\sum_{i,j=1}^{\infty}\left\{E_0\left[\sup_{0\leq u\leq t}\left|\bar{\mathsf{X}}_g(\widehat{\mathcal{I}})^i_u\right|^2\right] E_0\left[\sup_{0\leq u\leq t}\left|\bar{\mathsf{Y}}_g(\widehat{\mathcal{J}})^j_u\right|^2\right]\right\}^{1/2}\bar{K}^{ij}_s 1_{\{\widehat{S}^{i}\vee\widehat{T}^{j}\leq s<\widehat{S}^{i+k_n}\vee\widehat{T}^{j+k_n}\}}.
\end{align*}
Since both $\mathfrak{E}^X$ and $\mathfrak{E}^Y$ are martingales on $(\Omega^{(1)},\mathcal{F}^{(1)},\mathbf{F}^{(1)},Q(\omega^{(0)},\mathrm{d}z))$ for each $\omega^{(0)}$, we have 
\begin{equation}\label{noiseest2}
E_0\left[\sup_{0\leq u\leq t}\left|\widetilde{\mathfrak{E}}^X_g(\widehat{\mathcal{I}})^i_u\right|^2\right]\lesssim k_n^{-1},\qquad
E_0\left[\sup_{0\leq u\leq t}\left|\widetilde{\mathfrak{E}}^Y_g(\widehat{\mathcal{J}})^j_u\right|^2\right]\lesssim k_n^{-1}
\end{equation}
by the Doob inequality and [SC2]. Combining this with $(\ref{absmod2})$, we obtain
\begin{align*}
E_0\left[\sup_{0\leq s\leq t}|\mathbf{A}_{1,s}|\right]
\lesssim\frac{\bar{r}_n|\log b_n|}{k_n}\sup_{0\leq s\leq t}\sum_{i,j=1}^{\infty}\bar{K}^{ij}_s 1_{\{\widehat{S}^{i}\vee\widehat{T}^{j}\leq s<\widehat{S}^{i+k_n}\vee\widehat{T}^{j+k_n}\}}
\lesssim k_n\bar{r}_n|\log b_n|,
\end{align*}
and thus we conclude that $b_n^{-\gamma}\sup_{0\leq s\leq t}|\mathbf{A}_{1,s}|\to^p0$ as $n\to\infty$. 

Similarly we can show that $b_n^{-\gamma}\sup_{0\leq s\leq t}|\mathbf{A}_{l,s}|\to^p0$ as $n\to\infty$ for $l=2,3,4$. After all, we complete the proof of lemma.
\end{proof}


We define the processes $\mathfrak{M}^X$, $\mathfrak{M}^Y$, $\mathfrak{A}^X$ and $\mathfrak{A}^Y$ as follows:
\begin{align*}
\mathfrak{M}^X=-\mathfrak{I}_-\bullet\underline{M}^X,\qquad
\mathfrak{M}^Y=-\mathfrak{J}_-\bullet\underline{M}^Y,\qquad
\mathfrak{A}^X=-\mathfrak{I}_-\bullet\underline{A}^X,\qquad
\mathfrak{A}^Y=-\mathfrak{J}_-\bullet\underline{A}^Y.
\end{align*}

Let $\widehat{K}^{pq}_t=1_{\{\widehat{I}^p(t)\cap\widehat{J}^q(t)\neq\emptyset\}}$ for each $p,q,t$. Then Proposition \ref{advantage} yields
\begin{equation}\label{sumhatK}
\sum_{q=1}^\infty\widehat{K}^{pq}_t\leq 3,\qquad\sum_{p=1}^\infty\widehat{K}^{pq}_t\leq 3
\end{equation}
for every $p,q,t$. Furthermore, we have the following result:
\begin{lem}\label{lemoverlap}
Suppose that $[\mathrm{SC}1]$-$[\mathrm{SC}3]$ and $(\ref{SA4})$ are satisfied. Then we have
\begin{equation}\label{aimoverlap}
\sup_{0\leq s\leq t}\left|\sum_{i,j=1}^{\infty}\bar{K}^{i j}_s\sum_{p,q=0}^{k_n-1}\alpha^n_p\beta^n_q L(V,W)^{i+p,j+q}_s\widehat{K}^{i+p,j+q}_s\right|=o_p(k_n^2\cdot b_n^\gamma)
\end{equation}
as $n\to\infty$ for any $\alpha,\beta\in\Phi$, $V\in\{M^X,\mathfrak{E}^X,\mathfrak{M}^X,A^X,\mathfrak{A}^X\}$, $W\in\{M^Y,\mathfrak{E}^Y,\mathfrak{M}^Y,A^Y,\mathfrak{A}^Y\}$ and $t>0$, where $L(V,W)^{pq}_t=V(\widehat{I}^p)_-\bullet W(\widehat{J}^q)_t+W(\widehat{J}^q)_-\bullet V(\widehat{I}^p)_t$ for each $p,q,t$.
\end{lem}

\begin{proof}

First, by a localization procedure based on a representation of a continuous local martingale with Brownian motion and L\'{e}vy's theorem on the uniform modulus of continuity of Brownian motion, we may assume that there exist positive constants $K$ and $\delta_0$ such that
\begin{equation}\label{absmod1}
\sup_{
\begin{subarray}{c}
s,u\in[0,t]\\
|s-u|\leq\delta
\end{subarray}}
\frac{|M^X_s-M^X_u|+|M^Y_s-M^Y_u|+|L^X_s-L^X_u|+|L^Y_s-L^Y_u|}{\sqrt{2\delta|\log\delta|}}\leq K
\end{equation}
whenever $0<\delta<\delta_0$. Moreover, we only consider sufficiently large $n$ such that $\bar{r}_n<\delta_0$.

First we consider the case that $V\in\{M^X,\mathfrak{E}^X,\mathfrak{M}^X\}$ and  $W\in\{M^Y,\mathfrak{E}^Y,\mathfrak{M}^Y\}$. when $1\leq p\leq k_n-1$ and $1\leq q\leq k_n-1$, we have $\widehat{S}^i\vee\widehat{T}^j\leq\widehat{S}^{i+p-1}\vee\widehat{T}^{j+q-1}<\widehat{S}^{i+p}\wedge\widehat{T}^{j+q}\leq\widehat{S}^{i+k_n}\wedge\widehat{T}^{j+k_n}$ on $\{\widehat{I}^p\cap\widehat{J}^q\neq\emptyset\}$, hence we can decompose the target quantity as 
\begin{align*}
&\sum_{i,j=1}^{\infty}\bar{K}^{i j}_s\sum_{p,q=0}^{k_n-1}\alpha^n_p\beta^n_q L(V,W)^{i+p,j+q}_s\widehat{K}^{i+p,j+q}_s\\
=&\sum_{i,j=1}^{\infty}\left\{\sum_{p,q>0}+\sum_{p=0\textrm{ or }q=0}\bar{K}^{i j}_s\right\}\alpha^n_p\beta^n_q L(V,W)^{i+p,j+q}_s\widehat{K}^{i+p,j+q}_s
=:\mathbb{A}_{1,s}+\mathbb{A}_{2,s}.
\end{align*}

Consider $\mathbb{A}_{1,s}$ first. By an argument similar to the proof of Lemma \ref{HYlem3.1}, we can rewrite it as
\begin{align*}
\mathbb{A}_{1,s}=\sum_{i,j=1}^{\infty}\sum_{p,q=1}^{k_n-1}\alpha^n_p\beta^n_q \widehat{K}^{i+p,j+q}_-\bullet L(V,W)^{i+p,j+q}_s,
\end{align*}  
and thus $\mathbb{A}_{1,\cdot}$ is a locally square-integrable martingale because both $V$ and $W$ are locally square-integrable martingales. Therefore, it is sufficient to prove $\langle\mathbb{A}_{1,\cdot}\rangle_s=o_p(b_n^{2\gamma}\cdot k_n^4)$ as $n\to\infty$ for any $s>0$ due to the Lenglart inequality. Since
\begin{align*}
\mathbb{A}_{1,s}=\sum_{p,q=1}^{\infty}\left(\sum_{i=(p-k_n+1)\vee1}^{p-1}\sum_{j=(q-k_n+1)\vee1}^{q-1}\alpha^n_{p-i}\beta^n_{q-j}\right) \widehat{K}^{p,q}_-\bullet L(V,W)^{p,q}_s,
\end{align*}
we have
\begin{align*}
\langle\mathbb{A}_{1,\cdot}\rangle_s
&\lesssim k_n^4\Biggl[\sum_{p,p',q=1}^{\infty}\widehat{K}^{pq}_-\widehat{K}^{p'q}_-\bullet\left\{ V(\widehat{I}^{p})_-V(\widehat{I}^{p'})_-\bullet\langle W\rangle(\widehat{J}^{q})\right\}_s
+\sum_{p,q,q'=1}^{\infty}\widehat{K}^{pq}_-\widehat{K}^{pq'}_-\bullet\left\{ W(\widehat{J}^{q})_-W(\widehat{J}^{q'})_-\bullet\langle V\rangle(\widehat{I}^{p})\right\}_s\Biggr]\\
&=:\mathbb{I}+\mathbb{II}.
\end{align*}
If $V\neq \mathfrak{E}^X$ or $W\neq\mathfrak{E}^Y$, $(\ref{absmod1})$, [SC2], the Schwarz inequality and $(\ref{sumhatK})$ yield 
\begin{align*}
E_0\left[\mathbb{I}\right]
\lesssim k_n^4\bar{r}_n|\log b_n|E_0\left[\langle W\rangle_s\right].
\end{align*}
On the other hand, since $\mathfrak{E}^X(\widehat{I}^{p})_t\mathfrak{E}^X(\widehat{I}^{p'})_t=k_n^{-2}\epsilon^X_{\widehat{S}^p}\epsilon^X_{\widehat{S}^{p'}}1_{\{\widehat{S}^{p\vee p'}\leq t\}}$ and $[\mathfrak{E}^Y](\widehat{J}^{q})_t=k_n^{-2}(\epsilon^Y_{\widehat{T}^q})^21_{\{\widehat{T}^q\leq t\}}$, we have
\begin{equation}\label{noiseint} 
\mathfrak{E}^X(\widehat{I}^{p})_-\mathfrak{E}^X(\widehat{I}^{p'})_-\bullet[\mathfrak{E}^Y](\widehat{J}^{q})_s
=k_n^{-4}\epsilon^X_{\widehat{S}^p}\epsilon^X_{\widehat{S}^{p'}}(\epsilon^Y_{\widehat{T}^q})^21_{\{\widehat{S}^{p\vee p'}<\widehat{T}^q\leq t\}}.
\end{equation}
Therefore, [SC2]-[SC3], the Schwarz inequality and $(\ref{sumhatK})$ yield
\begin{align*}
E_0\left[\sum_{p,p',q=1}^{\infty}\widehat{K}^{pq}_-\widehat{K}^{p'q}_-\bullet\left\{ \mathfrak{E}^X(\widehat{I}^{p})_-\mathfrak{E}^X(\widehat{I}^{p'})_-\bullet[\mathfrak{E}^X](\widehat{J}^{q})\right\}_s\right]
\lesssim k_n^{-2}.
\end{align*}
Note that Proposition 4.50 in \cite{JS}, [SC1]-[SC3] and the above estimates imply that $\mathbb{I}=o_p(k_n^4\bar{r}_n)$. By symmetry we also obtain $\mathbb{II}=o_p(k_n^4\bar{r}_n)$. Consequently, we conclude that $\langle\mathbb{A}_{1,\cdot}\rangle_s=o_p(b_n^{2\gamma}\cdot k_n^4)$ because $\xi'/2-(\xi'-1/2)=(1-\xi')/2>0$.

Next consider $\mathbb{A}_{2,s}$. Since integration by parts yields $L(V,W)^{i+p,j+q}_s=V(\widehat{I}^{i+p})_sW(\widehat{J}^{j+q})_s-\widehat{I}^{i+p}_-\widehat{J}^{j+q}_-\bullet[V,W]_s,$ $(\ref{absmod1})$, [SC1]-[SC3], the Schwarz inequality and $(\ref{sumhatK})$ imply that
\begin{align*}
E_0\left[\sup_{0\leq s\leq t}|\mathbb{A}_{2,s}|\right]
\lesssim\bar{r}_n|\log b_n|\sum_{i,j=1}^{\infty}\sum_{\begin{subarray}{c}
0\leq p,q\leq k_n-1\\
p=0\textrm{ or }q=0
\end{subarray}}\widehat{K}^{i+p,j+q}_t
\lesssim\bar{r}_n|\log b_n|\cdot k_n b_n^{-1}=k_n^2\cdot b_n^{\xi'-1/2}|\log b_n|. 
\end{align*}
Hence we obtain $\sup_{0\leq s\leq t}|\mathbb{A}_{2,s}|=o_p(k_n^2\cdot b_n^{\gamma})$. Consequently, we conclude that $(\ref{aimoverlap})$.

Next we consider the case that $V\in\{A^X,\mathfrak{A}^X\}$. Since integration by parts yields $$L(V,W)^{i+p,j+q}_s=V(\widehat{I}^{i+p})_sW(\widehat{J}^{j+q})_s,$$ $(\ref{absmod1})$, [SC1]-[SC3], the Schwarz inequality and $(\ref{sumhatK})$ imply that
\begin{align*}
E_0\left[\sup_{0\leq s\leq t}\left|\sum_{i,j=1}^{\infty}\bar{K}^{i j}_s\sum_{p,q=0}^{k_n-1}\alpha^n_p\beta^n_q L(V,W)^{i+p,j+q}_s\widehat{K}^{i+p,j+q}_s\right|\right]
\lesssim &\sqrt{\bar{r}_n|\log b_n|}\sum_{i,j=1}^{\infty}\sum_{p,q=0}^{k_n-1}\widehat{K}^{i+p,j+q}_t|\widehat{I}^{i+p}(t)|\\
\lesssim &b_n^{\xi'/2}\sqrt{|\log b_n|}\cdot k_n^2.
\end{align*}
Since $\xi'/2>\xi'-1/2$, we conclude that $(\ref{aimoverlap})$. By symmetry we also obtain $(\ref{aimoverlap})$ in the case that $W\in\{A^Y,\mathfrak{A}^Y\}$. Consequently, we complete the proof of the lemma.
\end{proof}


\begin{lem}\label{basicess}
Suppose that $(\ref{SA4})$ and $[\mathrm{SC}1]$-$[\mathrm{SC}3]$ are satisfied. Then
\begin{align*}
&\mathrm{(a)}\ b_n^{-\gamma}\{\mathbf{I}-[X,Y]-\mathbf{M}(1)^n\}\xrightarrow{ucp}0,&&
\mathrm{(b)}\ b_n^{-\gamma}\{\mathbf{II}-\mathbf{M}(2)^n\}\xrightarrow{ucp}0,\\
&\mathrm{(c)}\ b_n^{-\gamma}\{\mathbf{III}-\mathbf{M}(3)^n\}\xrightarrow{ucp}0,&&
\mathrm{(d)}\ b_n^{-\gamma}\{\mathbf{IV}-\mathbf{M}(4)^n\}\xrightarrow{ucp}0
\end{align*}
as $n\to\infty$.
\end{lem}

\begin{proof}
(a) By integration by parts we have
\begin{align*}
\mathbf{I}_t-\mathbf{M}(1)^n_t=\frac{1}{(\psi_{HY}k_n)^2}\sum_{i,j=1}^{\infty}\bar{K}^{ij}_t[\bar{X}_g(\widehat{\mathcal{I}})^i,\bar{Y}_g(\widehat{\mathcal{J}})^j]_t.
\end{align*}
Since
\begin{align*}
[\bar{X}_g(\widehat{\mathcal{I}})^i,\bar{Y}_g(\widehat{\mathcal{J}})^j]_t
=\sum_{p,q=0}^{k_n-1}g^n_pg^n_q(\widehat{I}^{i+p}_-\widehat{J}^{j+q}_-)\bullet[X,Y]_t
=\sum_{p=i}^{i+k_n-1}\sum_{q=i}^{j+k_n-1}g^n_{p-i}g^n_{q-j}(\widehat{I}^{p}_-\widehat{J}^{q}_-)\bullet[X,Y]_t,
\end{align*}
we obtain
\begin{align*}
\mathbf{I}_t-\mathbf{M}(1)^n_t=\frac{1}{(\psi_{HY}k_n)^2}
\sum_{p,q=1}^{\infty}\left(\sum_{i=(p-k_n+1)\vee1}^p\sum_{j=(q-k_n+1)\vee1}^q g^n_{p-i}g^n_{q-j}\bar{K}^{ij}_t\right)(\widehat{I}^{p}_-\widehat{J}^{q}_-)\bullet[X,Y]_t.
\end{align*}
On $\{\widehat{I}^p\cap\widehat{J}^q\neq\emptyset\}$ we have $\widehat{S}^{p-1}<\widehat{T}^q$ and $\widehat{T}^{q-1}<\widehat{S}^p$, hence for $i\in\{(p-k_n+1)\vee1,\dots,p-1\}$ and $j\in\{(q-k_n+1)\vee1,\dots,q-1\}$ we have $\widehat{S}^i<\widehat{T}^{j+k_n-1}$ and $\widehat{T}^j<\widehat{S}^{i+k_n-1}$, so that $\bar{K}^{ij}=1$. Therefore, for $p,q\geq k_n$ we have
\begin{align*}
\sum_{i=(p-k_n+1)\vee1}^p\sum_{j=(q-k_n+1)\vee1}^q g^n_{p-i}g^n_{q-j}\bar{K}^{ij}_t
=\left(\sum_{i=1}^{k_n-1}g_i^n\right)^2
\end{align*}
on $\{\widehat{I}^p\cap\widehat{J}^q\neq\emptyset\}$ because $g(0)=0$. Since $(\widehat{I}^{p}_-\widehat{J}^{q}_-)\bullet[X,Y]_t=1_{\{\widehat{I}^p\cap\widehat{J}^q\neq\emptyset\}}(\widehat{I}^{p}_-\widehat{J}^{q}_-)\bullet[X,Y]_t$, we obtain
\begin{align*}
&\mathbf{I}_t-\mathbf{M}(1)^n_t\\
=&\frac{1}{(\psi_{HY}k_n)^2}\Biggl[
\left(\sum_{i=1}^{k_n-1}g_i^n\right)^2\sum_{p,q=k_n}^{\infty}1_{\{\widehat{I}^p\cap\widehat{J}^q\neq\emptyset\}}(\widehat{I}^{p}_-\widehat{J}^{q}_-)\bullet[X,Y]_t\\
&\hphantom{\frac{1}{(\psi_{HY}k_n)^2}\Biggl[}+\sum_{p,q:p\wedge q<k_n}\left(\sum_{i=(p-k_n+1)\vee1}^p\sum_{j=(q-k_n+1)\vee1}^q g^n_{p-i}g^n_{q-j}\bar{K}^{ij}_t\right)1_{\{\widehat{I}^p\cap\widehat{J}^q\neq\emptyset\}}(\widehat{I}^{p}_-\widehat{J}^{q}_-)\bullet[X,Y]_t\Biggr]\\
=:&\mathbf{B}_{1,t}+\mathbf{B}_{2,t}.
\end{align*}
Since $\frac{1}{(\psi_{HY}k_n)^2}\left(\sum_{i=1}^{k_n-1}g_i^n\right)^2=1+O(k_n^{-1})$ by the Lipschitz continuity of $g$ and 
\begin{align*}
\sup_{0\leq s\leq t}\left|\sum_{p,q=k_n}^{\infty}1_{\{\widehat{I}^p\cap\widehat{J}^q\neq\emptyset\}}(\widehat{I}^{p}_-\widehat{J}^{q}_-)\bullet[X,Y]_s-[X,Y]_s\right|=O_p\left(k_n\bar{r}_n\right)=o_p\left(b_n^{1/4}\right)
\end{align*}
by [SC1] and $(\ref{SA4})$, we have $\sup_{0\leq s\leq t}|\mathbf{B}_{1,s}-[X,Y]_s|=o_p(b_n^{\gamma})$. Moreover, [SC1] and $(\ref{SA4})$ also yield $\sup_{0\leq s\leq t}|\mathbf{B}_{2,s}|=o_p(b_n^{\gamma})$. Consequently, we complete the proof of (a).


(b) Since
\begin{align*}
\mathbf{II}_t
=&\frac{1}{(\psi_{HY}k_n)^2}\sum_{i,j=1}^{\infty}\sum_{p,q=0}^{k_n-1}k_n\Delta(g)_p^n k_n\Delta(g)_q^n \mathfrak{U}^{X}(\widehat{I}^p)_t\mathfrak{U}^{Y}(\widehat{J}^q)_t\bar{K}^{i j}_t\\
=&\frac{1}{(\psi_{HY}k_n)^2}\sum_{p,q=1}^{\infty}\sum_{i=(p-k_n+1)\vee1}^{p}\sum_{j=(q-k_n+1)\vee 1}^{q}k_n\Delta(g)_{p-i}^n k_n\Delta(g)_{q-j}^n \mathfrak{U}^{X}(\widehat{I}^p)_t\mathfrak{U}^{Y}(\widehat{J}^q)_t\bar{K}^{i j}_t,
\end{align*}
we can decompose it as
\begin{align*}
&\mathbf{II}_t\\
=&\frac{1}{(\psi_{HY}k_n)^2}\Biggl\{\left(\sum_{p,q:p\wedge q\geq k_n}+\sum_{p\wedge q<k_n}\right)\widetilde{c}(p,q)_t\widehat{K}^{pq}_t \mathfrak{U}^{X}(\widehat{I}^p)_t\mathfrak{U}^{Y}(\widehat{J}^q)_t
+\sum_{p,q=1}^{\infty}\widetilde{c}(p,q)_t(1-\widehat{K}^{pq}_t)\mathfrak{U}^{X}(\widehat{I}^p)_t\mathfrak{U}^{Y}(\widehat{J}^q)_t\Biggr\}\\
=:&\mathbf{II}_{1,t}+\mathbf{II}_{2,t}+\mathbf{II}_{3,t},
\end{align*}
where
\begin{align*}
\widetilde{c}(p,q)_t=
\sum_{i=(p-k_n+1)\vee1}^{p}\sum_{j=(q-k_n+1)\vee1}^{q}k_n\Delta(g)_{p-i}^n k_n\Delta(g)_{q-j}^n \bar{K}^{i j}_t.
\end{align*}


When $(p-k_n+1)\vee1\leq i\leq p-1$ and $(q-k_n+1)\vee1\leq j\leq q-1$, we have $\widehat{S}^i\vee\widehat{T}^j\leq\widehat{S}^{p-1}\vee\widehat{T}^{q-1}<\widehat{S}^p\vee\widehat{T}^q\leq\widehat{S}^{i+k_n}\vee\widehat{T}^{j+k_n}$ on $\{\widehat{I}^p\cap\widehat{J}^q\neq\emptyset\}$, hence we obtain
\begin{align*}
&\mathbf{II}_{1,t}\\
=&\frac{1}{(\psi_{HY}k_n)^2}\sum_{p,q:p\wedge q\geq k_n}\Biggl(\sum_{i=(p-k_n+1)\vee1}^{p-1}\sum_{j=(q-k_n+1)\vee1}^{q-1}k_n\Delta(g)_{p-i}^n k_n\Delta(g)_{q-j}^n+ (k_n\Delta(g)_{0}^n)^2\bar{K}^{pq}_t\\
&+\sum_{j=(q-k_n+1)\vee1}^{q-1}k_n\Delta(g)_{0}^n k_n\Delta(g)_{q-j}^n\bar{K}^{pj}_t+\sum_{i=(p-k_n+1)\vee1}^{p-1}k_n\Delta(g)_{p-i}^n k_n\Delta(g)_{0}^n\bar{K}^{iq}_t\Biggr)\widehat{K}^{pq}_t \mathfrak{U}^{X}(\widehat{I}^p)_t\mathfrak{U}^{Y}(\widehat{J}^q)_t.
\end{align*} 
Note that $\sum_{w=(v-k_n+1)\vee1}^{v}\Delta(g)_{v-w}^n=\sum_{w=0}^{k_n-1}\Delta(g)_{w}^n=g(1)-g(0)=0$ when $v\geq k_n$ and $g$ is Lipschitz continuous, we have
\begin{align*}
|\mathbf{II}_{1,t}|\lesssim\frac{1}{k_n}\sum_{p,q:p\wedge q\geq k_n}\widehat{K}^{pq}_t |\mathfrak{U}^{X}(\widehat{I}^p)_t\mathfrak{U}^{Y}(\widehat{J}^q)_t|
\lesssim\frac{1}{k_n}\left[\sum_{p}|\mathfrak{U}^{X}(\widehat{I}^p)_t|^2+\sum_q|\mathfrak{U}^{Y}(\widehat{J}^q)_t|^2\right],
\end{align*}
hence we obtain $E[\sup_{0\leq s\leq t}|\mathbf{II}_{1,s}|]\lesssim k_n^{-1}$ by the Doob inequality and [SC1]-[SC3]. Therefore, we conclude that
\begin{equation}\label{phyII1}
\sup_{0\leq s\leq t}|\mathbf{II}_{1,s}|=O_p(b_n^{1/2}).
\end{equation}
On the other hand, since $g$ is Lipschitz continuous and $(\widehat{I}^p\cap\widehat{J}^q\neq\emptyset)\Rightarrow(|p-q|\leq 1)$ by Lemma \ref{advantage}(b), we have $|\mathbf{II}_{2,t}|\lesssim\sum_{p=1}^{k_n}|\mathfrak{U}^{X}(\widehat{I}^p)_t|^2+\sum_{q=1}^{k_n}|\mathfrak{U}^{Y}(\widehat{J}^q)_t|^2,$
hence the Doob inequality, [SC1]-[SC2] and $(\ref{SA4})$ yield $E\left[\sup_{0\leq s\leq t}|\mathbf{II}_{2,s}|\right]\lesssim b_n^{\xi'-1/2}$. Therefore, we obtain
\begin{equation}\label{phyII2}
\sup_{0\leq s\leq t}|\mathbf{II}_{2,s}|=o_p(b_n^{\gamma}).
\end{equation}


Now we estimate $\mathbf{II}_{3,t}$. Since
\begin{align*}
\mathbf{II}_{3,t}
=\frac{1}{(\psi_{HY}k_n)^2}\sum_{i,j=1}^{\infty}\bar{K}^{i j}_t\sum_{p,q=0}^{k_n-1}k_n\Delta(g)_{p}^n k_n\Delta(g)_{q}^n\mathfrak{U}^{X}(\widehat{I}^{i+p})\mathfrak{U}^{Y}(\widehat{J}^{j+q})_t 1_{\{\widehat{I}^{i+p}\cap\widehat{J}^{j+q}=\emptyset\}},
\end{align*}
we can decompose it as
\begin{align*}
\mathbf{II}_{3,t}
=&\frac{1}{(\psi_{HY}k_n)^2}\sum_{i,j=1}^{\infty}\bar{K}^{i j}_t\sum_{p,q=0}^{k_n-1}\left\{(g')^n_p(g')^n_q+(g')^n_p\Delta^2(g)^n_q+\Delta^2(g)^n_p(g')^n_q+ \Delta^2(g)_{p}^n\Delta^2(g)_{q}^n\right\}\\
&\hphantom{\frac{1}{(\psi_{HY}k_n)^2}\sum_{i,j=1}^{\infty}\bar{K}^{i j}_t\sum_{p,q=0}^{k_n-1}}\times\mathfrak{U}^{X}(\widehat{I}^{i+p})_t\mathfrak{U}^{Y}(\widehat{J}^{j+q})_t 1_{\{\widehat{I}^{i+p}\cap\widehat{J}^{j+q}=\emptyset\}}\\
=:&\mathbf{II}_{3,t}^{(1)}+\mathbf{II}_{3,t}^{(2)}+\mathbf{II}_{3,t}^{(3)}+\mathbf{II}_{3,t}^{(4)},
\end{align*}
where $\Delta^2(g)_{p}^{n}=k_n\Delta(g)_{p}^{n}-(g')^n_p$. 

Consider $\mathbf{II}_{3,t}^{(1)}$ first. Integration by parts yields
$$\mathfrak{U}^{X}(\widehat{I}^{i+p})_t\mathfrak{U}^{Y}(\widehat{J}^{j+q})_t1_{\{\widehat{I}^{i+p}\cap\widehat{J}^{j+q}=\emptyset\}}=L(\mathfrak{U}^X,\mathfrak{U}^Y)^{i+p,j+q}_t1_{\{\widehat{I}^{i+p}\cap\widehat{J}^{j+q}=\emptyset\}},$$ 
hence we obtain $\sup_{0\leq s\leq t}|\mathbf{II}_{3,s}^{(1)}-\mathbf{M}(2)^n_s|=o_p(b_n^{\gamma})$ by Lemma \ref{lemoverlap} and linearity of integration.

Next we consider $\mathbf{II}_{3,t}^{(2)}$. Since $\bar{K}^{ij}=0$ if $|i-j|>k_n$ due to Lemma \ref{advantage}(b), we have
\begin{align*}
\sup_{0\leq s\leq t}|\mathbf{II}_{3,s}^{(2)}|\leq\frac{1}{(\psi_{HY}k_n)^2}\sum_{i,j:|i-j|\leq k_n}\sup_{0\leq s\leq t}\left|\sum_{p=0}^{k_n-1}(g')^n_p\mathfrak{U}^{X}(\widehat{I}^{i+p})_s\right|\sup_{0\leq s\leq t}\left|\sum_{q=0}^{k_n-1}\Delta^2(g)^n_q\mathfrak{U}^{Y}(\widehat{J}^{j+q})_s\right|,
\end{align*}
hence the Schwarz inequality and the Doob inequality yield
\begin{align*}
&E\left[\sup_{0\leq s\leq t}|\mathbf{II}_{3,s}^{(2)}|\right]\\
\leq&\frac{4}{(\psi_{HY}k_n)^2}\sum_{i,j:|i-j|\leq k_n}\left\{\sum_{p=0}^{k_n-1}|(g')^n_p|^2 E[|\mathfrak{U}^{X}(\widehat{I}^{i+p})_t|^2]\right\}^{1/2}\left\{\sum_{q=0}^{k_n-1}|\Delta^2(g)^n_q|^2 E[|\mathfrak{U}^{Y}(\widehat{J}^{j+q})_t|^2]\right\}^{1/2}.
\end{align*}
Since $|\Delta^2(g)^n_q|\lesssim k_n^{-1}$ because of the piecewise Lipschitz continuity of $g'$, we obtain
\begin{align*}
E\left[\sup_{0\leq s\leq t}|\mathbf{II}_{3,s}^{(2)}|\right]\lesssim\frac{1}{k_n^2}\left\{\sum_i\sum_{p=0}^{k_n-1}E[|\mathfrak{U}^{X}(\widehat{I}^{i+p})_t|^2]+\sum_j\sum_{q=0}^{k_n-1}E[|\mathfrak{U}^{Y}(\widehat{J}^{j+q})_t|^2]\right\},
\end{align*}
hence [SC1]-[SC3] imply that $E\left[\sup_{0\leq s\leq t}|\mathbf{II}_{3,s}^{(2)}|\right]\lesssim k_n^{-1}$. Consequently, we conclude that $\sup_{0\leq s\leq t}|\mathbf{II}_{3,s}^{(2)}|=o_p(b_n^{1/4})$. Similarly we can show $\sup_{0\leq s\leq t}|\mathbf{II}_{3,s}^{(3)}|=o_p(b_n^{1/4})$ and $\sup_{0\leq s\leq t}|\mathbf{II}_{3,s}^{(4)}|=o_p(b_n^{1/4})$, and thus we conclude that
\begin{equation}\label{phyII3}
b_n^{-\gamma}\{\mathbf{II}_{3,\cdot}-\mathbf{M}(2)^n\}\xrightarrow{ucp}0.
\end{equation}

Consequently, $(\ref{phyII1})$, $(\ref{phyII2})$ and $(\ref{phyII3})$ yield
\begin{equation}\label{phyII}
b_n^{-\gamma}\{\mathbf{II}-\mathbf{M}(2)^n\}\xrightarrow{ucp}0
\end{equation}
as $n\to\infty$.


(c) Note that $\sum_{i=(p-k_n+1)\vee1}^{p-1}\sum_{j=(q-k_n+1)\vee1}^{q-1}g_{p-i}^n k_n\Delta(g)_{q-j}^n=0$ when $p\wedge q\geq k_n$, we can adopt an argument similar to the proof of (b). 

(d) Similar to the proof of (c).
\end{proof}

\begin{proof}[\upshape{\bfseries{Proof of Lemma \ref{lembasic}}}]
The claim of Lemma \ref{lembasic} follows immediately from Lemma \ref{basicrep} and \ref{basicess}.
\end{proof}


\section{Proof of Lemma \ref{orthogonal}}\label{prooforthogonal}

By a localization procedure, we may assume that [SC1]-[SC3] instead of [C1]-[C3] respectively.

We will follow the strategy used in \cite{JLMPV2009} and \cite{JPV2010}. Fix a $t\in\mathbb{R}_+$ and let $\mathcal{N}$ be the set of all square-integrable martingales orthogonal to $(X,Y,Z^X,Z^Y)$ and satisfying $b_n^{-1/4}\langle \mathbf{M}^n,N\rangle_t\to^p0$ as $n\to\infty$.  Then $\mathcal{N}$ is a closed subset of the Hilbert space $\mathcal{M}_2^\perp$ of all square-integrable martingales orthogonal to $(X,Y,Z^X,Z^Y)$ by [A$1^*$] and the Kunita-Watanabe inequality.

Let $N$ be in the set $\mathcal{N}^0$ of all square-integrable martingales on $\mathcal{B}^{(0)}$ orthogonal to $(X,Y,Z^X,Z^Y)$. Then it is easy to check that $\langle \mathfrak{E}^X,N\rangle=\langle \mathfrak{E}^Y,N\rangle=0$. Hence we have $\langle\mathbf{M}^n,N\rangle=0$ because $N$ is orthogonal to $(X,Y,Z^X,Z^Y)$, so that $N\in\mathcal{N}$. Consequently, we conclude that $\mathcal{N}^0\subset\mathcal{N}$.

Let $N$ be in the set $\mathcal{N}^1$ of all square-integrable martingales having
\begin{equation}\label{ninfty}
N_{\infty}=f(\epsilon_{t_1},\dots,\epsilon_{t_q}),
\end{equation}
where $f$ is any bounded Borel function on $\mathbb{R}^{2q}$, $t_1<\cdots<t_q$ and $q\geq 1$. Then it is easy to check that $N$ takes the following form (by convention $t_0=0$ and $t_{q+1}=\infty$):
\begin{align*}
t_l\leq t<t_{l+1}\Rightarrow N_t=M(l;\epsilon_{t_1},\dots,\epsilon_{t_l})_t
\end{align*}
for $l=0,\dots,q$, and where $M(l;z_1,\dots,z_l)$ is a version of the martingale
\begin{align*}
M(l;z_1,\dots,z_l)_t=E^{(0)}\left[\int f(z_1,\dots,z_l,z_{l+1},\dots,z_q)Q_{t_{l+1}}(\mathrm{d}z_{l+1})\cdots Q_{t_{q}}(\mathrm{d}z_{q})|\mathcal{F}^{(0)}_t\right]
\end{align*}
(with obvious conventions when $l=0$ and $l=q$), which is measurable in $(z_1,\dots,z_l,\omega^{(0)})$. In particular $N$ has a locally finite variation, hence $N$ is a purely discontinuous local martingale because of Lemma I-4.14 of \cite{JS} and
\begin{align*}
[\bar{K}^{ij}_-\bullet\{\bar{V}_\alpha(\hat{\mathcal{I}}^i)_-\bullet \bar{W}_\beta(\hat{\mathcal{J}}^j)\},N]_t
=&\sum_{s:0\leq s\leq t}\bar{K}^{ij}_s\bar{V}_\alpha(\hat{\mathcal{I}}^i)_s\Delta\bar{W}_\beta(\hat{\mathcal{J}}^j)_s\Delta N_s\\
=&\sum_{l:t_l\leq t}\bar{K}^{ij}_{t_l}\bar{V}_\alpha(\hat{\mathcal{I}}^i)_{t_l}\Delta\bar{W}_\beta(\hat{\mathcal{J}}^j)_{t_l}\Delta N_{t_l}
\end{align*}
for any semimartingales $V,W$, $\alpha,\beta\in\Phi$ and $t\in\mathbb{R}_+$. Therefore, for any $V\in\{X,\mathfrak{E}^X,\mathfrak{Z}^X\}$, we have
\begin{align*}
[\bar{K}^{ij}_-\bullet\{\bar{V}_\alpha(\hat{\mathcal{I}}^i)_-\bullet \bar{Y}_\beta(\hat{\mathcal{J}}^j)\},N]_t=[\bar{K}^{ij}_-\bullet\{\bar{V}_\alpha(\hat{\mathcal{I}}^i)_-\bullet \bar{\mathfrak{Z}}^Y_\beta(\hat{\mathcal{J}}^j)\},N]_t=0,
\end{align*}
and note that the boundedness of $N$ we also have
\begin{align*}
|[\bar{K}^{ij}_-\bullet\{\bar{V}_g(\hat{\mathcal{I}}^i)_-\bullet \bar{\mathfrak{E}}^Y_{g'}(\hat{\mathcal{J}}^j)\},N]_t|
\lesssim&\sum_{l:t_l\leq t}\bar{K}^{ij}_{t_l}\left|\bar{V}_\alpha(\hat{\mathcal{I}}^i)_{t_l}\right|\left|\frac{1}{k_n}\sum_{k=0}^{k_n-1}(g')^n_k\epsilon^Y_{t_l}1_{\{\hat{T}^{j+k}=t_l\}}\right|.
\end{align*}
Hence, the Schwarz inequality and $(\ref{sumbarK})$ yield
\begin{align*}
&E\left[\left|\sum_{i,j=1}^{\infty}[\bar{K}^{ij}_-\bullet\{\bar{V}_\alpha(\hat{\mathcal{I}}^i)_-\bullet \bar{\mathfrak{E}}^Y_{g'}(\hat{\mathcal{J}}^j)\},N]_t\right|\right]\\
\leq&\sum_{l:t_l\leq t}\left\{\sum_{i,j=1}^{\infty}E\left[\bar{K}^{ij}_{t_l}|\bar{V}_{\alpha}(\hat{\mathcal{I}}^i)_{t_l}|^2\right]\right\}^{1/2}\left\{\sum_{i,j=1}^{\infty}E\left[\bar{K}^{ij}_{t_l}\left|\frac{1}{k_n}\sum_{k=0}^{k_n-1}(g')^n_k\epsilon^Y_{t_l}1_{\{\hat{T}^{j+k}=t_l\}}\right|^2\right]\right\}^{1/2}\\
\lesssim&\sum_{l:t_l\leq t}\left\{\sum_{i=1}^{\infty}E\left[|\bar{V}_{\alpha}(\hat{\mathcal{I}}^i)_{t_l}|^2\right]\right\}^{1/2}\left\{\sum_{j=1}^{\infty}E\left[\left|\sum_{k=0}^{k_n-1}(g')^n_k\epsilon^Y_{t_l}1_{\{\hat{T}^{j+k}=t_l\}}\right|^2\right]\right\}^{1/2}.
\end{align*}
[SC1]-[SC3] imply that $\sum_{i=1}^{\infty}E\left[|\bar{V}_{\alpha}(\hat{\mathcal{I}}^i)_{t_l}|^2\right]\lesssim k_n$. Moreover, since
\begin{align*}
\sum_j\left|\sum_{k=0}^{k_n-1}(g')^n_k\epsilon^Y_{t_l}1_{\{\hat{T}^{j+k}=t_l\}}\right|^2=\sum_j\sum_{k=0}^{k_n-1}|(g')^n_k\epsilon^Y_{t_l}|^21_{\{\hat{T}^{j+k}=t_l\}}\leq k_n\| g'\|_\infty|\epsilon^Y_{t_l}|^2,
\end{align*}
we conclude that
\begin{align*}
E\left[\left|\sum_{i,j=1}^{\infty}[\bar{K}^{ij}_-\bullet\{\bar{V}_\alpha(\hat{\mathcal{I}}^i)_-\bullet \bar{\mathfrak{E}}^Y_{g'}(\hat{\mathcal{J}}^j)\},N]_t\right|\right]\lesssim l k_n.
\end{align*}
Consequently, we obtain
\begin{align*}
\sum_{i,j=1}^{\infty}[\bar{K}^{ij}_-\bullet\{\bar{L}_{\alpha}(\hat{\mathcal{I}}^i)_-\bullet \bar{M}_{\beta}(\hat{\mathcal{J}}^j)\},N]_t=O_p(k_n)
\end{align*}
for any $(L,\alpha)\in\{(X,g),(\mathfrak{E}^X,g'),(\mathfrak{Z}^X,g')\}$ and $(M,\beta)\in\{(Y,g),(\mathfrak{E}^Y,g'),(\mathfrak{Z}^Y,g')\}$. By symmetry we also obtain
\begin{align*}
\sum_{i,j=1}^{\infty}[\bar{K}^{ij}_-\bullet\{\bar{M}_{\beta}(\hat{\mathcal{J}}^j)_-\bullet \bar{L}_{\alpha}(\hat{\mathcal{I}}^i)\},N]_t=O_p(k_n)
\end{align*}
for any $(L,\alpha)\in\{(X,g),(\mathfrak{E}^X,g'),(\mathfrak{Z}^X,g')\}$ and $(M,\beta)\in\{(Y,g),(\mathfrak{E}^Y,g'),(\mathfrak{Z}^Y,g')\}$. 
Consequently, we obtain $[\mathbf{M}^n,N]=O_p(k_n^{-1})=o_p(b_n^{1/4})$. Since $\langle \mathbf{M}^n,N\rangle$ is the predictable compensator of $[\mathbf{M}^n,N]$ by Proposition I-4.50 of \cite{JS}, we conclude that $N\in\mathcal{N}$.

Since $\mathcal{N}^0\cup\mathcal{N}^1$ is a total subset of $\mathcal{M}_2^\perp$, we conclude that $\mathcal{N}=\mathcal{M}_2^\perp$. This completes the proof of the lemma.

\hfill $\Box$


\section{Proof of Lemma \ref{Lindeberg}}\label{proofLindeberg}

Exactly as in Section \ref{prooflembasic}, we can use a localization procedure for the proof, and which allows us to replace the conditions [A4] and [N] by the following strengthened versions:
\begin{enumerate}
\item[{[SA4]}] $\xi\vee\frac{9}{10}<\xi'$ and $(\ref{SA4})$ holds.
\item[{[SN]}] We have [N], and the process $(\int |z|^8Q_t(\mathrm{d}z))_{t\in\mathbb{R}_+}$ is bounded. Furthermore, for any $\lambda>0$ there exists a positive constant $C_\lambda$ such that
\begin{align*}
E\left[|\Psi^{ij}_{t}-\Psi^{ij}_{(t-h)_+}|\big|\mathcal{F}_{(t-h)_+}\right]\leq C_{\lambda} h^{1/2-\lambda}
\end{align*}
for every $i,j\in\{1,2\}$ and every $t,h>0$.
\end{enumerate}

\begin{proof}[\upshape{\bfseries{Proof of Lemma \ref{Lindeberg}}}]
By a localization procedure, we may assume that [SC1]-[SC3], [SA4], [SN] and $(\ref{absmod2})$ hold.

Since Lemma \ref{HYlem3.1} and Eq.~I-4.36 in \cite{JS} yield
\begin{align*}
\Delta\mathbf{M}^n_s=\frac{1}{(\psi_{HY}k_n)^2}\sum_{i,j}\left\{\bar{K}^{ij}_s\bar{\mathsf{X}}(\widehat{\mathcal{I}})^i_s\Delta\bar{\mathfrak{E}}^Y_{g'}(\widehat{\mathcal{J}})^j_s+\bar{K}^{ij}_s\bar{\mathsf{Y}}(\widehat{\mathcal{I}})^i_s\Delta\bar{\mathfrak{E}}^X_{g'}(\widehat{\mathcal{J}})^j_s\right\},
\end{align*}
where $\bar{\mathsf{X}}(\widehat{\mathcal{I}})^i_s=\bar{X}_g(\widehat{\mathcal{I}})^i_s+\bar{\mathfrak{U}}^X_{g'}(\widehat{\mathcal{I}})^i_s$ and $\bar{\mathsf{Y}}(\widehat{\mathcal{J}})^j_s=\bar{Y}_g(\widehat{\mathcal{J}})^j_s+\bar{\mathfrak{U}}^Y_{g'}(\widehat{\mathcal{J}})^j_s$, it is sufficient to prove that
\begin{equation}\label{aimLindeberg}
\frac{b_n^{-1}}{k_n^8}\sum_{0\leq s\leq t}\left|\sum_{i,j}\bar{K}^{ij}_s\bar{\mathsf{X}}(\widehat{\mathcal{I}})^i_s\Delta\bar{\mathfrak{E}}^Y_{g'}(\widehat{\mathcal{J}})^j_s\right|^4\to^p 0,\qquad
\frac{b_n^{-1}}{k_n^8}\sum_{0\leq s\leq t}\left|\sum_{i,j}\bar{K}^{ij}_s\bar{\mathsf{Y}}(\widehat{\mathcal{J}})^j_s\Delta\bar{\mathfrak{E}}^X_{g'}(\widehat{\mathcal{I}})^i_s\right|^4\to^p 0
\end{equation}
as $n\to\infty$ for any $t>0$. Since
\begin{align*}
\sum_{i,j}\bar{K}^{ij}_s\bar{\mathsf{X}}(\widehat{\mathcal{I}})^i_s\Delta\bar{\mathfrak{E}}^Y_{g'}(\widehat{\mathcal{J}})^j_s
=&-\frac{1}{k_n}\sum_{i,j}\sum_{q=0}^{k_n-1}(g')^n_q\bar{K}^{ij}_s\bar{\mathsf{X}}(\widehat{\mathcal{I}})^i_s\epsilon^Y_{\widehat{T}^{j+q}}1_{\{\widehat{T}^{j+q}=s\}}\\
=&-\frac{1}{k_n}\sum_{q=1}^{\infty}\epsilon^Y_{\widehat{T}^{q}}1_{\{\widehat{T}^{q}=s\}}\sum_{i=1}^{\infty}\sum_{j=(q-k_n+1)\vee1}^{q}(g')^n_{q-j}\bar{K}^{ij}_{\widehat{T}^q}\bar{\mathsf{X}}(\widehat{\mathcal{I}})^i_{\widehat{T}^q},
\end{align*}
we have
\begin{align*}
\left|\sum_{i,j}\bar{K}^{ij}_s\bar{\mathsf{X}}(\widehat{\mathcal{I}})^i_s\Delta\bar{\mathfrak{E}}^Y_{g'}(\widehat{\mathcal{J}})^j_s\right|^4
=\frac{1}{k_n^4}\sum_{q=1}^{\infty}\left(\epsilon^Y_{\widehat{T}^{q}}\right)^41_{\{\widehat{T}^{q}=s\}}\left|\sum_{i=1}^{\infty}\sum_{j=(q-k_n+1)\vee1}^{q}(g')^n_{q-j}\bar{K}^{ij}_{\widehat{T}^q}\bar{\mathsf{X}}(\widehat{\mathcal{I}})^i_{\widehat{T}^q}\right|^4
\end{align*}
because $\widehat{T}^q\neq\widehat{T}^{q'}$ if $q\neq q'$. Moreover, since $\bar{K}^{ij}\equiv 0$ when $|i-j|>k_n$ due to Lemma \ref{advantage}(b), we have
\begin{align*}
\left|\sum_{i,j}\bar{K}^{ij}_s\bar{\mathsf{X}}(\widehat{\mathcal{I}})^i_s\Delta\bar{\mathfrak{E}}^Y_{g'}(\widehat{\mathcal{J}})^j_s\right|^4
\leq&\| g'\|_\infty k_n^2\sum_{q=1}^{\infty}\left(\epsilon^Y_{\widehat{T}^{q}}\right)^41_{\{\widehat{T}^{q}=s\}}\sum_{j=(q-k_n+1)\vee1}^{q}\sum_{i:|i-j|\leq k_n}\left|\bar{\mathsf{X}}(\widehat{\mathcal{I}})^i_{\widehat{T}^q}\right|^4\\
=&\| g'\|_\infty k_n^2\sum_{i,j:|i-j|\leq k_n}\sum_{q=0}^{k_n-1}\left(\epsilon^Y_{\widehat{T}^{j+q}}\right)^4\left|\bar{\mathsf{X}}(\widehat{\mathcal{I}})^i_{\widehat{T}^{j+q}}\right|^41_{\{\widehat{T}^{j+q}=s\}},
\end{align*}
hence we obtain
\begin{align*}
\sum_{0\leq s\leq t}\left|\sum_{i,j}\bar{K}^{ij}_s\bar{\mathsf{X}}(\widehat{\mathcal{I}})^i_s\Delta\bar{\mathfrak{E}}^Y_{g'}(\widehat{\mathcal{J}})^j_s\right|^4
\leq\| g'\|_\infty k_n^2\sum_{i,j:|i-j|\leq k_n}\sum_{q=0}^{k_n-1}\left(\epsilon^Y_{\widehat{T}^{j+q}}\right)^4\left|\bar{\mathsf{X}}(\widehat{\mathcal{I}})^i_{\widehat{T}^{j+q}\wedge t}\right|^4 1_{\{\widehat{T}^{j+q}\leq t\}}.
\end{align*}
Therefore, the Schwarz inequality and [SN] yield
\begin{align*}
E_0\left[\sum_{0\leq s\leq t}\left|\sum_{i,j}\bar{K}^{ij}_s\bar{\mathsf{X}}(\widehat{\mathcal{I}})^i_s\Delta\bar{\mathfrak{E}}^Y_{g'}(\widehat{\mathcal{J}})^j_s\right|^4\right]
\lesssim k_n^2\sum_{i,j:|i-j|\leq k_n}\sum_{q=0}^{k_n-1}\left\{E_0\left[\left|\bar{\mathsf{X}}(\widehat{\mathcal{I}})^i_{\widehat{T}^{j+q}\wedge t}\right|^8\right]\right\}^{1/2}1_{\{\widehat{T}^{j+q}\leq t\}}.
\end{align*}
Now, the Burkholder-Davis-Gundy inequality and [SN] imply
\begin{align*}
E_0\left[\left|\bar{\mathfrak{E}}^X_{g'}(\widehat{\mathcal{I}})^i_{\widehat{T}^{j+q}\wedge t}\right|^8\right]
\lesssim k_n^{-8}E_0\left[\left\{\sum_{p=0}^{k_n-1}\left|(g')^n_p\epsilon^X_{\widehat{S}^{i+q}}\right|^21_{\{\widehat{S}^{i+p}\leq t\}}\right\}^4\right]
\lesssim k_n^{-4}.
\end{align*} 
Combining this with $(\ref{absmod2})$ and [SC3], we conclude that
\begin{align*}
E_0\left[\sum_{0\leq s\leq t}\left|\sum_{i,j}\bar{K}^{ij}_s\bar{\mathsf{X}}(\widehat{\mathcal{I}})^i_s\Delta\bar{\mathfrak{E}}^Y_{g'}(\widehat{\mathcal{J}})^j_s\right|^4\right]
\lesssim k_n^4 b_n^{-1}(k_n\bar{r}_n|\log b_n|)^2=k_n^6 b_n^{2\xi'-1}|\log b_n|^2,
\end{align*}
and thus we obtain
\begin{align*}
\frac{b_n^{-1}}{k_n^8}E_0\left[\sum_{0\leq s\leq t}\left|\sum_{i,j}\bar{K}^{ij}_s\bar{\mathsf{X}}(\widehat{\mathcal{I}})^i_s\Delta\bar{\mathfrak{E}}^Y_{g'}(\widehat{\mathcal{J}})^j_s\right|^4\right]
\lesssim b_n^{2\xi'-1}|\log b_n|^2=o(1)
\end{align*}
because $\xi'>9/10$. Consequently, we have proved the first equation of $(\ref{aimLindeberg})$. By symmetry we also obtain the second equation of $(\ref{aimLindeberg})$, hence we complete the proof.
\end{proof}


\section{Proof of Proposition \ref{HYprop3.2}}\label{proofHYprop3.2}

\begin{lem}\label{psi}
Let $\alpha,\beta\in\Phi$.
\begin{enumerate}[\normalfont (a)]
\item $\psi_{\beta,\alpha}(x)=\psi_{\alpha,\beta}(-x)$ for all $x\in\mathbb{R}$.
\item $\psi_{\alpha,\beta}(x)=0$ for $x\notin[-2,2]$.
\item $\psi_{\alpha,\beta}$ is differentiable, and $\psi'_{\alpha,\beta}=\psi_{\alpha,\beta'}$ if $\beta$ is piecewise $C^1$.
\item $\psi_{\alpha,\alpha}$ is an even functions. Furthermore, $\psi_{\alpha,\alpha'}$ and $\psi_{\alpha',\alpha}$ are odd functions if $\alpha$ is piecewise $C^1$. 
\end{enumerate}
\end{lem}

\begin{proof}
Since $(x+u-1\leq v\leq x+u+1)\Leftrightarrow(-x+v-1\leq u\leq -x+v+1)$ and $\alpha(w)=\beta(w)=0$ if $w\notin[0,1]$, Fubini's theorem yields (a). (b) and (c) are obvious. (d) immediately follows (a) and (b).
\end{proof}

\begin{lem}\label{lemHYprop3.2}
Let  $\alpha,\beta,\alpha',\beta'\in\Phi$ and let $(M^n)$, $(N^n)$, $(M^{\prime n})$ and $(N^{\prime n})$ be four sequences of locally square-integrable martingales such that
\begin{equation}\label{tight}
\langle M^n\rangle_t=O_p(1),\qquad\langle N^n\rangle_t=O_p(1),\qquad
\langle M^{\prime n}\rangle_t=O_p(1),\qquad\langle N^{\prime n}\rangle_t=O_p(1)
\end{equation}
and
\begin{equation}\label{nagasa}
\left.\begin{array}{c}
\sup_{p\in\mathbb{N}}\langle M^n\rangle(\widehat{I}^p)_t=o_p(b_n^{\xi'}),\qquad
\sup_{q\in\mathbb{N}}\langle N^n\rangle(\widehat{J}^q)_t=o_p(b_n^{\xi'}),\\
\sup_{p\in\mathbb{N}}\langle M^{\prime n}\rangle(\widehat{I}^p)_t=o_p(b_n^{\xi'}),\qquad
\sup_{q\in\mathbb{N}}\langle N^{\prime n}\rangle(\widehat{J}^q)_t=o_p(b_n^{\xi'})
\end{array}\right\}
\end{equation}
as $n\to\infty$ for any $t\in\mathbb{R}_+$. Then, we have
\begin{align*}
&\frac{1}{k_n^4}\sum_{i,j,i',j'}(\bar{K}^{ij}_-\bar{K}^{i'j'}_-)\bullet V_{\alpha,\beta;\alpha',\beta'}(M^n,N^n;M^{\prime n},N^{\prime n})^{iji'j'}_t\\
=&\sum_{p,q=1}^{\infty}\psi_{\alpha,\beta}\left(\frac{q-p}{k_n}\right)\psi_{\alpha',\beta'}\left(\frac{q-p}{k_n}\right)\langle M^n,M^{\prime n}\rangle(\widehat{I}^{p})_t\langle N^n,N^{\prime n}\rangle(\widehat{J}^{q})_t+\mathbf{II}_t
\end{align*}
and
\begin{equation}\label{HYprop3.2II}
\left. \begin{aligned}
\mathbf{II}_t
=&\sum_{p,q=1}^{\infty}\psi_{\alpha,\beta}\left(\frac{q-p}{k_n}\right)\psi_{\alpha',\beta'}\left(\frac{p-q}{k_n}\right)\langle M^n,N^{\prime n}\rangle(\widehat{I}^{p})_t\langle M^{\prime n},N^{n}\rangle(\widehat{J}^{q})_t+o_p(b_n^{1/2})\\
=&\sum_{p,q=1}^{\infty}\psi_{\alpha,\beta}\left(\frac{q-p}{k_n}\right)\psi_{\alpha',\beta'}\left(\frac{p-q}{k_n}\right)\langle M^n,N^{\prime n}\rangle(\widehat{I}^{p})_t\langle M^{\prime n},N^{n}\rangle(\widehat{I}^{q})_t+o_p(b_n^{1/2})\\
=&\sum_{p,q=1}^{\infty}\psi_{\alpha,\beta}\left(\frac{q-p}{k_n}\right)\psi_{\alpha',\beta'}\left(\frac{p-q}{k_n}\right)\langle M^n,N^{\prime n}\rangle(\widehat{J}^{p})_t\langle M^{\prime n},N^{n}\rangle(\widehat{J}^{q})_t+o_p(b_n^{1/2})
\end{aligned}\right\}
\end{equation}
as $n\to\infty$ for any $t\in\mathbb{R}_+$.
\end{lem}

\begin{proof}
First note that $\bar{K}^{ij}_t\bar{K}^{i'j'}_t$ is a step function starting from 0 at $t=0$ and jumps to $+1$ at $t=R^\vee(i\vee i',j\vee j')$ when $\bar{I}^i\cap\bar{J}^j\neq\emptyset,\bar{I}^{i'}\cap\bar{J}^{j'}\neq\emptyset$ and that $V_{\alpha,\beta;\alpha',\beta'}(M^n,N^n;M^{\prime n},N^{\prime n})^{iji'j'}_t=0$ if $t\leq R^\vee(i\vee i',j\vee j')$, so we have $$(\bar{K}^{ij}_-\bar{K}^{i'j'}_-)\bullet V_{\alpha,\beta;\alpha',\beta'}(M^n,N^n;M^{\prime n},N^{\prime n})^{iji'j'}_t=\bar{K}^{ij}\bar{K}^{i'j'} V_{\alpha,\beta;\alpha',\beta'}(M^n,N^n;M^{\prime n},N^{\prime n})^{iji'j'}_t$$ by integration by parts. Therefore, we have
\begin{align*}
&\frac{1}{k_n^4}\sum_{i,j,i',j'}(\bar{K}^{ij}_-\bar{K}^{i'j'}_-)\bullet V_{\alpha,\beta;\alpha',\beta'}(M^n,N^n;M^{\prime n},N^{\prime n})^{iji'j'}_t\\
=&\frac{1}{k_n^4}\sum_{i,j,i',j'}\sum_{p,q,p',q'=1}^{k_n-1}\bar{K}^{ij}\bar{K}^{i'j'}\alpha^n_p\beta^n_q(\alpha')^n_{p'}(\beta')^n_{q'}\Bigl\{(\widehat{I}^{i+p}_-\widehat{I}^{i'+p'}_-\bullet\langle M^n,M^{\prime n}\rangle_t)(\widehat{J}^{j+q}_-\widehat{J}^{j'+q'}_-\bullet\langle N^n,N^{\prime n}\rangle_t)\\
&\hphantom{\frac{1}{k_n^4}\sum_{i,j,i',j'}\sum_{p,q,p',q'=1}^{k_n-1}\bar{K}^{ij}\bar{K}^{i'j'}\alpha^n_p\beta^n_q(\alpha')^n_{p'}(\beta')^n_{q'}}+(\widehat{I}^{i+p}_-\widehat{J}^{j'+q'}_-\bullet\langle M^n,N^{\prime n}\rangle_t)(\widehat{I}^{i'+p'}_-\widehat{J}^{j+q}_-\bullet\langle M^{\prime n},N^n\rangle_t)\Bigr\}\\
=&\sum_{p,q,p',q'=1}^{\infty}c_{\alpha,\beta}(p,q)c_{\alpha',\beta'}(p',q')\Bigl\{(\widehat{I}^{p}_-\widehat{I}^{p'}_-\bullet\langle M^n,M^{\prime n}\rangle_t)(\widehat{J}^{q}_-\widehat{J}^{q'}_-\bullet\langle N^n,N^{\prime n}\rangle_t)\\
&\hphantom{\sum_{p,q,p',q'=1}^{\infty}c_{\alpha,\beta}(p,q)c_{\alpha',\beta'}(p',q')}+(\widehat{I}^{p}_-\widehat{J}^{q'}_-\bullet\langle M^n,N^{\prime n}\rangle_t)(\widehat{I}^{p'}_-\widehat{J}^{q}_-\bullet\langle M^{\prime n},N^n\rangle_t)\Bigr\}\\
=:&\mathbf{I}_t+\mathbf{II}_t.
\end{align*}
Since $\widehat{I}^{p}\cap\widehat{I}^{p'}=\widehat{J}^{q}\cap\widehat{J}^{q'}=\emptyset$ if $p\neq p',q\neq q'$, we have
\begin{align*}
\mathbf{I}_t=\sum_{p,q=1}^{\infty}c_{\alpha,\beta}(p,q)c_{\alpha',\beta'}(p,q)\langle M^n,M^{\prime n}\rangle(\widehat{I}^{p})_t\langle N^n,N^{\prime n}\rangle(\widehat{J}^{q})_t.
\end{align*}
Moreover, since $|c_{\alpha,\beta}(p,q)|,|c_{\alpha',\beta'}(p,q)|\lesssim 1$ and $c_{\alpha,\beta}(p,q)=c_{\alpha',\beta'}(p,q)=0$ if $|p-q|>2k_n$ due to Lemma \ref{advantage}(b), $(\ref{nagasa})$ and the Kunita-Watanabe inequality imply that
\begin{align*}
\mathbf{I}_t=\sum_{\begin{subarray}{c}
p,q:p,q\geq k_n\\
|p-q|\leq 2k_n
\end{subarray}}c_{\alpha,\beta}(p,q)c_{\alpha',\beta'}(p,q)\langle M^n,M^{\prime n}\rangle(\widehat{I}^{p})_t\langle N^n,N^{\prime n}\rangle(\widehat{J}^{q})_t+o_p(b_n^{1/2}),
\end{align*}
hence Lemma \ref{key}, $(\ref{tight})$, $(\ref{nagasa})$ and the Kunita-Watanabe inequality yield
\begin{align*}
\mathbf{I}_t=\sum_{p,q=1}^{\infty}\psi_{\alpha,\beta}\left(\frac{q-p}{k_n}\right)\psi_{\alpha',\beta'}\left(\frac{q-p}{k_n}\right)\langle M^n,M^{\prime n}\rangle(\widehat{I}^{p})_t\langle N^n,N^{\prime n}\rangle(\widehat{J}^{q})_t+o_p(b_n^{1/2}).
\end{align*}
On the other hand, an argument similar to the above yields
\begin{align*}
\mathbf{II}_t=\sum_{p,q,p',q'=1}^{\infty}\psi_{\alpha,\beta}\left(\frac{q-p}{k_n}\right)\psi_{\alpha',\beta'}\left(\frac{q'-p'}{k_n}\right)\langle M^n,N^{\prime n}\rangle(\widehat{I}^{p}\cap\widehat{J}^{q'})_t\langle M^{\prime n},N^{n}\rangle(\widehat{I}^{p'}\cap\widehat{J}^{q})_t+o_p(b_n^{1/2}).
\end{align*}
Since $\widehat{I}^{p}\cap\widehat{J}^{q}=\emptyset$ if $|p-q|>1$ by Lemma \ref{advantage}(b), we obtain
\begin{align*}
\mathbf{II}_t
=&\sum_{p,q':|q'-p|\leq 1}\sum_{p',q:|p'-q|\leq 1}\psi_{\alpha,\beta}\left(\frac{q-p}{k_n}\right)\psi_{\alpha',\beta'}\left(\frac{q'-p'}{k_n}\right)\\
&\hphantom{\sum_{p,q':|q'-p|\leq 1}\sum_{p',q:|p'-q|\leq 1}}\times\langle M^n,N^{\prime n}\rangle(\widehat{I}^{p}\cap\widehat{J}^{q'})_t\langle M^{\prime n},N^{n}\rangle(\widehat{I}^{p'}\cap\widehat{J}^{q})_t+o_p(b_n^{1/2}),
\end{align*}
however, since $\psi_{\alpha,\beta}$ and $\psi_{\alpha',\beta'}$ are Lipschitz continuous, an argument similar to the above yield
\begin{align*}
\mathbf{II}_t
=&\sum_{p,q=1}^{\infty}\psi_{\alpha,\beta}\left(\frac{q-p}{k_n}\right)\sum_{p':|p'-q|\leq 1}\sum_{q':|q'-p|\leq 1}\psi_{\alpha',\beta'}\left(\frac{p-q}{k_n}\right)\\
&\hphantom{\sum_{p,q=1}^{\infty}\psi_{\alpha,\beta}\left(\frac{q-p}{k_n}\right)\sum_{p':|p'-q|\leq 1}\sum_{q':|q'-p|\leq 1}}\times\langle M^n,N^{\prime n}\rangle(\widehat{I}^{p}\cap\widehat{J}^{q'})_t\langle M^{\prime n},N^{n}\rangle(\widehat{I}^{p'}\cap\widehat{J}^{q})_t+o_p(b_n^{1/2}).
\end{align*}
Therefore, we conclude that
\begin{align*}
\mathbf{II}_t=\sum_{p,q=1}^{\infty}\psi_{\alpha,\beta}\left(\frac{q-p}{k_n}\right)\psi_{\alpha',\beta'}\left(\frac{p-q}{k_n}\right)\langle M^n,N^{\prime n}\rangle(\widehat{I}^{p})_t\langle M^{\prime n},N^{n}\rangle(\widehat{J}^{q})_t+o_p(b_n^{1/2}).
\end{align*}
In a similar manner we can show $(\ref{HYprop3.2II})$, and thus we complete the proof of the lemma.
\end{proof}

\begin{lem}\label{lemvanish}
Let $\alpha\in\Phi$ and let $A^n$ and $B^n$ be two processes with locally bounded variations such that
\begin{equation}\label{tightnagasa}
A^n_t=O_p(1),\qquad B^n_t=O_p(1),\qquad\sup_{p\in\mathbb{N}}|A^n(\widehat{I}^p)_t|=o_p(\bar{r}_n),\qquad\sup_{q\in\mathbb{N}}|B^n(\widehat{J}^q)_t|=o_p(\bar{r}_n)
\end{equation}
as $n\to\infty$ for any $t\in\mathbb{R}_+$. Then
\begin{align*}
b_n^{-1/2}\sum_{p,q=1}^{\infty}\psi_{g,g'}\left(\frac{q-p}{k_n}\right)\psi_{\alpha,\alpha}\left(\frac{q-p}{k_n}\right)A^n(\widehat{I}^{p})_t B^n(\widehat{J}^{q})_t\to^p 0
\end{align*}
as $n\to\infty$ for any $t\in\mathbb{R}_+$.
\end{lem}

\begin{proof}
Lemma \ref{psi}(b) yields
\begin{align*}
\left|\sum_{p,q=1}^{\infty}\psi_{g,g'}\left(\frac{q-p}{k_n}\right)\psi_{\alpha,\alpha}\left(\frac{q-p}{k_n}\right)A^n(\widehat{I}^{p})_t B^n(\widehat{J}^{q})_t\right|
\leq |A^n_t|\sup_{q\in\mathbb{N}}|B^n(\widehat{J}^q)_t|\left|\sum_{r=-2 k_n}^{2 k_n}\psi_{g,g'}\left(\frac{r}{k_n}\right)\psi_{\alpha,\alpha}\left(\frac{r}{k_n}\right)\right|.
\end{align*}
Since $\psi_{g,g'}(x)\psi_{\alpha,\alpha}(x)$ is Lipschitz continuous, we have
\begin{align*}
\frac{1}{k_n}\sum_{r=-2 k_n}^{2 k_n}\psi_{g,g'}\left(\frac{r}{k_n}\right)\psi_{\alpha,\alpha}\left(\frac{r}{k_n}\right)
=\int_{-2}^2\psi_{g,g'}(x)\psi_{\alpha,\alpha}(x)\mathrm{d}x+O\left(\frac{1}{k_n}\right),
\end{align*}
however, $\int_{-2}^2\psi_{g,g'}(x)\psi_{\alpha,\alpha}(x)\mathrm{d}x=0$ because $\psi_{g,g'}$ is an odd function and $\psi_{\alpha,\alpha}$ is an even function by Lemma \ref{psi}. Combining this with $(\ref{tightnagasa})$, we obtain
\begin{align*}
b_n^{-1/2}\left|\sum_{p,q=1}^{\infty}\psi_{g,g'}\left(\frac{q-p}{k_n}\right)\psi_{\alpha,\alpha}\left(\frac{q-p}{k_n}\right)A^n(\widehat{I}^{p})_t B^n(\widehat{J}^{q})_t\right|
\leq o_p\left(b_n^{\xi'-1/2}\right)=o_p(1)
\end{align*}
and thus we complete the proof of the lemma.
\end{proof}

\begin{proof}[\upshape{\bfseries{Proof of Proposition \ref{HYprop3.2}}}]
[B2], Lemma \ref{lemHYprop3.2} and the fact that both $\psi_{g,g}$ and $\psi_{g',g'}$ are even functions and $\psi_{g,g'}$ is an odd function yield $\langle\mathbf{M}(l)^n\rangle_t=\bar{V}^{n,l}_t+o_p(b_n^{1/2})$ for $l=1,2,3,4$ and $\langle\mathbf{M}(1)^n,\mathbf{M}(2)^n\rangle_t=\bar{V}^{n,12}_t+o_p(b_n^{1/2})$, $\langle\mathbf{M}(3)^n,\mathbf{M}(4)^n\rangle_t=\bar{V}^{n,34}_t+o_p(b_n^{1/2})$ as $n\to\infty$ for any $t\in\mathbb{R}_+$. Moreover, by Lemma \ref{lemHYprop3.2}-\ref{lemvanish} we have
\begin{align*}
\langle\mathbf{M}(1)^n,\mathbf{M}(3)^n\rangle_t=o_p(b_n^{1/2}),\qquad
\langle\mathbf{M}(1)^n,\mathbf{M}(4)^n\rangle_t=o_p(b_n^{1/2}),\\
\langle\mathbf{M}(2)^n,\mathbf{M}(3)^n\rangle_t=o_p(b_n^{1/2}),\qquad
\langle\mathbf{M}(2)^n,\mathbf{M}(4)^n\rangle_t=o_p(b_n^{1/2})
\end{align*}
as $n\to\infty$ for any $t\in\mathbb{R}_+$. Consequently, we obtain the desired result.
\end{proof}  



\section{Proof of Lemma \ref{HYlem4.1}}\label{proofHYlem4.1}


Before starting the proof, we strengthen the condition [A3] as follows:
\begin{enumerate}
\item[{[SA3]}] For each $V,W=X,Y,Z^X,Z^Y$, $[V,W]$ is absolutely continuous with a c\`adl\`ag bounded derivative adapted to $\mathbf{H}^n$, and for any $\lambda>0$ there exists a positive constant $C_\lambda$ such that
\begin{align*}
E\left[|f_{\tau_1}-f_{\tau_2}|^2\big|\mathcal{F}_{\tau_1\wedge\tau_2}\right]\leq C_{\lambda}E\left[|\tau_1-\tau_2|^{1-\lambda}\big|\mathcal{F}_{\tau_1\wedge\tau_2}\right]
\end{align*}
for any bounded $\mathbf{F}^{(0)}$-stopping times $\tau_1$ and $\tau_2$, for the density process $f=[V,W]'$.
\end{enumerate}

First we prove some lemmas.
\begin{lem}\label{checktime}
$\check{S}^k$ is $\mathcal{H}^{n}_{\widehat{S}^k}$-measurable and $\check{T}^k$ is $\mathcal{H}^{n}_{\widehat{T}^k}$-measurable for every $k$.
\end{lem}

\begin{proof}
For any $t\geq0$ we have $\{\check{S}^k\leq t\}=\bigcap_i[\{S^i\leq t, S^i<\widehat{S}^k\}\cup\{\widehat{S}^k\leq S^i\}]$. Since $\{S^i\leq t, S^i<\widehat{S}^k\}, \{\widehat{S}^k\leq S^i\}\in\mathcal{H}^{n}_{\widehat{S}^k}$ for every $k$, we obtain the desired result. Similarly we can show that $\check{T}^k$ is $\mathcal{H}^{n}_{\widehat{T}^k}$-measurable.
\end{proof}

The following lemma is a version of Lemma 2.3 of \citet{Fu2010b}:
\begin{lem}\label{useful}
Suppose that for each $n\in\mathbb{N}$ we have a sequence $(\tau^n_k)$ of $\mathbf{H}^n$-stopping times and a sequence $(\zeta^n_k)$ of random variables such that $\zeta^n_k$ is adapted to $\mathcal{H}^n_{\tau^n_k}$ for every $k$. Let $\rho>1$ and $t\geq0$, and set $N(\tau)^n_t=\sum_{k=1}^\infty1_{\{\tau^n_k\leq t\}}$.
\begin{enumerate}[\normalfont (a)]

\item If $\sum_{k=1}^{N(\tau)^n_t+1} E\left[|\zeta^n_k|^{2\wedge\rho}\big|\mathcal{H}^n_{\tau^n_{k-1}}\right]\to^p0$ as $n\to\infty$, then $\sum_{k=1}^{N(\tau)^n_t+1}\left\{\zeta^n_k-E\left[\zeta^n_k\big|\mathcal{H}^n_{\tau^n_{k-1}}\right]\right\}\to^p0$ as $n\to\infty$.

\item If $b_nN(\tau)^n_t=O_p(1)$ as $n\to\infty$ and 
$b_n\sum_{k=1}^{N(\tau)^n_t+1} E\left[|\zeta^n_k|^\rho\big|\mathcal{H}^n_{\tau^n_{k-1}}\right]=O_p(1)$ as $n\to\infty$, then
$$b_n\sum_{k=1}^{N(\tau)^n_t+1}\left\{\zeta^n_k-E\left[\zeta^n_k\big|\mathcal{H}^n_{\tau^n_{k-1}}\right]\right\}\to^p0$$
as $n\to\infty$.

\end{enumerate} 
\end{lem}

\begin{proof}
Let $\varpi=2\wedge\rho$ and let $T$ be a bounded stopping time with respect to the filtration $(\mathcal{H}^n_{\tau^n_k})_{k\in\mathbb{Z}_+}$. Then, the Burkholder-Davis-Gundy inequality and the $C_p$-inequality yield
\begin{align*}
E\left[\left|\sum_{k=1}^{T}\left\{\zeta^n_k-E\left[\zeta^n_k\big|\mathcal{H}^n_{\tau^n_{k-1}}\right]\right\}\right|^\varpi\right]
\leq C E\left[\sum_{k=1}^{T}\left\{|\zeta^n_k|^\varpi+\left|E\left[\zeta^n_k\big|\mathcal{H}^n_{\tau^n_{k-1}}\right]\right|^\varpi\right\}\right]
\end{align*}
for some positive constant $C$ independent of $n$. Since $E\left[\sum_{k=1}^{T}|\zeta^n_k|^\varpi\right]=E\left[\sum_{k=1}^{T}E\left[|\zeta^n_k|^\varpi\big|\mathcal{H}^n_{\tau^n_{k-1}}\right]\right]$ by the optional stopping theorem and $\left|E\left[\zeta^n_k\big|\mathcal{H}^n_{\tau^n_{k-1}}\right]\right|^\varpi\leq E\left[|\zeta^n_k|^\varpi\big|\mathcal{H}^n_{\tau^n_{k-1}}\right]$ by the H\"older inequality, we obtain
\begin{align*}
E\left[\left|\sum_{k=1}^{T}\left\{\zeta^n_k-E\left[\zeta^n_k\big|\mathcal{H}^n_{\tau^n_{k-1}}\right]\right\}\right|^\varpi\right]
\leq 2 C E\left[\sum_{k=1}^{T}E\left[|\zeta^n_k|^\varpi\big|\mathcal{H}^n_{\tau^n_{k-1}}\right]\right].
\end{align*}
Therefore, note that $N(\tau)^n_t+1$ is a stopping time with respect to the filtration $(\mathcal{H}^n_{\tau^n_k})$, (a) holds due to the Lenglart inequality. On the other hand, since $$\sum_{k=1}^{N(\tau)^n_t+1}E\left[|\zeta^n_k|^\varpi\big|\mathcal{H}^n_{\tau^n_{k-1}}\right]\leq(N(\tau)^n_t+1)^{1-\varpi/\rho}\left\{\sum_{k=1}^{N(\tau)^n_t+1} E\left[|\zeta^n_k|^\rho\big|\mathcal{H}^n_{\tau^n_{k-1}}\right]\right\}^{\varpi/\rho}$$ by the H\"older inequality, (b) holds due to the Lenglart inequality and the fact that $\varpi>1$.
\end{proof}


For a c\`adl\`ag function $x$ on $\mathbb{R}_+$ and an interval $I\subset\mathbb{R}_+$, set $w(x;I)=\sup_{s\in I}|x(s)|$. Moreover, define 
\begin{align*}
w'(x;\delta,T)=\inf\left\{\max_{i\leq r}w(x;[t_{i-1},t_i))\big|0=t_0<\cdots <t_r=T,\inf_{i<r}(t_i-t_{i-1})\geq\delta\right\}
\end{align*}
for each $\delta,T>0$.
\begin{lem}\label{riemann}
Let $(x_n)_{n\in\mathbb{R}_+}$ be a sequence of c\`adl\`ag functions on $\mathbb{R}_+$ which converges a c\`adl\`ag function $x$ on $\mathbb{R}_+$ for the Skorokhod topology. Let $t$ be a positive number. Suppose that for each $n\in\mathbb{N}$ there are points $s^n_i$ such that $0=s^n_0<s^n_1<\cdots<s^n_{K_n}=t$ and $\sup_i(s^n_i-s^n_{i-1})\to0$ as $n\to\infty$. Then we have
\begin{equation*}
\sum_{i=1}^{K_n}x_n(s^n_{i-1})(s^n_i-s^n_{i-1})\to\int_0^t x(s)\mathrm{d}s
\end{equation*}
as $n\to\infty$.
\end{lem}

\begin{proof}
Since $\int_0^tx_n(s)\mathrm{d}s\to\int_0^tx(s)\mathrm{d}s$ as $n\to\infty$ by the bounded convergence theorem, it is sufficient to show that
\begin{equation}\label{eqriemann}
\sum_{i=1}^{K_n}x_n(s^n_{i-1})(s^n_i-s^n_{i-1})-\int_0^tx_n(s)\mathrm{d}s\to0
\end{equation}
as $n\to\infty$. Take $\eta>0$ arbitrarily. Since $\lim_{\delta\downarrow0}\sup_{n\in\mathbb{N}}w'(x_n;\delta,t)=0$ by Theorem VI-1.5 of \cite{JS}, we can take a positive number $\delta>0$ such that $\sup_{n\in\mathbb{N}}w'(x_n;\delta,t)<\eta$. Then there exist points $\xi_i^n$ such that $0=\xi_0^n<\xi_1^n<\cdots<\xi_{\bar{m}_n}^n=t$, $\inf_{i<\bar{m}_n}(\xi_i^n-\xi_{i-1}^n)\geq\delta$ and that $\max_{m\in\{0,\dots,\bar{m}_n-1\}}\sup_{s\in[\xi_m^n,\xi_{m+1}^n)}|x_n(s)-x_n(\xi_m^n)|<\eta$, and we have
\begin{align*}
\left|\sum_{i=1}^{K_n}x_n(s^n_{i-1})(s^n_i-s^n_{i-1})-\int_0^tx_n(s)\mathrm{d}s\right|
\leq\sum_{i=1}^{K_n}\int_{s^n_{i-1}}^{s^n_i}|x_n(s^n_{i-1})-x_n(s)|\mathrm{d}s
\leq \eta t+\bar{m}_n\sup_i(s_i^n-s_{i-1}^n).
\end{align*}
Since $\bar{m}_n<t/\delta+1$ by $\inf_{i<\bar{m}_n}(\xi_i^n-\xi_{i-1}^n)\geq\delta$, we obtain
\begin{align*}
\limsup_{n\to\infty}\left|\sum_{i=1}^{K_n}x_n(s^n_{i-1})(s^n_i-s^n_{i-1})-\int_0^tx_n(s)\mathrm{d}s\right|
\leq \eta t.
\end{align*}
Since $\eta$ is arbitrary, we conclude that $(\ref{eqriemann})$ holds, and thus we complete the proof of the lemma.
\end{proof}

The following lemma is a version of Lemma 2.2 of \citet{HJY2011}. 


\begin{lem}\label{HJYlem2.2}
Suppose that $[\mathrm{A}1'](\mathrm{i})$-$(\mathrm{ii})$ holds. Suppose also that there are a sequence $(H^n)_{n\in\mathbb{R}_+}$ of c\`adl\`ag $\mathbf{H}^n$-adapted processes and a c\`adl\`ag process $H$ such that $H^n\xrightarrow{\mathrm{Sk.p.}}H$ as $n\to\infty$. Then we have 
\begin{align*}
b_n\sum_{k=1}^{N^n_t+1}H^n_{R^{k-1}}\to^p\int_0^t\frac{H_s}{G_s}\mathrm{d}s
\end{align*}
as $n\to\infty$ for any $t\in\mathbb{R}_+$. In particular it holds that $b_nN^n_t\to^p\int_0^t\frac{1}{G_s}\mathrm{d}s$ as $n\to\infty$ for any $t\in\mathbb{R}_+$.
\end{lem}

\begin{proof}
First we show that [C3] holds. Take a positive number $L$ arbitrarily. Then we have
\begin{align*}
&E\left[\sum_{k=1}^\infty 1_{\{R^k\leq t,G(1)^n_{R^k}\leq L, G^n_{R^k}\geq L^{-1},\#\mathcal{N}^n_0\leq L\}}\right]
=E\left[\sum_{k=1}^\infty\frac{G^n_{R^k}}{G^n_{R^k}}1_{\{R^k\leq t,G(1)^n_{R^k}\leq L, G^n_{R^k}\geq L^{-1},\#\mathcal{N}^n_0\leq L\}}\right]\\
\leq &E\left[\sum_{k=1}^\infty\frac{G(1)^n_{R^k}}{G^n_{R^k}}1_{\{R^k\leq t,G(1)^n_{R^k}\leq L, G^n_{R^k}\geq L^{-1}\}}\right]+L
=b_n^{-1}E\left[\sum_{k=1}^\infty\frac{|\Gamma^{k+1}|}{G^n_{R^k}}1_{\{R^k\leq t,G(1)^n_{R^k}\leq L, G^n_{R^k}\geq L^{-1}\}}\right]+L\\
\leq &L b_n^{-1}E\left[\sum_{k=1}^\infty|\Gamma^{k+1}|1_{\{R^k\leq t,G(1)^n_{R^k}\leq L\}}\right]+L
\leq L b_n^{-1}\left\{E\left[\sum_{k=1}^{N^n_t}|\Gamma^{k}|\right]+E\left[|\Gamma^{N^n_t+1}|1_{\{G(1)^n_{R^{N^n_t}}\leq L\}}\right]\right\}+L\\
\leq &L b_n^{-1}t+L^2+L.
\end{align*}
On the other hand, since $$\bigcup_{k=1}^\infty\left[\{R^k\leq t,G(1)^n_{R^k}\leq L, G^n_{R^k}\geq L^{-1},\#\mathcal{N}^n_0\leq L\}^c\right]\subset\left\{\sup_{0\leq s\leq t}G(1)^n_s>L\right\}\cup\left\{\inf_{0\leq s\leq t}G^n_s<L^{-1}\right\}\cup\{\#\mathcal{N}^0_n>L\},$$ note that 
\begin{equation}\label{holder}
G(1)^n_t\leq\{G(\rho)^n_t\}^{1/\rho}
\end{equation}
by the H\"older inequality, [A1$'$](i)-(ii) yield $$\limsup_{L\to\infty}\limsup_{n\to\infty}P\left(\bigcup_{k=1}^\infty\left[\{R^k\leq t,G(1)^n_{R^k}\leq L, G^n_{R^k}\geq L^{-1},\#\mathcal{N}^n_0\leq L\}^c\right]\right)=0.$$ Consequently, we obtain [C3].

Next, since [A1$'$](i)-(ii) and [C3] imply that
$b_n\sum_{k=1}^{N^n_t+1}\left|\frac{H^n_{R^{k-1}}}{G^n_{R^{k-1}}}\right|^\rho E\left[\left(b_n^{-1}|\Gamma^k|\right)^\rho\big|\mathcal{H}^n_{R^{k-1}}\right]=O_p(1),$
Lemma \ref{useful}(b) yields $b_n\sum_{k=1}^{N^n_t+1}\frac{H^n_{R^{k-1}}}{G^n_{R^{k-1}}}\left\{b_n^{-1}|\Gamma^k|-G(1)^n_{R^{k-1}}\right\}=o_p(1).$ Therefore, note that the fact that $(\#\mathcal{N}^0_n)_{n\in\mathbb{N}}$ is tight, we conclude that
\begin{align*}
b_n\sum_{k=1}^{N^n_t+1}H^n_{R^{k-1}}-\sum_{k=1}^{N^n_t+1}\frac{H^n_{R^{k-1}}}{G(1)^n_{R^{k-1}}}|\Gamma^k|\to^p 0
\end{align*}
as $n\to\infty$. Since Lemma \ref{riemann} implies that
\begin{align*}
\sum_{k=1}^{N^n_t+1}\frac{H^n_{R^{k-1}}}{G(1)^n_{R^{k-1}}}|\Gamma^k|\to^p\int_0^t\frac{H_s}{G_s}\mathrm{d}s
\end{align*}
as $n\to\infty$, we complete the proof.
\end{proof}


\begin{proof}[\upshape{\bfseries{Proof of Lemma \ref{HYlem4.1}}}]
By a localization procedure, we may assume that [SA3] and [SC2] hold instead of [A3] and [C2] respectively. Recall that for a function $\alpha$ on $\mathbb{R}$ we write $\alpha^n_p=\alpha(p/k_n)$ for each $n\in\mathbb{N}$ and $p\in\mathbb{Z}$.

(a) Without loss of generality, we may assume that $G(1)^n\to G$ a.s. as $n\to\infty$ in $\mathbb{D}(\mathbb{R}_+)$.
First we show that
\begin{equation}\label{upsilon}
\sum_{p,q=1}^{\infty}|(\psi_{g,g})^n_{q-p}|^2[ X](\widehat{I}^p)_t[ Y](\widehat{J}^q)_t
=\sum_{p,q=1}^{\infty}|(\psi_{g,g})^n_{q-p}|^2[ X](\widehat{I}^p)_t[ Y](\Gamma^q)_t+o_p(b_n^{1/2}).
\end{equation}
We have
\begin{align*}
&\sum_{p,q=1}^{\infty}|(\psi_{g,g})^n_{q-p}|^2[ X](\widehat{I}^p)_t\{[ Y](\widehat{J}^q)_t-[ Y](\Gamma^q)_t\}\\
=&\sum_{p,q=1}^{\infty}|(\psi_{g,g})^n_{q-p}|^2[ X](\widehat{I}^p)_t\{([ Y]_{\widehat{T}^q\wedge t}-[ Y]_{R^q\wedge t})-([ Y]_{\widehat{T}^{q-1}\wedge t}-[ Y]_{R^{q-1}\wedge t})\}\\
=&\sum_{p,q=1}^{\infty}\left\{|(\psi_{g,g})^n_{q-p}|^2-|(\psi_{g,g})^n_{q+1-p}|^2\right\}[ X](\widehat{I}^p)_t([ Y]_{\widehat{T}^q\wedge t}-[ Y]_{R^q\wedge t}).
\end{align*}
Therefore, [SA3], [A4] and the fact that $\psi_{g,g}$ is Lipschitz continuous and equal to 0 outside $[-2,2]$ by Lemma \ref{psi}(b) imply the desired claim. By symmetry, we also obtain
\begin{equation}\label{upsilon2}
\sum_{p,q=1}^{\infty}|(\psi_{g,g})^n_{q-p}|^2[ X](\widehat{I}^p)_t[ Y](\widehat{J}^q)_t
=\sum_{p,q=1}^{\infty}|(\psi_{g,g})^n_{q-p}|^2[ X](\Gamma^p)_t[ Y](\Gamma^q)_t+o_p(b_n^{1/2}),
\end{equation}
hence we have
\begin{align*}
&\sum_{p,q=1}^{\infty}|(\psi_{g,g})^n_{q-p}|^2[ X](\widehat{I}^p)_t[ Y](\widehat{J}^q)_t\\
=&\sum_{p,q:p<q}|(\psi_{g,g})^n_{q-p}|^2[ X](\Gamma^p)_t[ Y](\Gamma^q)_t
+\sum_{p,q:p>q}|(\psi_{g,g})^n_{q-p}|^2[ X](\Gamma^p)_t[ Y](\Gamma^q)_t+o_p(b_n^{1/2}).
\end{align*}

Consider the first term of the right hand of the above equation. Let $\upsilon_n=(t+1)\wedge\inf\{s|r_n(s)>\bar{r}_n\}$, $\widetilde{R}^k=R^k\wedge\upsilon_n$ and $\widetilde{\Gamma}^k=[\widetilde{R}^{k-1},\widetilde{R}^k)$. Then obviously $\upsilon_n$ is $\mathbf{F}^{(0)}$-stopping time and $\sup_k|\widetilde{\Gamma}^k(t)|\leq 2\bar{r}_n$. Therefore, we have
\begin{align*}
E\left[\big|[Y](\widetilde{\Gamma}^q)_t-[Y]'_{\widetilde{R}^{q-1}\wedge t}|\widetilde{\Gamma}^q(t)|\big|\Big|\mathcal{F}_{\widetilde{R}^{q-1}\wedge t}\right]
\leq \int_{\widetilde{R}^{q-1}\wedge t}^{(\widetilde{R}^{q-1}+2\bar{r}_n)\wedge t}E\left[\big|[Y]'_u-[Y]'_{\widetilde{R}^{q-1}\wedge t}|\big|\Big|\mathcal{F}_{\widetilde{R}^{q-1}\wedge t}\right]\mathrm{d}u
\lesssim b_n^{\frac{3}{2}\xi'-\lambda}
\end{align*}
for any $\lambda>0$ by the Schwarz inequality and [SA3]. Hence we obtain
\begin{align*}
&E\left[\left|\sum_{p=1}^{\infty}[ X](\Gamma^p)_t\sum_{q:p<q}|(\psi_{g,g})^n_{q-p}|^2\left\{[ Y](\Gamma^q)_t-[ Y]'_{R^{q-1}}|\Gamma^q(t)| \right\}\right|;\upsilon_n>t\right]\\
\leq&E\left[\sum_{p=1}^{\infty}[ X](\widetilde{\Gamma}^p)_t\sum_{q:p<q}|(\psi_{g,g})^n_{q-p}|^2E\left[\big|[Y](\widetilde{\Gamma}^q)_t-[Y]'_{\widetilde{R}^{q-1}\wedge t}|\widetilde{\Gamma}^q(t)|\big|\Big|\mathcal{F}_{\widetilde{R}^{q-1}\wedge t}\right]\right]\\
\lesssim&b_n^{\frac{3}{2}\xi'-\lambda}E\left[\sum_{p=1}^{\infty}[ X](\widetilde{\Gamma}^p)_t\sum_{q:p<q}|(\psi_{g,g})^n_{q-p}|^2\right]
=o\left(b_n^{\frac{3}{2}\xi'-\lambda-\frac{1}{2}}\right)
\end{align*}
by Lemma \ref{psi}(b). Since $P(\upsilon_n\leq t)\to 0$ as $n\to\infty$ by [A4], we conclude that
\begin{equation}\label{localcons}
\sum_{p,q:p<q}|(\psi_{g,g})^n_{q-p}|^2[ X](\Gamma^p)_t[ Y](\Gamma^q)_t
=\sum_{p,q:p<q}|(\psi_{g,g})^n_{q-p}|^2[ X](\Gamma^p)_t[ Y]'_{R^{q-1}}|\Gamma^q(t)|+o_p(b_n^{1/2}).
\end{equation}
Next, [SC1] yields 
\begin{align*}
\left|\sum_{p,q:p<q}|(\psi_{g,g})^n_{q-p}|^2[ X](\Gamma^p)_t[ Y]'_{R^{q-1}}\left\{|\Gamma^q(t)|-|\Gamma^q|1_{\{R^{q-1}\leq t\}}\right\}\right|
\lesssim \sum_{p=1}^{N^n_t}|(\psi_{g,g})^n_{N^n_t+1-p}|^2|\Gamma^p(t)||\Gamma^{N^n_t+1}|.
\end{align*}
Take $L>0$ arbitrarily. Then, we have
\begin{align*}
&b_n^{-1/2}E\left[\sum_{p=1}^{N^n_t}|(\psi_{g,g})^n_{N^n_t+1-p}|^2|\Gamma^p(t)||\Gamma^{N^n_t+1}|;G(1)^n_t\leq L\right]\\
=&b_n^{1/2} E\left[\sum_{p=1}^{N^n_t}|(\psi_{g,g})^n_{N^n_t+1-p}|^2|\Gamma^p(t)|G(1)^n_t;G(1)^n_t\leq L\right]
\leq 2k_n b_n^{1/2}\|\psi_{g,g}\|_\infty LE[r_n(t)]
\end{align*}
by Lemma \ref{psi}(b). Since $E[r_n(t)]\to0$ as $n\to\infty$ by the bounded convergence theorem, for any $\eta>0$ we obtain
$\limsup_{n}P(b_n^{-1/2}\sum_{p=1}^{N^n_t}|(\psi_{g,g})^n_{N^n_t+1-p}|^2|\Gamma^p(t)||\Gamma^{N^n_t+1}|>\eta)\leq\limsup_{n}P(G(1)^n_t>L)$. Since $(\ref{holder})$ and [A1$'$](ii) imply that $\limsup_{n}P(G(1)^n_t>L)\to0$ as $L\to\infty$, we conclude that $$\sum_{p=1}^{N^n_t}|(\psi_{g,g})^n_{N^n_t+1-p}|^2|\Gamma^p(t)||\Gamma^{N^n_t+1}|=o_p(b_n^{1/2}),$$ and thus we obtain
\begin{equation}\label{one}
\sum_{p,q:p<q}|(\psi_{g,g})^n_{q-p}|^2[ X](\Gamma^p)_t[ Y]'_{R^{q-1}}|\Gamma^q(t)|
=\sum_{p,q:p<q}|(\psi_{g,g})^n_{q-p}|^2[ X](\Gamma^p)_t[ Y]'_{R^{q-1}}|\Gamma^q|1_{\{R^{q-1}\leq t\}}+o_p(b_n^{1/2}).
\end{equation}
On the other hand, [SC1], [SC3] and Lemma \ref{psi}(b) yield
\begin{align*}
\sum_{q=2}^\infty\left| b_n^{1/2}\sum_{p:p<q}|(\psi_{g,g})^n_{q-p}|^2[X](\Gamma^p)_t\right|^\varpi E\left[\left(b_n^{-1}|\Gamma^q|\right)^\varpi\big|\mathcal{H}^n_{R^{q-1}}\right]1_{\{R^{q-1}\leq t\}}
\lesssim b_n^{-1}r_n(t)^\varpi\sup_{0\leq s\leq t}G(\varpi)^n_{s},
\end{align*}
where $\varpi=2\wedge\rho$. Hence [A1$'$](ii), the H\"older inequality, $(\ref{A4})$ and the fact that $\varpi\xi'\geq 1$ we obtain
\begin{align*}
\sum_{q=2}^\infty\left| b_n^{1/2}\sum_{p:p<q}|(\psi_{g,g})^n_{q-p}|^2[X](\Gamma^p)_t\right|^\varpi E\left[\left(b_n^{-1}|\Gamma^q|\right)^\varpi\big|\mathcal{H}^n_{R^{q-1}}\right]1_{\{R^{q-1}\leq t\}}
\to^p 0
\end{align*}
as $n\to\infty$. Therefore, Lemma \ref{useful}(a) yields
\begin{align*}
\sum_{p,q:p<q}|(\psi_{g,g})^n_{q-p}|^2[ X](\Gamma^p)_t[ Y]'_{R^{q-1}}|\Gamma^q|1_{\{R^{q-1}\leq t\}}
=b_n\sum_{p,q:p<q}|(\psi_{g,g})^n_{q-p}|^2[ X](\Gamma^p)_t[ Y]'_{R^{q-1}}G(1)^n_{R^{q-1}}1_{\{R^{q-1}\leq t\}}\\+o_p(b_n^{1/2}),
\end{align*}
and thus [A1$'$](i)-(ii), $(\ref{holder})$, [SC1], $(\ref{A4})$ and Lemma \ref{psi}(b) imply that
\begin{align}
&\sum_{p,q:p<q}|(\psi_{g,g})^n_{q-p}|^2[ X](\Gamma^p)_t[ Y]'_{R^{q-1}}|\Gamma^q|1_{\{R^{q-1}\leq t\}}\nonumber\\
=&b_n\sum_{p,q:p<q}|(\psi_{g,g})^n_{q-p}|^2[ X](\Gamma^p)_t[ Y]'_{R^{q-1}}G^n_{R^{q-1}}1_{\{R^{q-1}\leq t\}}+o_p(b_n^{1/2}).\label{HJYarg}
\end{align}

Now we show that
\begin{align}
&b_n\sum_{p,q:p<q}|(\psi_{g,g})^n_{q-p}|^2[ X](\Gamma^p)_t[ Y]'_{R^{q-1}}G^n_{R^{q-1}}1_{\{R^{q-1}\leq t\}}\nonumber\\
=&b_n\sum_{p=1}^{\infty}[ X](\Gamma^p)_t[ Y]'_{R^{p-1}}G^n_{R^{p-1}}\sum_{q:p<q}|(\psi_{g,g})^n_{q-p}|^21_{\{R^{q-1}\leq t\}}+o_p(b_n^{1/2}).\label{cadlagG}
\end{align}
We have $\lim_{\delta\downarrow0}\sup_{n\in\mathbb{N}}w'(F^n;\delta,T)=0$ a.s. for any $T>0$ by Theorem VI-1.5 of \cite{JS}, where $F^n=[Y]'G^n$. Therefore, for any $\eta>0$ we can take a positive (random) number $\delta$ such that a.s. $\sup_{n\in\mathbb{N}}w'(F^n;\delta,t)<\eta$. Then we can take (random) points $\xi_i^n$ such that $0=\xi_0^n<\xi_1^n<\cdots<\xi_{\bar{m}_n}^n=t$, $\inf_{i<\bar{m}_n}(\xi_i^n-\xi_{i-1}^n)\geq\delta$ and that $\max_{m\leq\bar{m}_n}w(F^n;[\xi^n_{m-1},\xi^n_m))<\eta$. Let $\Xi^n=\{\xi_m^n|m=1,\dots,\bar{m}_n\}$. Then
\begin{align*}
&b_n^{1/2}\left|\sum_{p=1}^{\infty}[ X](\Gamma^p)_t\sum_{q:p<q}|(\psi_{g,g})^n_{q-p}|^2\left( F^n_{R^{q-1}}-F^n_{R^{p-1}}\right)1_{\{R^{q-1}\leq t\}}\right|\\
\leq& b_n^{1/2}\sum_{p=1}^{\infty}[ X](\Gamma^p)_t\sum_{q:p<q}|(\psi_{g,g})^n_{q-p}|^2\cdot 2\eta
+b_n^{1/2}\sum_{p\in\mathbb{I}^n}[ X](\Gamma^p)_t\sum_{q:p<q}|(\psi_{g,g})^n_{q-p}|^2\cdot 2 (F^n)^*_t,
\end{align*}
where $\mathbb{I}^n=\{p\in\mathbb{N}|[R^p,R^{p+2k_n})\cap\Xi^n\neq\emptyset\}$ and $(F^n)^*_t=\sup_{n\in\mathbb{N}}\sup_{s\in[0,t]}|F^n_s|$, due to Lemma \ref{psi}(b). Hence there exits a positive constant $C$ such that
\begin{align*}
b_n^{1/2}\left|\sum_{p=1}^{\infty}[ X](\Gamma^p)_t\sum_{q:p<q}|(\psi_{g,g})^n_{q-p}|^2\left( F^n_{R^{q-1}}-F^n_{R^{p-1}}\right)1_{\{R^{q-1}\leq t\}}\right|
\leq C\left( t\cdot 2\eta+2 (F^n)^*_t\sum_{p\in\mathbb{I}^n}|\Gamma^p(t)|\right)
\end{align*}
by [SC1] and Lemma \ref{psi}(b). Now [A1$'$](i) implies that $\sup_{n\in\mathbb{N}}(F^n)^*_t<\infty$, while $\bar{m}_n<t/\delta+1$ because $\inf_{i<\bar{m}_n}(\xi_i^n-\xi_{i-1}^n)\geq\delta$. Moreover, for sufficiently large $n$ we have $\#\mathbb{I}^n\leq 4k_n\bar{m}_n$, and thus we obtain
\begin{align*}
\limsup_{n\to\infty}b_n^{1/2}\left|\sum_{p=1}^{\infty}[ X](\Gamma^p)_t\sum_{q:p<q}|(\psi_{g,g})^n_{q-p}|^2\left( F^n_{R^{q-1}}-F^n_{R^{p-1}}\right)1_{\{R^{q-1}\leq t\}}\right|
\leq 2Ct\eta
\end{align*}
by [A4]. Since $\eta$ is arbitrary, we conclude that Eq.~$(\ref{cadlagG})$ holds. On the other hand, since
\begin{align*}
&b_n\left|\sum_{p=1}^{\infty}[ X](\Gamma^p)_t[ Y]'_{R^{p-1}}G^n_{R^{p-1}}\sum_{q:p<q}|(\psi_{g,g})^n_{q-p}|^21_{\{R^{p-1}\leq t<R^{q-1}\}}\right|\\
\lesssim &b_n\sup_{0\leq s\leq t}G^n_s\sum_{p:R^{p-1}\leq t<R^{p+2k_n}}^{\infty}|\Gamma^p(t)|\sum_{q:p<q}|(\psi_{g,g})^n_{q-p}|^2
=o_p(\bar{r}_n)=o_p(b_n^{1/2}),
\end{align*}
we have
\begin{align*}
&b_n\sum_{p=1}^{\infty}[ X](\Gamma^p)_t[ Y]'_{R^{p-1}}G^n_{R^{p-1}}\sum_{q:p<q}|(\psi_{g,g})^n_{q-p}|^21_{\{R^{q-1}\leq t\}}\\
=&b_n\sum_{p=1}^{\infty}[ X](\Gamma^p)_t[ Y]'_{R^{p-1}}G^n_{R^{p-1}}\sum_{q:p<q}|(\psi_{g,g})^n_{q-p}|^21_{\{R^{p-1}\leq t\}}+o_p(b_n^{1/2}).
\end{align*}
Therefore, by an argument similar to the above we obtain
\begin{align}
&b_n\sum_{p=1}^{\infty}[ X](\Gamma^p)_t[ Y]'_{R^{p-1}}G^n_{R^{p-1}}\sum_{q:p<q}|(\psi_{g,g})^n_{q-p}|^21_{\{R^{q-1}\leq t\}}\nonumber\\
=&b_n^2\sum_{p=1}^{\infty}[ X]'_{R^{p-1}}[ Y]'_{R^{p-1}}\left|G^n_{R^{p-1}}\right|^21_{\{R^{p-1}\leq t\}}\sum_{q:p<q}|(\psi_{g,g})^n_{q-p}|^2+o_p(b_n^{1/2}).\label{shift}
\end{align}

Consequently, we obtain
\begin{align*}
&\sum_{p,q:p<q}|(\psi_{g,g})^n_{q-p}|^2[ X](\Gamma^p)_t[ Y](\Gamma^q)_t\\
=&b_n^2\sum_{p=1}^{\infty}[ X]'_{R^{p-1}}[ Y]'_{R^{p-1}}|G^n_{R^{p-1}}|^21_{\{R^{p-1}\leq t\}}\sum_{q:p<q}|(\psi_{g,g})^n_{q-p}|^2+o_p(b_n^{1/2})\\
=&b_n^{3/2}\theta\int_0^2\psi_{g,g}(x)^2\mathrm{d}x\sum_{p=1}^{\infty}[ X]'_{R^{p-1}}[ Y]'_{R^{p-1}}|G^n_{R^{p-1}}|^21_{\{R^{p-1}\leq t\}}+o_p(b_n^{1/2})\\
=&b_n^{1/2}\theta\int_0^2\psi_{g,g}(x)^2\mathrm{d}x\int_0^t[ X]'_{s}[ Y]'_{s}G_s\mathrm{d}s+o_p(b_n^{1/2})
\end{align*}
due to Lemma \ref{HJYlem2.2}. By symmetry we also obtain
\begin{align*}
\sum_{p,q:p>q}|(\psi_{g,g})^n_{q-p}|^2[ X](\Gamma^p)_t[ Y](\Gamma^q)_t
=b_n^{1/2}\theta\int_{-2}^0\psi_{g,g}(x)^2\mathrm{d}x\int_0^t[ X]'_{s}[ Y]'_{s}G_s\mathrm{d}s+o_p(b_n^{1/2}).
\end{align*}
After all, we complete the proof of (a).

(b) Similar to the proof of (a).

(c) Since $\widehat{S}^k<\widehat{T}^l$ and $\widehat{T}^k<\widehat{S}^l$ if $k<l$ and $\widehat{S}^k\vee\widehat{T}^k=R^k$ by Lemma \ref{advantage}, we can rewrite the target quantity as
\begin{align*}
&\frac{1}{k_n^4}\sum_{p,q=1}^{\infty}|(\psi_{g',g'})^n_{q-p}|^2\Psi^{11}_{\widehat{S}^p}\Psi^{22}_{\widehat{T}^q}1_{\{\widehat{S}^p\vee\widehat{T}^q\leq t\}}\\
=&\frac{1}{k_n^4}\sum_{q=2}^{\infty}\Psi^{22}_{\widehat{T}^q}1_{\{\widehat{T}^q\leq t\}}\sum_{p:p<q}|(\psi_{g',g'})^n_{q-p}|^2\Psi^{11}_{\widehat{S}^p}
+\frac{1}{k_n^4}\sum_{p=2}^{\infty}\Psi^{11}_{\widehat{S}^p}1_{\{\widehat{S}^p\leq t\}}\sum_{q:q<p}|(\psi_{g',g'})^n_{q-p}|^2\Psi^{22}_{\widehat{T}^q}\\
&+\frac{\psi_{g',g'}(0)^2}{k_n^4}\sum_{k=1}^{\infty}\Psi^{11}_{\widehat{S}^k}\Psi^{22}_{\widehat{T}^k}1_{\{R^k\leq t\}}.
\end{align*}
Note that $\Psi$ is c\`adl\`ag, by an argument similar to the above we obtain
\begin{align*}
\frac{1}{k_n^4}\sum_{q=2}^{\infty}\Psi^{22}_{\widehat{T}^q}1_{\{\widehat{T}^q\leq t\}}\sum_{p:p<q}|(\psi_{g',g'})^n_{q-p}|^2\Psi^{11}_{\widehat{S}^p}
=&\frac{1}{k_n^4}\sum_{q=2k_n+1}^{\infty}\Psi^{11}_{\widehat{T}^q}\Psi^{22}_{\widehat{T}^q}1_{\{\widehat{T}^q\leq t\}}\sum_{p=q-2k_n}^{q-1}|(\psi_{g',g'})^n_{q-p}|^2+o_p(b_n^{1/2})\\
=&\frac{1}{k_n^3}\sum_{q=2k_n+1}^{\infty}\Psi^{11}_{\widehat{T}^q}\Psi^{22}_{\widehat{T}^q}1_{\{\widehat{T}^q\leq t\}}\int_0^2\psi_{g',g'}(x)^2\mathrm{d}x+o_p(b_n^{1/2}).
\end{align*}
Note that $b_n\sum_{q=1}^{\infty}1_{\{\widehat{T}^q\leq t\}}=b_n\sum_{k=1}^{\infty}1_{\{R^k\leq t\}}+O_p(b_n)$, Lemma \ref{HJYlem2.2} yields
\begin{align*}
\frac{1}{k_n^4}\sum_{q=2}^{\infty}\Psi^{22}_{\widehat{T}^q}1_{\{\widehat{T}^q\leq t\}}\sum_{p:p<q}|(\psi_{g',g'})^n_{q-p}|^2\Psi^{11}_{\widehat{S}^p}
=b_n^{1/2}\theta^{-3}\int_0^2\psi_{g',g'}(x)^2\mathrm{d}x\int_0^t\Psi^{11}_s\Psi^{22}_s G_s^{-1}\mathrm{d}s+o_p(b_n^{1/2}).
\end{align*}
By symmetry we also obtain
\begin{align*}
\frac{1}{k_n^4}\sum_{p=2}^{\infty}\Psi^{11}_{\widehat{S}^p}1_{\{\widehat{S}^p\leq t\}}\sum_{q:q<p}|(\psi_{g',g'})^n_{q-p}|^2\Psi^{22}_{\widehat{T}^q}
=b_n^{1/2}\theta^{-3}\int_{-2}^0\psi_{g',g'}(x)^2\mathrm{d}x\int_0^t\Psi^{11}_s\Psi^{22}_s G_s^{-1}\mathrm{d}s+o_p(b_n^{1/2}).
\end{align*}
Since $\sum_{k=1}^{\infty}\Psi^{11}_{\widehat{S}^k}\Psi^{22}_{\widehat{T}^k}1_{\{R^k\leq t\}}=O_p(b_n^{-1})$, we conclude that
\begin{align*}
\frac{1}{k_n^4}\sum_{p,q=1}^{\infty}|(\psi_{g',g'})^n_{q-p}|^2\Psi^{11}_{\widehat{S}^p}\Psi^{22}_{\widehat{T}^q}1_{\{\widehat{S}^p\vee\widehat{T}^q\leq t\}}
=b_n^{1/2}\theta^{-3}\widetilde{\kappa}\int_0^t\Psi^{11}_s\Psi^{22}_s G_s^{-1}\mathrm{d}s+o_p(b_n^{1/2}).
\end{align*}

(d) First, Proposition \ref{advantage}(a), [SC2]-[SC3], Lemma \ref{psi}(b) and an argument similar to the proof of $(\ref{cadlagG})$ imply that
\begin{align*}
&\frac{1}{k_n^4}\sum_{p,q=1}^{\infty}|(\psi_{g',g'})^n_{q-p}|^2\Psi^{12}_{\widehat{S}^p}1_{\{\widehat{S}^p=\widehat{T}^p\leq t\}}\Psi^{21}_{\widehat{T}^q}1_{\{\widehat{S}^q=\widehat{T}^q\leq t\}}\\
=&\frac{1}{k_n^4}\sum_{p,q:p\neq q}|(\psi_{g',g'})^n_{q-p}|^2\Psi^{12}_{R^{p-1}}1_{\{\widehat{S}^p=\widehat{T}^p\}}\Psi^{21}_{R^{q-1}}1_{\{\widehat{S}^q=\widehat{T}^q\}}1_{\{R^{p\wedge q-1}\leq t\}}+o_p(b_n^{1/2}).
\end{align*}
On the other hand, note that $0\leq1_{\{\widehat{S}^k=\widehat{T}^k\}}\leq1$, by an argument similar to that in the proof of (a) we can show that
\begin{align*}
&\frac{1}{k_n^4}\sum_{q=2}^\infty\sum_{p:p<q}|(\psi_{g',g'})^n_{q-p}|^2\Psi^{12}_{R^{p-1}}1_{\{\widehat{S}^p=\widehat{T}^p\}}\Psi^{21}_{R^{q-1}}1_{\{\widehat{S}^q=\widehat{T}^q\}}1_{\{R^{p-1}\leq t\}}\\
=&\frac{1}{k_n^4}\sum_{q=2}^\infty\sum_{p:p<q}|(\psi_{g',g'})^n_{q-p}|^2\left(\Psi^{12}_{R^{p-1}}\chi^{\prime n}_{R^{p-1}}\right)^21_{\{R^{p-1}\leq t\}}
+o_p(b_n^{1/2})
\end{align*}
using [SC2] and [A1$'$](iii) instead of [SC1] and [A1$'$](i) respectively. By symmetry we also obtain
\begin{align*}
&\frac{1}{k_n^4}\sum_{p=2}^\infty\sum_{q:q<p}|(\psi_{g',g'})^n_{q-p}|^2\Psi^{12}_{R^{p-1}}1_{\{\widehat{S}^p=\widehat{T}^p\}}\Psi^{21}_{R^{q-1}}1_{\{\widehat{S}^q=\widehat{T}^q\}}1_{\{R^{q-1}\leq t\}}\\
=&\frac{1}{k_n^4}\sum_{p=2}^\infty\sum_{q:q<p}|(\psi_{g',g'})^n_{q-p}|^2\left(\Psi^{12}_{R^{q-1}}\chi^{\prime n}_{R^{q-1}}\right)^21_{\{R^{q-1}\leq t\}}
+o_p(b_n^{1/2}).
\end{align*}
Therefore, Lemma \ref{HJYlem2.2} yields desired result.

(e) By an argument similar to the proof of $(\ref{upsilon})$, we can show
\begin{align*}
\frac{1}{k_n^2}\sum_{p,q=1}^{\infty}|(\psi_{g,g'})^n_{q-p}|^2[ X](\widehat{I}^p)_t\Psi^{22}_{\widehat{T}^q}1_{\{\widehat{T}^q\leq t\}}
=\frac{1}{k_n^2}\sum_{p,q=1}^{\infty}|(\psi_{g,g'})^n_{q-p}|^2[ X](\Gamma^p)_t\Psi^{22}_{\widehat{T}^q}1_{\{\widehat{T}^q\leq t\}}+o_p(b_n^{1/2}).
\end{align*}
On the other hand, we have
\begin{align*}
\frac{1}{k_n^2}\left|\sum_{p,q=1}^{\infty}|(\psi_{g,g'})^n_{q-p}|^2[ X](\Gamma^p)_t\Psi^{22}_{\widehat{T}^q}1_{\{\widehat{T}^q> t\}}\right|
\leq \frac{1}{k_n^2}\sum_{p:R^{p-1}<t<R^{p+2k_n}}[ X](\Gamma^p)_t\sum_{q:|q-p|\leq 2k_n}|(\psi_{g,g'})^n_{q-p}|^2\Psi^{22}_{\widehat{T}^q}1_{\{\widehat{T}^q> t\}},
\end{align*} 
hence by [A4] and [SC1] we obtain
\begin{align*}
\frac{1}{k_n^2}\left|\sum_{p,q=1}^{\infty}|(\psi_{g,g'})^n_{q-p}|^2[ X](\Gamma^p)_t\Psi^{22}_{\widehat{T}^q}1_{\{\widehat{T}^q> t\}}\right|=O_p(\bar{r}_n).
\end{align*}
Therefore, we conclude that
\begin{align*}
\frac{1}{k_n^2}\sum_{p,q=1}^{\infty}|(\psi_{g,g'})^n_{q-p}|^2[ X](\widehat{I}^p)_t\Psi^{22}_{\widehat{T}^q}1_{\{\widehat{T}^q\leq t\}}
=\frac{1}{k_n^2}\sum_{p=1}^{\infty}[ X](\Gamma^p)_t\sum_{q=1}^{\infty}|(\psi_{g,g'})^n_{q-p}|^2\Psi^{22}_{\widehat{T}^q}+o_p(b_n^{1/2}),
\end{align*}
hence, note that $\Psi$ is c\`adl\`ag, an argument similar to the proof of $(\ref{cadlagG})$ yields
\begin{align*}
\frac{1}{k_n^2}\sum_{p,q=1}^{\infty}|(\psi_{g,g'})^n_{q-p}|^2[ X](\widehat{I}^p)_t\Psi^{22}_{\widehat{T}^q}1_{\{\widehat{T}^q\leq t\}}
=&\frac{1}{k_n^2}\sum_{p=2k_n+1}^{\infty}[X](\Gamma^p)_t\Psi^{22}_{R^{p-1}}\sum_{q=p-2k_n}^{p+2k_n}|(\psi_{g,g'})^n_{q-p}|^2+o_p(b_n^{1/2})\\
=&b_n^{1/2}\theta^{-1}\overline{\kappa}\int_0^t[ X]'_s\Psi^{22}_s\mathrm{d}s+o_p(b_n^{1/2}).
\end{align*}
This completes the proof of (e).

(f) Similar to the proof of (e).   

(g) An argument similar to that in the proof of (e) and the fact that $\Psi$ is c\`adl\`ag yield 
\begin{align*}
\frac{1}{k_n^2}\sum_{p,q=1}^{\infty}|(\psi_{g,g'})^n_{q-p}|^2[ X,Y](\widehat{I}^p)_t\Psi^{12}_{R^q}1_{\{\widehat{S}^q=\widehat{T}^q\leq t\}}
=&\frac{1}{k_n^2}\sum_{p,q=1}^{\infty}|(\psi_{g,g'})^n_{q-p}|^2[ X,Y](\widehat{I}^p)_t\Psi^{12}_{R^{q-1}}1_{\{\widehat{S}^q=\widehat{T}^q\}}+o_p(b_n^{1/2}).
\end{align*}
Therefore, we can show the desired result by an argument similar to the proof of (a).

(h) We decompose the target quantity as
\begin{align*}
\sum_{p,q=1}^{\infty}|(\psi_{g',g'})^n_{q-p}|^2[Z^X](\check{I}^p)_t [Z^Y](\check{J}^q)_t
=&\left\{\sum_{p,q:|p-q|\leq 1}+\sum_{p,q:p<q-1}+\sum_{p,q:q<p-1}\right\}|(\psi_{g',g'})^n_{q-p}|^2[Z^X](\check{I}^p)_t [Z^Y](\check{J}^q)_t\\
=:&\mathbb{B}_1+\mathbb{B}_2+\mathbb{B}_3.
\end{align*}
Evidently $\mathbb{B}_1=o_p(b_n^{1/2})$. On the other hand, an argument similar to the proof of $(\ref{localcons})$ yields
\begin{align*}
\mathbb{B}_2=\sum_{p,q:p<q-1}|(\psi_{g',g'})^n_{q-p}|^2[Z^X](\check{I}^p)_t [Z^Y]'_{R^{q-2}\wedge t}|\check{J}^q(t)|+o_p(b_n^{1/2}).
\end{align*}
Hence, note that $|\check{J}^k|$ is $\mathcal{H}^n_{\widehat{T}^k}$-measurable for every $k$ due to Lemma \ref{checktime}, by an argument similar to the proof of $(\ref{HJYarg})$ we obtain
\begin{align*}
\mathbb{B}_2=b_n\sum_{p,q:p<q-1}|(\psi_{g',g'})^n_{q-p}|^2[Z^X](\check{I}^p)_t [Z^Y]'_{R^{q-2}}F^2_{R^{q-2}}1_{\{R^{q-2}\leq t\}}+o_p(b_n^{1/2})
\end{align*}
and, by arguments similar to those in the proofs of $(\ref{cadlagG})$ and $(\ref{shift})$, we conclude that
\begin{align*}
\mathbb{B}_2=b_n^2\sum_{p=2}^\infty [Z^X]'_{R^{p-2}}[Z^Y]'_{R^{p-2}}F^1_{R^{p-2}}F^2_{R^{p-2}}1_{\{R^{p-2}\leq t\}}\sum_{q:p<q-1}|(\psi_{g',g'})^n_{q-p}|^2+o_p(b_n^{1/2}).
\end{align*}
Therefore, Lemma \ref{HJYlem2.2} yields
\begin{align*}
\mathbb{B}_2=b_n^{1/2}\theta\int_0^2\psi_{g',g'}(x)^2\mathrm{d}x\int_0^t[Z^X]'_s[Z^Y]'_s\frac{F^1_s F^2_s}{G_s}\mathrm{d}s+o_p(b_n^{1/2}).
\end{align*}
By symmetry we also obtain
\begin{align*}
\mathbb{B}_3=b_n^{1/2}\theta\int_{-2}^0\psi_{g',g'}(x)^2\mathrm{d}x\int_0^t[Z^X]'_s[Z^Y]'_s\frac{F^1_s F^2_s}{G_s}\mathrm{d}s+o_p(b_n^{1/2}).
\end{align*}
Note that $b_n^{-1}/k_n^2\to\theta^{-2}$ as $n\to\infty$, we complete the proof of (h).

(i) An argument similar to the proof of $(\ref{upsilon2})$ yields
\begin{align*}
\sum_{p,q=1}^{\infty}|(\psi_{g',g'})^n_{q-p}|^2[\mathfrak{Z}^X,\mathfrak{Z}^Y](\widehat{I}^p)_t [\mathfrak{Z}^X,\mathfrak{Z}^Y](\widehat{J}^q)_t
=\sum_{p,q=1}^{\infty}|(\psi_{g',g'})^n_{q-p}|^2[\mathfrak{Z}^X,\mathfrak{Z}^Y](\Gamma^{p})_t [\mathfrak{Z}^X,\mathfrak{Z}^Y](\Gamma^{q})_t+o_p(b_n^{1/2}).
\end{align*}
Since $[\mathfrak{Z}^X,\mathfrak{Z}^Y](\Gamma^{k})=(\check{I}^k_-\check{J}^k_-+\check{I}^{k+1}_-\check{J}^k_-+\check{I}^k_-\check{J}^{k+1}_-)\bullet[Z^X,Z^Y]$, note that $|\check{I}^k\cap\check{J}^k|+|\check{I}^{k+1}\cap\check{J}^k|+|\check{I}^k\cap\check{J}^{k+1}|$ is $\mathcal{H}^n_{R^k}$-measurable for every $k$ due to Lemma \ref{checktime}, by an argument similar to the proof of (a) we complete the proof of (i).

(j) An argument similar to the proof of (e) yields
\begin{align*}
\frac{1}{k_n^2}\sum_{p,q=1}^{\infty}|(\psi_{g',g'})^n_{q-p}|^2\Psi^{11}_{\widehat{S}^p}1_{\{\widehat{S}^p\leq t\}}[Z^Y](\check{J}^q)_t
=&\frac{1}{k_n^2}\sum_{q=2k_n+1}^{\infty}\Psi^{11}_{R^{q-2}}[Z^Y](\check{J}^q)_t\sum_{q=p-2k_n}^{p+2k_n}|(\psi_{g',g'})^n_{q-p}|^2+o_p(b_n^{1/2})\\
=&b_n^{1/2}\theta^{-1}\widetilde{\kappa}\sum_{q=1}^{\infty}\Psi^{11}_{R^{q-2}}[Z^Y](\check{J}^q)_t+o_p(b_n^{1/2}).
\end{align*}
Therefore, note that $\Psi^{11}$ is $\mathbf{H}^n$-adapted, combining Lemma \ref{HJYlem2.2} with arguments similar to those in the proof of $(\ref{localcons})$ and $(\ref{HJYarg})$, we conclude that
\begin{align*}
\frac{1}{k_n^2}\sum_{p,q=1}^{\infty}|(\psi_{g',g'})^n_{q-p}|^2\Psi^{11}_{\widehat{S}^p}1_{\{\widehat{S}^p\leq t\}}[Z^Y](\check{J}^q)_t
=b_n^{1/2}\theta^{-1}\widetilde{\kappa}\int_0^t\Psi^{11}_{s}[Z^Y]'_s\frac{F^2_s}{G_s}\mathrm{d}s+o_p(b_n^{1/2}),
\end{align*}
and thus we complete the proof of (j).

(k) Similar to the proof of (j).

(l) An argument similar to the proof of $(\ref{upsilon})$ yields
\begin{align*}
\sum_{p,q=1}^{\infty}|(\psi_{g,g'})^n_{q-p}|^2[X](\widehat{I}^p)_t[Z^Y](\check{J}^q)_t
=\sum_{p,q=1}^{\infty}|(\psi_{g,g'})^n_{q-p}|^2[X](\Gamma^p)_t[Z^Y](\check{J}^q)_t+o_p(b_n^{1/2}).
\end{align*}
Hence, the desired result can be shown in a similar manner to the proof of (h).

(m) Similar to the proof of (l).

(n)-(r) Similar to the proof of (h) (note that $\int_{-2}^2\psi_{g',g'}(x)\psi_{g,g}(x)\mathrm{d}x=\overline{\kappa}$ due to integration by parts and Lemma \ref{psi}).

(s) Similar to the proof of (i).
\end{proof}


\section{Proof of Proposition \ref{HYprop5.1}}\label{proofHYprop5.1}



First note that an argument similar to the one in the first part of Section 12 of \cite{HY2011} allows us to assume that $(\ref{SA4})$, $\frac{9}{10}<\xi<\xi'<1$ and $\check{S}^k$, $\check{T}^k$ are $\mathbf{G}^{(n)}$-stopping times for every $k$ under [A2] and [A4] (note that $\hat{S}^k$, $\hat{T}^k$ and $R^k$ automatically become $\mathbf{G}^{(n)}$-stopping times under [A2]). Furthermore, in the following we only consider sufficiently large $n$ such that
\begin{equation}\label{spp}
k_n\bar{r}_n<b_n^{\xi-1/2}.
\end{equation}


\begin{lem}\label{HYlem12.2}
Suppose that $[\mathrm{A}2]$ and $[\mathrm{SA}4]$ are satisfied. Let $i,j\in\mathbb{Z}_+$ and let $\tau$ be a $\mathbf{G}^{(n)}$-stopping time. Then for any $A\in\mathcal{G}^{(n)}_\tau$ we have
\begin{align*}
A\cap \left\{\tau\leq R^{\vee}(i+k_n,j+k_n)\right\}\cap\left\{\bar{I}^i(\tau)\cap\bar{J}^j(\tau)\neq\emptyset\right\}
\in\mathcal{F}_{R^\wedge(i,j)}.
\end{align*}
\end{lem}

\begin{proof}
Let
\begin{align*}
B=\left\{\tau\leq R^{\vee}(i+k_n,j+k_n)\right\},\qquad
C=\left\{\bar{I}^i(\tau)\cap\bar{J}^j(\tau)\neq\emptyset\right\}.
\end{align*}
It is sufficient to show that $A\cap B\cap C\cap D\in\mathcal{F}_u$ for any $u\in\mathbb{R}_+$, where $D=\left\{ R_\wedge(i,j)\leq u\right\}$. On $C$ we have $$R^{\vee}(i+k_n,j+k_n)-R^{\wedge}(i,j)\leq|\bar{I}^i|\vee|\bar{J}^j|\vee(\hat{S}^{i+k_n}-\hat{T}^j)\vee(\hat{T}^{j+k_n}-\hat{S}^i)\leq k_n\bar{r}_n,$$ hence
\begin{align*}
R^{\vee}(i+k_n,j+k_n)
=\{ R^{\vee}(i+k_n,j+k_n)-R^{\wedge}(i ,j)\}
+R^{\wedge}(i,j)
\leq R^{\wedge}(i,j)+k_n\bar{r}_n,
\end{align*}
and thus we have
\begin{align*}
C\cap D=C\cap D\cap\{ R^{\vee}(i+k_n,j+k_n)\leq u+k_n\bar{r}_n\}.
\end{align*}
Since $A\in\mathcal{G}^{(n)}_{\tau}$ and $C\in\mathcal{G}^{(n)}_{\tau}$, we have $(A\cap C)\cap B\in\mathcal{G}^{(n)}_{R^\vee(i+k_n,j+k_n)}$. Therefore, we obtain
\begin{align*}
A\cap B\cap C\cap\{ R^{\vee}(i+k_n,j+k_n)\leq u+k_n\bar{r}_n\}\in\mathcal{G}_{u+k_n\bar{r}_n},
\end{align*}
however, $\mathcal{G}^{(n)}_{u+k_n\bar{r}_n}=\mathcal{F}_{(u+k_n\bar{r}_n-b_n^{\xi-\frac{1}{2}})_+}\subset\mathcal{F}_u$ by $(\ref{spp})$. This together with the fact that $\{R^\wedge(i,j)\leq u\}\in\mathcal{F}_u$ implies $A\cap B\cap C\cap D\in\mathcal{F}_u$.
\end{proof}


\begin{lem}\label{HYlem12.3}
Suppose that $[\mathrm{A}2]$ and $[\mathrm{SA}4]$ are satisfied. Let $Z=(Z_t)_{t\in\mathbb{R}_+}$ be a $\mathbf{G}^{(n)}$-adapted process. Let $i,j\in\mathbb{Z}_+$, $p,q\in\{0,1,\dots,k_n-1\}$ and let $\tau$ be a $\mathbf{G}^{(n)}$-stopping time. Then both  $\bar{K}^{ij}_{\tau-}\widehat{I}^{i+p}_{\tau-} Z_\tau$ and $\bar{K}^{ij}_{\tau-}\widehat{J}^{j+q}_{\tau-} Z_\tau$ are $\mathcal{F}_{R^\wedge(i,j)}$-measurable.
\end{lem}

\begin{proof}
On $\{\bar{I}^{i}(\tau)\cap\bar{J}^{j}(\tau)=\emptyset\}\cup\left\{\tau> R^{\vee}(i+k_n,j+k_n)\right\}$ we have $\bar{K}^{ij}_{\tau-}\widehat{I}^{i+p}_{\tau-}=0$. Therefore, for any Borel measurable set $B$ we have
\begin{align*}
\{\bar{K}^{ij}_{\tau-}\widehat{I}^{i+p}_{\tau-} Z_\tau\in B\}
=&\left[\{0\in B\}\cap\left(\{\bar{I}^{i}(\tau)\cap\bar{J}^{j}(\tau)=\emptyset\}\cup\left\{\tau> R^{\vee}(i+k_n,j+k_n)\right\}\right)\right]\\
&\cap\left[\{\bar{K}^{ij}_{\tau-}\widehat{I}^{i+p}_{\tau-} Z_\tau\in B\}\cap\{\bar{I}^{i}(\tau)\cap\bar{J}^{j}(\tau)\neq\emptyset\}\cap\left\{\tau\leq R^{\vee}(i+k_n,j+k_n)\right\}\right],
\end{align*}
so we obtain $\{\bar{K}^{ij}_{\tau-}\widehat{I}^{i+p}_{\tau-} Z_\tau\in B\}\in\mathcal{F}_{R^\wedge(i,j)}$ by Lemma \ref{HYlem12.2} because $\bar{K}^{ij}_{\tau-}\widehat{I}^{i+p}_{\tau-} Z_\tau$ is $\mathcal{G}_\tau$-measurable by construction. By symmetry we can also show that $\bar{K}^{ij}_{\tau-}\widehat{J}^{j+q}_{\tau-} Z_\tau$ is $\mathcal{G}_\tau$-measurable.
\end{proof}


In the remainder of this section, we fix $\alpha,\beta,\alpha',\beta'\in\Phi$. Let $\Xi^{iji'j'}_t:=\sum_{q,q'=0}^{k_n-1}\beta^n_{q}\beta^{\prime n}_{q'}\bar{K}^{ij}_t\bar{K}^{i'j'}_t\hat{J}^{j+q}_t\hat{J}^{j'+q'}_t$ and $\Lambda^{iji'j'}_t=\sum_{p',q=0}^{k_n-1}\alpha^{\prime n}_p \beta^n_{q}\bar{K}^{ij}_t\bar{K}^{i'j'}_t\hat{I}^{i'+p'}_t\hat{J}^{j+q}_t$ for each $t\in\mathbb{R}_+$. 

We introduce an auxiliary condition.
\begin{enumerate}
\item[{[H]}] For each $n\in\mathbb{N}$ we have four square-integrable martingales $M^n$, $N^n$, $M^{\prime n}$ and $N^{\prime n}$ satisfying the following conditions:
\begin{enumerate}[(i)]
\item There exists a positive constant $C$ such that
\begin{equation*}
E\left[[M^n](\hat{I}^i)_t^4|\mathcal{F}_{\hat{S}^{i-1}}\right]+E\left[[M^{\prime n}](\hat{I}^i)_t^4|\mathcal{F}_{\hat{S}^{i-1}}\right]+E\left[[N^n](\hat{J}^j)_t^4|\mathcal{F}_{\hat{T}^{j-1}}\right]+E\left[[N^{\prime n}](\hat{J}^j)_t^4|\mathcal{F}_{\hat{T}^{j-1}}\right]\leq C\bar{r}_n^4
\end{equation*}
for any $t\in\mathbb{R}_+$, $n\in\mathbb{N}$ and $i,j\in\mathbb{N}$.
\item For each $n\in\mathbb{N}$ we have
\begin{gather*}
\langle M^n,M^{\prime n}\rangle=H(1)D(1)^n\bullet B(1)^n,\qquad
\langle N^n,N^{\prime n}\rangle=H(2)D(2)^n\bullet B(2)^n,\\
\langle M^n,N^{\prime n}\rangle=H(3)D(3)^n\bullet B(3)^n,\qquad
\langle M^{\prime n},N^{n}\rangle=H(4)D(4)^n\bullet B(4)^n,
\end{gather*}
where $H(k)$ is a c\'{a}dl\'{a}g bounded $\mathbf{F}^{(0)}$-adapted process, $D(k)^n$ is a bounded $\mathbf{G}^{(n)}$-adapted process and $B(k)^n$ is a deterministic nondecreasing process or a $\mathbf{G}^{(n)}$-adapted point process for each $k\in\{1,2,3,4\}$.
\item For any $\lambda>0$, we have a positive constant $K_\lambda$ satisfying $$\max_{k\in\{1,2,3,4\}}\sup_{0\leq s\leq t}E\left[|H(k)_{s}-H(k)_{(s-h)_+}|^2\big|\mathcal{F}_{(s-h)_+}\right]\leq K_\lambda h^{1-\lambda}$$ for any $t,h>0$.
\item For each $T\in\mathbb{R}_+$ there is a positive constant $C_T$ such that $\max_{k\in\{1,2,3,4\}}B(k)^n_T\leq C_T$ for all $n$. Moreover, there exits a positive constant $C$ such that
\begin{equation*}
B(1)^n(\hat{I}^p)_t\vee B(2)^n(\hat{J}^q)_t\vee B(3)^n(\hat{I}^p)_t\vee B(3)^n(\hat{J}^q)_t\vee B(4)^n(\hat{I}^p)_t\vee B(4)^n(\hat{J}^q)_t\leq C\bar{r}_n
\end{equation*} 
for any $p,q,n\in\mathbb{N}$ and any $t>0$.
\end{enumerate}
\end{enumerate}

For simplicity of notation, set
\begin{align*}
\bar{M}^{ii'}:=\bar{M}^n_\alpha(\hat{\mathcal{I}})^i_t\bar{M}^{\prime n}_{\alpha'}(\hat{\mathcal{I}})^{i'}_t-\langle\bar{M}^n_\alpha(\hat{\mathcal{I}})^i, \bar{M}^{\prime n}_{\alpha'}(\hat{\mathcal{I}})^{i'}\rangle_t,\qquad
\bar{L}^{ij}:=\bar{M}^n_\alpha(\hat{\mathcal{I}})^i_t\bar{N}^{\prime n}_{\beta'}(\hat{\mathcal{J}})^{j'}_t-\langle\bar{M}^n_{\alpha}(\hat{\mathcal{I}})^i, \bar{N}^{\prime n}_{\beta'}(\hat{\mathcal{J}})^{j'}\rangle_t
\end{align*}
for each $i,i',j\in\mathbb{N}$ when we assume the condition [H]. In addition, for an $\mathbf{F}$-adapted process $Z$, we write $\widetilde{Z}_t:=Z_{(t-b_n^{\xi-1/2})_+}$. Then $\widetilde{Z}_t$ is clearly $\mathbf{G}^{(n)}$-adapted.

\begin{lem}\label{HYlem12.4}
Suppose that $[\mathrm{A}2]$, $[\mathrm{SA}4]$ and $[\mathrm{H}]$ hold. Let $r,s\in\mathbb{R}_+$ and $i,j,i',j',k,l,k',l'\in\mathbb{Z}_+$. 
\begin{enumerate}[\normalfont (a)]
\item For $|i-k|\geq k_n-1$, $|i'-k|\geq k_n-1$, $|i-k'|\geq k_n-1$ and $|i'-k'|\geq k_n-1$,
\begin{align*}
E\left[\int_0^t\int_0^t\Xi^{iji'j'}_{s-}\Xi^{klk'l'}_{r-}\bar{M}^{ii'}_{s-} \bar{M}^{kk'}_{r-}\widetilde{H(2)}_sD(2)^n_s\widetilde{H(2)}_rD(2)^n_r \mathrm{d}B(2)^n_s\mathrm{d}B(2)^n_r\right]=0.
\end{align*}
\item For $|i-k|\geq k_n-1$, $|i-l'|\geq k_n-1$, $|j'-k|\geq k_n-1$ and $|j'-l'|\geq k_n-1$,
\begin{align*}
E\left[\int_0^t\int_0^t\Lambda^{iji'j'}_{s-}\Lambda^{klk'l'}_{r-}\bar{L}^{ij'}_{s-} \bar{L}^{kl'}_{r-}\widetilde{H(4)}_sD(4)^n_s\widetilde{H(4)}_rD(4)^n_r \mathrm{d}B(4)^n_s\mathrm{d}B(4)^n_r\right]=0.
\end{align*} 
\end{enumerate}
\end{lem}

\begin{proof}
(a) Without loss of generality, we may assume that $i\geq i'\vee k\vee k'$, so we have $i\geq k\vee k'+k_n-1$. Since $B(2)^n$ is a deterministic nondecreasing process or a $\mathbf{G}^{(n)}$-adapted point process, it is sufficient to show that
\begin{equation}
E\left[\Xi^{iji'j'}_{\sigma-}\Xi^{klk'l'}_{\tau-}\bar{M}^{ii'}_{\sigma-} \bar{M}^{kk'}_{\tau-}\widetilde{H(2)}_\sigma D(2)^n_\sigma\widetilde{H(2)}_\tau D(2)^n_\tau 1_{\{\sigma\vee\tau\leq t\}}\right]=0
\end{equation}
for any bounded $\mathbf{G}^{(n)}$-stopping times $\sigma,\tau$. Lemma \ref{HYlem12.3} implies that $\Xi^{iji'j'}_{\sigma-}\widetilde{H(2)}_\sigma D(2)^n_\sigma1_{\{\sigma\leq t\}}$ is $\mathcal{F}_{R^\wedge(i',j\wedge j')}$-measurable and $\Xi^{klk'l'}_{\tau-}\widetilde{H(2)}_\tau D(2)^n_\tau1_{\{\tau\leq t\}}$ is $\mathcal{F}_{R^\wedge(k\wedge k',l\wedge l')}$-measurable. Moreover, $\bar{M}^{kk'}_r$ is $\mathcal{F}_{\hat{S}^{k\vee k'+k_n-1}}$-measurable by definition. Since $R^\wedge(i,j\wedge j')\leq\hat{S}^i$ and $R^\wedge(k\wedge k',l\wedge l')\leq\hat{S}^k\leq\hat{S}^i$, we obtain
\begin{align*}
&E\left[\Xi^{iji'j'}_{\sigma-}\Xi^{klk'l'}_{\tau-}\bar{M}^{ii'}_{\sigma-} \bar{M}^{kk'}_{\tau-}\widetilde{H(2)}_\sigma D(2)^n_\sigma\widetilde{H(2)}_\tau D(2)^n_\tau1_{\{\sigma\vee\tau\leq t\}}\right]\\
=&E\left[\Xi^{iji'j'}_{\sigma-}\Xi^{klk'l'}_{\tau-}\bar{M}^{kk'}_{\tau-}\widetilde{H(2)}_\sigma D(2)^n_\sigma\widetilde{H(2)}_\tau D(2)^n_\tau1_{\{\sigma\vee\tau\leq t\}} E\left[\bar{M}^{ii'}_{\sigma-} |\mathcal{F}_{\hat{S}^i}\right]\right].
\end{align*}
Since $\bar{M}^{ii'}$ is a martingale by the definition, the optional sampling theorem provides
\begin{align*}
E\left[\bar{M}^{ii'}_{\sigma-} |\mathcal{F}_{\hat{S}^i}\right]=\lim_{\delta\downarrow 0}E\left[\bar{M}^{ii'}_{(\sigma-\delta)_+} |\mathcal{F}_{\hat{S}^i}\right]
=\lim_{\delta\downarrow 0}\bar{M}^{ii'}_{\hat{S}^i\wedge(\sigma-\delta)_+}=0,
\end{align*}
which concludes the proof of (a).

(b) Similar to the proof of (a). 
\end{proof}

\begin{lem}\label{estXi}
Suppose that $[\mathrm{A}2]$, $[\mathrm{SA}4]$ and $[\mathrm{H}]$ are satisfied. Let $t\in\mathbb{R}_+$. Then:
\begin{enumerate}[\normalfont (a)]
\item There is a positive constant $C$ such that
\begin{align*}
\sum_{i,i',j,j'}\int_0^t\Xi^{iji'j'}_{s-}\mathrm{d}B(2)^n_s\leq Ck_n^4,\qquad
\sum_{i',j,j'}\int_0^t\Xi^{iji'j'}_s\mathrm{d}B(2)^n_s\leq C\bar{r}_nk_n^4.
\end{align*}
\item There is a positive constant $C$ such that
\begin{align*}
\sum_{i,i',j,j'}\int_0^t\Lambda^{iji'j'}_{s-}\mathrm{d}B(4)^n_s\leq Ck_n^4,\qquad
\sum_{i',j,j'}\int_0^t\Lambda^{iji'j'}_{s-}\mathrm{d}B(4)^n_s\leq C\bar{r}_nk_n^4.
\end{align*}
\end{enumerate} 
\end{lem}

\begin{proof}
Since $\sum_i\bar{K}^{ij}_{s-}\lesssim k_n$, we have $\sum_{i,i',j,j'}\Xi^{iji'j'}_{s-}\lesssim k_n^2\sum_{j,j'}\sum_{q,q'=0}^{k_n-1}\hat{J}^{j+q}_{s-}\hat{J}^{j'+q'}_{s-}$. Therefore, we obtain
\begin{align*}
\sum_{i,i',j,j'}\int_0^t\Xi^{iji'j'}_{s-}\mathrm{d}B(2)^n_s
\lesssim k_n^2\sum_{q,q'=0}^{k_n-1}\int_0^t\sum_{j,j'}\hat{J}^{j+q}_{s-}\hat{J}^{j'+q'}_{s-}\mathrm{d}B(2)^n_s
\leq k_n^2\sum_{q,q'=0}^{k_n-1}B(2)^n_t\lesssim k_n^4
\end{align*}
by [H](iv).

On the other hand, since $$\sum_{i',j,j'}\Xi^{iji'j'}_{s-}\lesssim\sum_{q,q'=0}^{k_n-1}\sum_{j,j'}\bar{K}^{ij}_{s-}\hat{J}^{j+q}_{s-}\hat{J}^{j'+q'}_{s-}\sum_{i'}\bar{K}^{i'j'}_{s-}\lesssim k_n\sum_{q,q'=0}^{k_n-1}\sum_{j,j'}\bar{K}^{ij}_{s-}\hat{J}^{j+q}_{s-}\hat{J}^{j'+q'}_{s-},$$
we have
\begin{align*}
\sum_{i',j,j'}\int_0^t\Xi^{iji'j'}_s\mathrm{d}B(2)^n_s
\lesssim k_n\sum_{q,q'=0}^{k_n-1}\sum_{j,j'}\int_0^t\bar{K}^{ij}_{s-}\hat{J}^{j+q}_{s-}\hat{J}^{j'+q'}_{s-}\mathrm{d}B(2)^n_s
\leq k_n^2\sum_{q=0}^{k_n-1}\sum_{j}\bar{K}^{ij}_{t}B(2)^n(\hat{J}^{j+q})_t
\end{align*}
because $\sum_{j'}\hat{J}^{j+q}_{s-}\hat{J}^{j'+q'}_{s-}\leq\hat{J}^{j+q}_{s-}$ and $\bar{K}^{ij}_{s-}\leq\bar{K}^{ij}_t$ for $s\in[0,t]$. Therefore, we obtain $\sum_{i',j,j'}\int_0^t\Xi^{iji'j'}_s\mathrm{d}B(2)^n_s\lesssim\bar{r}_nk_n^4$ due to [H](iv). Consequently, we complete the proof of (a). Similarly we can also prove (b).
\end{proof}


\begin{lem}\label{HYlem12.5}
Suppose that $[\mathrm{A}2]$, $[\mathrm{SA}4]$ and $[\mathrm{H}]$ are satisfied. Let $r\in[2,4]$.
\begin{enumerate}[\normalfont (a)]
\item There exists a positive constant $C_r$ such that $E\left[|\bar{M}^{ii'}|^r_t\big|\mathcal{F}_{\widehat{S}^{i\wedge i'}}\right]\leq C_r(k_n\bar{r}_n)^r$ for any $t\in\mathbb{R}_+$ and any $i,i'\in\mathbb{N}$.
\item There exists a positive constant $C'_r$ such that $E\left[|\bar{L}^{ij}|^r_t\big|\mathcal{F}_{R^{\wedge}(i,j)}\right]\leq C'_r(k_n\bar{r}_n)^r$ for any $t\in\mathbb{R}_+$ and any $i,j\in\mathbb{N}$. 
\end{enumerate}
\end{lem}

\begin{proof}

(a) First, since $\langle \bar{M}^n_\alpha(\mathcal{I})^i\rangle_t=\sum_{p=0}^{k_n-1}(\alpha^n_p)^2\langle M^n\rangle(\widehat{I}^{i+p})_t$, we have $\langle \bar{M}^n_\alpha(\mathcal{I})^i\rangle_t\lesssim k_n\bar{r}_n$. Similarly we can show $\langle \bar{M}^{\prime n}_{\alpha'}(\mathcal{I})^{i'}\rangle_t\lesssim k_n\bar{r}_n$. Therefore, by the Kunita-Watanabe inequality we obtain $$E\left[\left|\langle \bar{M}^n_\alpha(\mathcal{I})^i,\bar{M}^{\prime n}_{\alpha'}(\mathcal{I})^{i'}\rangle_t\right|^r\big|\mathcal{F}_{\widehat{S}^{i\wedge i'}}\right]\lesssim(k_n\bar{r}_n)^r.$$

Next, the Burkholder-Davis-Gundy inequality and [H](i) yield
\begin{align*}
E\left[\left|\bar{M}^n_\alpha(\mathcal{I})^i\right|^{2r}\big|\mathcal{F}_{\widehat{S}^{i}}\right]
\lesssim &E\left[\left\{\sum_{p=0}^{k_n-1}(\alpha^n_p)^2[M^n](\widehat{I}^{i+p})_t\right\}^{r}\big|\mathcal{F}_{\widehat{S}^{i}}\right]
\leq k_n^{r-1}\sum_{p=0}^{k_n-1}(\alpha^n_p)^2 E\left[[M^n](\widehat{I}^{i+p})_t^r\big|\mathcal{F}_{\widehat{S}^{i}}\right]\\
\lesssim &(k_n\bar{r}_n)^r.
\end{align*}
Similarly we can also show $E\left[\left|\bar{M}^{\prime n}_{\alpha'}(\mathcal{I})^{i'}\right|^{2r}\big|\mathcal{F}_{\widehat{S}^{i'}}\right]\lesssim(k_n\bar{r}_n)^r$. Therefore, by the Schwarz inequality we obtain
\begin{align*}
E\left[\left|\bar{M}^n_\alpha(\mathcal{I})^i\bar{M}^{\prime n}_{\alpha'}(\mathcal{I})^{i'}\right|^{r}\big|\mathcal{F}_{\widehat{S}^{i'}}\right]\lesssim(k_n\bar{r}_n)^r.
\end{align*}
Consequently, we complete the proof of (a). 

(b) Similar to the proof of (a).
\end{proof}

\begin{lem}\label{HYlem12.6and12.8}
Suppose $[\mathrm{A}2]$, $[\mathrm{SA}4]$ and $[\mathrm{H}]$. Then we have
\begin{align*}
&b_n^{-1/2}\sum_{i,j,i',j'}(\bar{K}^{ij}_-\bar{K}^{i'j'}_-)\bullet\langle\bar{L}_{\alpha,\beta}^{ij}(M^n,N^n),\bar{L}_{\alpha',\beta'}^{i'j'}(M^{\prime n},N^{\prime n})\rangle_t\\
=&b_n^{-1/2}\sum_{i,j,i',j'}(\bar{K}^{ij}_-\bar{K}^{i'j'}_-)\bullet V^{iji'j'}_{\alpha,\beta;\alpha',\beta'}(M^n,N^n;M^{\prime n},N^{\prime n})_t+o_p\left( k_n^4\right)
\end{align*}
as $n\to\infty$ for every $t\in\mathbb{R}_+$.
\end{lem}

\begin{proof}
We decompose the target quantity as
\begin{align*}
&\sum_{i,j,i',j'}(\bar{K}^{ij}_-\bar{K}^{i'j'}_-)\bullet\langle\bar{L}_{\alpha,\beta}^{ij}(M^n,N^n),\bar{L}_{\alpha',\beta'}^{i'j'}(M^{\prime n},N^{\prime n})\rangle_t-\sum_{i,j,i',j'}(\bar{K}^{ij}_-\bar{K}^{i'j'}_-)\bullet V^{iji'j'}_{\alpha,\beta;\alpha',\beta'}(M^n,N^n;M^{\prime n},N^{\prime n})_t\\
=&\Delta_{1,t}+\Delta_{2,t}+\Delta_{3,t}+\Delta_{4,t},
\end{align*}
where
\begin{align*}
\Delta_{1,t}=&\sum_{i,j,i',j'}(\bar{K}^{ij}_-\bar{K}^{i'j'}_-)\bullet(\{\bar{M}^n_\alpha(\hat{\mathcal{I}})^i_-\bar{M}^{\prime n}_{\alpha'}(\hat{\mathcal{I}})^{i'}_-\}\bullet\langle\bar{N}^n_\beta(\hat{\mathcal{J}})^j,\bar{N}^{\prime n}_{\beta'}(\hat{\mathcal{J}})^{j'}\rangle)_t\\
&-\sum_{i,j,i',j'}(\bar{K}^{ij}_-\bar{K}^{i'j'}_-)\bullet(\langle\bar{M}^n_\alpha(\hat{\mathcal{I}})^i,\bar{M}^{\prime n}_{\alpha'}(\hat{\mathcal{I}})^{i'}\rangle_-\bullet\langle\bar{N}^n_\beta(\hat{\mathcal{J}})^j,\bar{N}^{\prime n}_{\beta'}(\hat{\mathcal{J}})^{j'}\rangle)_t,\\
\Delta_{2,t}=&\sum_{i,j,i',j'}(\bar{K}^{ij}_-\bar{K}^{i'j'}_-)\bullet(\{\bar{N}^n_\beta(\hat{\mathcal{J}})^j_-\bar{N}^{\prime n}_{\beta'}(\hat{\mathcal{J}})^{j'}_-\}\bullet\langle\bar{M}^n_\alpha(\hat{\mathcal{I}})^i,\bar{M}^{\prime n}_{\alpha'}(\hat{\mathcal{I}})^{i'}\rangle)_t\\
&-\sum_{i,j,i',j'}(\bar{K}^{ij}_-\bar{K}^{i'j'}_-)\bullet(\langle\bar{N}^n_\beta(\hat{\mathcal{J}})^j,\bar{N}^{\prime n}_{\beta'}(\hat{\mathcal{J}})^{j'}\rangle_-\bullet\langle\bar{M}^n_\alpha(\hat{\mathcal{I}})^i,\bar{M}^{\prime n}_{\alpha'}(\hat{\mathcal{I}})^{i'}\rangle)_t
\end{align*}
and
\begin{align*}
\Delta_{3,t}=&\sum_{i,j,i',j'}(\bar{K}^{ij}_-\bar{K}^{i'j'}_-)\bullet(\{\bar{M}^n_\alpha(\hat{\mathcal{I}})^i_-\bar{N}^{\prime n}_{\beta'}(\hat{\mathcal{J}})^{j'}_-\}\bullet\langle\bar{N}^n_\beta(\hat{\mathcal{J}})^j,\bar{M}^{\prime n}_{\alpha'}(\hat{\mathcal{I}})^{i'}\rangle)_t\\
&-\sum_{i,j,i',j'}(\bar{K}^{ij}_-\bar{K}^{i'j'}_-)\bullet(\langle\bar{M}^n_\alpha(\hat{\mathcal{I}})^i,\bar{N}^{\prime n}_{\beta'}(\hat{\mathcal{J}})^{j'}\rangle_-\bullet\langle\bar{N}^n_\beta(\hat{\mathcal{J}})^j,\bar{M}^{\prime n}_{\alpha'}(\hat{\mathcal{I}})^{i'}\rangle)_t,\\
\Delta_{4,t}=&\sum_{i,j,i',j'}(\bar{K}^{ij}_-\bar{K}^{i'j'}_-)\bullet(\{\bar{N}^n_\beta(\hat{\mathcal{J}})^j_-\bar{M}^{\prime n}_{\alpha'}(\hat{\mathcal{I}})^{i'}_-\}\bullet\langle\bar{M}^n_\alpha(\hat{\mathcal{I}})^i,\bar{N}^{\prime n}_{\beta'}(\hat{\mathcal{J}})^{j'}\rangle)_t\\
&-\sum_{i,j,i',j'}(\bar{K}^{ij}_-\bar{K}^{i'j'}_-)\bullet(\langle\bar{N}^n_\beta(\hat{\mathcal{J}})^j,\bar{M}^{\prime n}_{\alpha'}(\hat{\mathcal{I}})^{i'}\rangle_-\bullet\langle\bar{M}^n_\alpha(\hat{\mathcal{I}})^i,\bar{N}^{\prime n}_{\beta'}(\hat{\mathcal{J}})^{j'}\rangle)_t.
\end{align*}

Consider $\Delta_{1,t}$ first. By the use of associativity and linearity of integration, we can rewrite $\Delta_{1,t}$ as
\begin{align*}
\Delta_{1,t}=\sum_{i,j,i',j'}\left\{(\bar{K}^{ij}_-\bar{K}^{i'j'}_-)\bar{M}^{ii'}_-\right\}\bullet\langle\bar{N}^n_\beta(\hat{\mathcal{J}})^j,\bar{N}^{\prime n}_{\beta'}(\hat{\mathcal{J}})^{j'}\rangle_t.
\end{align*}
Moreover, we have
\begin{align*}
\langle\bar{N}^n_\beta(\hat{\mathcal{J}})^j,\bar{N}^{\prime n}_{\beta'}(\hat{\mathcal{J}})^{j'}\rangle_t
=&\sum_{q,q'=0}^{k_n-1}\beta^n_q \beta^{\prime n}_{q'}(\hat{J}^{j+q}_-\hat{J}^{j'+q'}_-)\bullet\langle N^n,N^{\prime n}\rangle_t\\
=&\sum_{q,q'=0}^{k_n-1}\beta^n_q \beta^{\prime n}_{q'}(\hat{J}^{j+q}_-\hat{J}^{j'+q'}_-H(2)D(2)^n)\bullet B(2)^n_t,
\end{align*}
hence we obtain
\begin{align*}
\Delta_{1,t}=\sum_{i,j,i',j'}\int_0^t\Xi^{iji'j'}_{s-}\bar{M}^{ii'}_{s-} H(2)_{s} D(2)^n_{s}\mathrm{d}B(2)^n_s.
\end{align*}

Let $\mathcal{R}_t:=H(2)_t-\widetilde{H(2)}_t$. Then we have $(\Delta_{1,t})^2=\mathbf{I}+\mathbf{II}+\mathbf{III}+\mathbf{IV}$, where
\begin{align*}
&\mathbf{I}=\sum_{i,i',j,j'}\sum_{k,k',l,l'}\int_0^t\int_0^t\Xi^{iji'j'}_{s-}\Xi^{klk'l'}_{r-}\bar{M}^{ii'}_{s-} \bar{M}^{kk'}_{r-}\widetilde{H(2)}_sD(2)^n_s\widetilde{H(2)}_rD(2)^n_r \mathrm{d}B(2)^n_s\mathrm{d}B(2)^n_r,\\
&\mathbf{II}=\sum_{i,i',j,j'}\sum_{k,k',l,l'}\int_0^t\int_0^t\Xi^{iji'j'}_{s-}\Xi^{klk'l'}_{r-}\bar{M}^{ii'}_{s-} \bar{M}^{kk'}_{r-}\widetilde{H(2)}_sD(2)^n_s\mathcal{R}_rD(2)^n_r \mathrm{d}B(2)^n_s\mathrm{d}B(2)^n_r,\\
&\mathbf{III}=\sum_{i,i',j,j'}\sum_{k,k',l,l'}\int_0^t\int_0^t\Xi^{iji'j'}_{s-}\Xi^{klk'l'}_{r-}\bar{M}^{ii'}_{s-} \bar{M}^{kk'}_{r-}\mathcal{R}_sD(2)^n_s\widetilde{H(2)}_rD(2)^n_r \mathrm{d}B(2)^n_s\mathrm{d}B(2)^n_r,\\
&\mathbf{IV}=\sum_{i,i',j,j'}\sum_{k,k',l,l'}\int_0^t\int_0^t\Xi^{iji'j'}_{s-}\Xi^{klk'l'}_{r-}\bar{M}^{ii'}_{s-} \bar{M}^{kk'}_{r-}\mathcal{R}_sD(2)^n_s\mathcal{R}_rD(2)^n_r \mathrm{d}B(2)^n_s\mathrm{d}B(2)^n_r.
\end{align*}
The condition [H](ii) and the Schwarz inequality yield
\begin{align*}
&|\mathbf{II}|\\
\lesssim&\left(\sum_{i,i',j,j'}\int_0^t|\Xi^{iji'j'}_{s-}\bar{M}^{ii'}_{s-}|\mathrm{d}B(2)^n_s\right)\left(\sum_{k,k',l,l'}\int_0^t|\Xi^{klk'l'}_{r-}\bar{M}^{ii'}_{r-}\mathcal{R}_r|\mathrm{d}B(2)^n_r\right)\\
\leq &\left(\sum_{i,i',j,j'}\int_0^t|\Xi^{iji'j'}_{s-}|\mathrm{d}B(2)^n_s\right)^{1/2}\left(\sum_{i,i',j,j'}\int_0^t|\Xi^{iji'j'}_{s-}||\bar{M}^{ii'}_{s-}|^2\mathrm{d}B(2)^n_s\right)\left(\sum_{k,k',l,l'}\int_0^t|\Xi^{klk'l'}_{r-}||\mathcal{R}_r|^2\mathrm{d}B(2)^n_r\right)^{1/2},
\end{align*}
hence by Lemma \ref{estXi}(a) and the Schwarz inequality we obtain
\begin{align*}
E[|\mathbf{II}|]\lesssim
k_n^2\left\{E\left[\left(\sum_{i,i',j,j'}\int_0^t|\Xi^{iji'j'}_{s-}||\bar{M}^{ii'}_{s-}|^2\mathrm{d}B(2)^n_s\right)^{2}\right]E\left[\sum_{k,k',l,l'}\int_0^t|\Xi^{klk'l'}_{r-}||\mathcal{R}_r|^2\mathrm{d}B(2)^n_r\right]\right\}^{1/2}.
\end{align*}
By the Schwarz inequality and Lemma \ref{estXi}(a) we have
\begin{align*}
&E\left[\left(\sum_{i,i',j,j'}\int_0^t|\Xi^{iji'j'}_{s-}||\bar{M}^{ii'}_{s-}|^2\mathrm{d}B(2)^n_s\right)^{2}\right]\\
\leq &E\left[\left(\sum_{i,i',j,j'}\int_0^t|\Xi^{iji'j'}_{s-}|\mathrm{d}B(2)^n_s\right)\left(\sum_{i,i',j,j'}\int_0^t|\Xi^{iji'j'}_{s-}||\bar{M}^{ii'}_{s-}|^4\mathrm{d}B(2)^n_s\right)\right]
\lesssim k_n^4 E\left[\sum_{i,i',j,j'}\int_0^t|\Xi^{iji'j'}_{s-}||\bar{M}^{ii'}_{s-}|^4\mathrm{d}B(2)^n_s\right].
\end{align*}
Moreover, since $B(2)^n$ is a deterministic nondecreasing process or a $\mathbf{G}^{(n)}$-adapted point process, by Lemma \ref{HYlem12.3} we have
\begin{align*}
E\left[\sum_{i,i',j,j'}\int_0^t|\Xi^{iji'j'}_{s-}||\bar{M}^{ii'}_{s-}|^4\mathrm{d}B(2)^n_s\right]
=E\left[\sum_{i,i',j,j'}\int_0^t|\Xi^{iji'j'}_{s-}|E\left[|\bar{M}^{ii'}_{s-}|^4\big|\mathcal{F}_{S^{i\wedge i'}}\right]\mathrm{d}B(2)^n_s\right],
\end{align*}
hence by Lemma \ref{HYlem12.5}(a) and Lemma \ref{estXi}(a) we obtain
\begin{align*}
E\left[\left(\sum_{i,i',j,j'}\int_0^t|\Xi^{iji'j'}_{s-}||\bar{M}^{ii'}_{s-}|^2\mathrm{d}B(2)^n_s\right)^{2}\right]
\lesssim k_n^{8}(k_n\bar{r}_n)^4.
\end{align*}
On the other hand, since $B(2)^n$ is a deterministic nondecreasing process or a $\mathbf{G}^{(n)}$-adapted point process, by Lemma \ref{HYlem12.3} we obtain
\begin{align*}
E\left[\sum_{k,k',l,l'}\int_0^t|\Xi^{klk'l'}_{r-}||\mathcal{R}_r|^2\mathrm{d}B(2)^n_r\right]=E\left[\sum_{k,k',l,l'}\int_0^t|\Xi^{klk'l'}_{r-}|E\left[|\mathcal{R}_r|^2\big|\mathcal{G}^{(n)}_r\right]\mathrm{d}B(2)^n_r\right],
\end{align*}
and thus [H](iii) and Lemma \ref{estXi}(a) yield
\begin{align*}
E\left[\sum_{k,k',l,l'}\int_0^t|\Xi^{klk'l'}_{r-}||\mathcal{R}_r|^2\mathrm{d}B(2)^n_r\right]\lesssim b_n^{\left(\xi-\frac{1}{2}\right)\left(1-\lambda\right)}k_n^4
\end{align*}
for any $\lambda>0$. Consequently, we conclude that
$E[|\mathbf{II}|]\lesssim (k_n\bar{r}_n)^{2}b_n^{\left(\xi-\frac{1}{2}\right)\frac{1-\lambda}{2}}k_n^8,$
and thus we obtain
\begin{align*}
b_n^{-1}|\mathbf{II}|=O_p\left(b_n^{-1+2\left(\xi'-\frac{1}{2}\right)+\left(\xi-\frac{1}{2}\right)\frac{1-\lambda}{2}}k_n^8\right)
=O_p\left(b_n^{2(\xi'-\xi)+\frac{5}{2}\left(\xi-\frac{9}{10}\right)-\frac{\lambda}{2}\left(\xi-\frac{1}{2}\right)}k_n^8\right).
\end{align*}
Since $\xi'>\xi>9/10$ and $\lambda>0$ can be taken arbitrarily small, we conclude that $b_n^{-1}\mathbf{II}=o_p(k_n^8)$. In a similar manner, we can show that $b_n^{-1}\mathbf{III}=o_p(k_n^8)$ and $b_n^{-1}\mathbf{IV}=o_p(k_n^8)$.

Next, we evaluate $E[\mathbf{I}]$. In light of Lemma \ref{HYlem12.4}(a), the terms contribute to the sum only when $|i-k|\wedge|i'-k|<k_n$ or $|i-k'|\wedge|i'-k'|<k_n$. Therefore, [H](ii) and the Schwarz inequality yield
\begin{align*}
&|E[\mathbf{I}]|\\
\lesssim &\sum_{j,j',l,l'}\left\{\sum_{\begin{subarray}{c}i,i',k,k'\\|i-k|<k_n\end{subarray}}+\sum_{\begin{subarray}{c}i,i',k,k'\\|i'-k|<k_n\end{subarray}}+\sum_{\begin{subarray}{c}i,i',k,k'\\|i-k'|<k_n\end{subarray}}+\sum_{\begin{subarray}{c}i,i',k,k'\\|i'-k'|<k_n\end{subarray}}\right\}E\left[\int_0^t\int_0^t|\Xi^{iji'j'}_s\Xi^{klk'l'}_r\bar{M}^{ii'}_{s-}\bar{M}^{ii'}_{r-}|\mathrm{d}B(2)^n_s\mathrm{d}B(2)^n_r\right]\\
=:&\mathbf{A}_{1}+\mathbf{A}_{2}+\mathbf{A}_{3}+\mathbf{A}_{4}.
\end{align*}
Consider $\mathbf{A}_1$. We rewrite it as
\begin{align*}
\mathbf{A}_1=\sum_{i,k:|i-k|<k_n}E\left[\left(\sum_{i',j,j'}\int_0^t|\Xi^{iji'j'}_s\bar{M}^{ii'}_{s-}|\mathrm{d}B(2)^n_s\right)\left(\sum_{k',l,l'}\int_0^t|\Xi^{klk'l'}_s\bar{M}^{ii'}_{s-}|\mathrm{d}B(2)^n_s\right)\right].
\end{align*}
Then by the Schwarz inequality and the inequality of arithmetic and geometric means we obtain
\begin{align*}
\mathbf{A}_1\leq k_n\sum_{i}E\left[\left(\sum_{i',j,j'}\int_0^t|\Xi^{iji'j'}_s\bar{M}^{ii'}_{s-}|\mathrm{d}B(2)^n_s\right)^2\right].
\end{align*}
Therefore, by the Schwarz inequality and the second inequality of Lemma \ref{estXi}(a) we obtain
\begin{align*}
\mathbf{A}_1\lesssim k_n\bar{r}_n k_n^4 E\left[\sum_{i,i',j,j'}\int_0^t|\Xi^{iji'j'}_s||\bar{M}^{ii'}_{s-}|^2\mathrm{d}B(2)^n_s\right].
\end{align*}
Since $B(2)^n$ is a deterministic nondecreasing process or a $\mathbf{G}^{(n)}$-adapted point process, by Lemma \ref{HYlem12.3} we have
\begin{align*}
E\left[\sum_{i,i',j,j'}\int_0^t|\Xi^{iji'j'}_{s-}||\bar{M}^{ii'}_{s-}|^2\mathrm{d}B(2)^n_s\right]
=E\left[\sum_{i,i',j,j'}\int_0^t|\Xi^{iji'j'}_{s-}|E\left[|\bar{M}^{ii'}_{s-}|^2\big|\mathcal{F}_{S^{i\wedge i'}}\right]\mathrm{d}B(2)^n_s\right],
\end{align*}
hence by Lemma \ref{HYlem12.5}(a) and Lemma \ref{estXi}(a) we obtain
$\mathbf{A}_1\lesssim k_n\bar{r}_n k_n^8(k_n\bar{r}_n)^2=k_n^8(k_n\bar{r}_n)^3.$
Similarly we can also show that $\mathbf{A}_l\lesssim k_n^8(k_n\bar{r}_n)^{3}$ for $l\in\{2,3,4\}$. Consequently, we have $|E[\mathbf{I}]|\lesssim k_n^8(k_n\bar{r}_n)^{3}$, so that we conclude that $b_n^{-1}|E[\mathbf{I}]|=o(k_n^8)$ because $b_n^{-1}(k_n\bar{r}_n)^{3}=O(b_n^{3\xi'-\frac{5}{2}})=o(1)$.

After all, we conclude that $b_n^{-1/2}\Delta_{1,t}=o_p(k_n^4)$. By symmetry, we also obtain $b_n^{-1/2}\Delta_{2,t}=o_p(k_n^4)$.

Next we consider $\Delta_{3,t}$. By the use of associativity and linearity of integration, we have
\begin{align*}
\Delta_{3,t}=&\sum_{i,j,i',j'}(\bar{K}^{ij}_-\bar{K}^{i'j'}_-\bar{L}^{ij'}_-)\bullet\langle\bar{M}^{\prime n}_{\alpha'}(\hat{\mathcal{I}})^{i'},\bar{N}^n_\beta(\hat{\mathcal{J}})^j\rangle_t.
\end{align*}
Moreover, we have
\begin{align*}
\langle\bar{M}^{\prime n}_{\alpha'}(\hat{\mathcal{I}})^{i'},\bar{N}^n_\beta(\hat{\mathcal{J}})^{j}\rangle_t
=&\sum_{p,q=0}^{k_n-1}\alpha^{\prime n}_p \beta^n_{q}(\hat{I}^{i'+p}_-\hat{J}^{j+q}_-)\bullet\langle M^{\prime n},N^n\rangle_t\\
=&\sum_{p,q=0}^{k_n-1}\alpha^{\prime n}_p \beta^n_{q}\{\hat{I}^{i'+p}_-\hat{J}^{j+q}_-H(4)D(4)^n\}\bullet B(4)^n_t,
\end{align*}
hence we obtain
\begin{align*}
\Delta_{3,t}=\sum_{i,j,i',j'}\int_0^t\Lambda^{iji'j'}_{s-}\bar{L}^{ij'}_{s-}H(4)_sD(4)^n_s\mathrm{d}B(4)^n_s.
\end{align*}
Therefore, using Lemma \ref{HYlem12.4}(b), Lemma \ref{estXi}(b) and Lemma \ref{HYlem12.5}(b) instead of Lemma \ref{HYlem12.4}(a),  Lemma \ref{estXi}(a) and Lemma \ref{HYlem12.5}(a) respectively, we can adopt an argument similar to the above one. After all, we conclude that $b_n^{-1/2}\Delta_{3,t}=o_p(k_n^4)$. By symmetry, we also obtain $b_n^{-1/2}\Delta_{4,t}=o_p(k_n^4)$. 
\end{proof}

\begin{proof}[\upshape{\bfseries{Proof of Proposition \ref{HYprop5.1}}}]

By a localization procedure, we may assume that [SA3], [SA4], [SN] and [SC3] hold instead of [A3], $(\ref{SA4})$, [N] and [C3] respectively.

Lemma \ref{HYlem12.6and12.8} implies that it is sufficient to prove that the condition [H] holds for $M^n,M^{\prime n}\in\{X,\mathfrak{E}^X,\mathfrak{Z}^X\}$ and $N^n,N^{\prime n}\in\{Y,\mathfrak{E}^Y,\mathfrak{Z}^Y\}$, but it immediately follows from [SA3], [SA4], [SN] and [SC3].
\end{proof}


\section{Proof of Theorem \ref{HYthm6.1} and \ref{HYthm6.2}}\label{drift}

Before starting the proof, we strengthened the conditions [A5] and [A6] as follows:
\begin{enumerate}
\item[[{SA5]}] We have [A5], and the processes $A^{X}$, $A^{Y}$, $\underline{A}^{X}$, $\underline{A}^{Y}$, $(A^{X})'$, $(A^{Y})'$, $(\underline{A}^{X})'$ and $(\underline{A}^{Y})'$ are bounded. Furthermore, there exist a positive constant $C$ and $\lambda\in(0,3/4)$ satisfying
\begin{equation}\label{eqSA5}
E\left[|f_t-f_{\tau\wedge t}|^2\big|\mathcal{F}_{\tau\wedge t}\right]\leq C|t-\tau|^{1-\lambda}
\end{equation}
for every $t>0$ and any bounded $\mathbf{F}^{(0)}$-stopping time $\tau$, for the density processes $f=(A^X)'$, $(A^{Y})'$, $(\underline{A}^X)'$ and $(\underline{A}^{Y})'$.
\item[{[SA6]}] There exists a positive constant $C$ such that $b_n^{-1}H_n(t)\leq C$ for every $t$.
\end{enumerate}

The following lemma can be shown by arguments analogous to those used in the proofs of Lemma 13.1 and 13.2 in \cite{HY2011}.
\begin{lem}\label{HYlem13.1and13.2}
Suppose that $[\mathrm{A}2]$ and $[\mathrm{SA}4]$ are satisfied. Let $(M^n)$ be a sequence of square-integrable martingales such that there exists a positive constant $C_1$ satisfying
\begin{equation}\label{nagasa2}
\sup_{q\in\mathbb{N}}\langle M^n\rangle(\widehat{J}^q)_t\leq C_1\bar{r}_n
\end{equation}
for any $t\in\mathbb{R}_+$ and any $n\in\mathbb{N}$. Let $\alpha,\beta\in\Phi$ and let $A$ be an $\mathbf{F}^{(0)}$-adapted process with a bounded derivative such that there are a positive constant $C_2$ and a constant $\lambda\in(0,3/4)$ satisfying
\begin{equation}\label{SA5}
E\left[(A'_{t}-A'_{\tau\wedge t})^2\big|\mathcal{F}_{\tau\wedge t}\right]\leq C_2|t-\tau|^{1-\lambda}
\end{equation}
for any $t\in\mathbb{R}_+$ and any bounded $\mathbf{F}^{(0)}$-stopping time $\tau$. Let $D^n$ be a c\`adl\`ag $\mathbf{G}^{(n)}$-adapted process for each $n$ and suppose that $\sup_n|D^n|$ is a bounded process. Set $A^n=D^n\bullet A$ and define  
\begin{align*}
\mathbb{I}_t=\sum_{i,j=1}^{\infty}\bar{K}^{ij}_-\bullet\{\bar{A}^n_\alpha(\widehat{\mathcal{I}})^i_-\bullet\bar{M}^n_\beta(\widehat{\mathcal{J}})^j\}_t,\qquad
\mathbb{II}_t=\sum_{i,j=1}^{\infty}\bar{K}^{ij}_-\bullet\{\bar{M}^n_\beta(\widehat{\mathcal{J}})^j_-\bullet\bar{A}^n_\alpha(\widehat{\mathcal{I}})^i\}_t.
\end{align*}
for each $t\in\mathbb{R}_+$. Then
\begin{enumerate}[\normalfont (a)]
\item $b_n^{-1/4}\sup_{0\leq s\leq t}|\mathbb{I}_{s}|=o_p(k_n^2)$ for every $t$.
\item $b_n^{-1/4}|\mathbb{II}_{t}|=o_p(k_n^2)$ for every $t$.
\item Suppose that $[\mathrm{SC}1]$-$[\mathrm{SC}2]$ and $[\mathrm{SA}6]$ are fulfilled. Then $b_n^{-1/4}\sup_{0\leq s\leq t}|\mathbb{II}_{s}|=o_p(k_n^2)$ for every $t$ if $M^n\in\{M^Y,\mathfrak{E}^Y,\mathfrak{M}^Y\}$.
\end{enumerate}
\end{lem}

\begin{proof}
(a) The process $\mathbb{I}$ is clearly a locally square-integrable martingale with the predictable quadratic variation
\begin{equation*}
\langle\mathbb{I}\rangle_t
=\sum_{i,i',j,j'=1}^{\infty}\bar{K}^{i j}_-\bar{K}^{i' j'}_-\bar{A}^n_\alpha(\widehat{\mathcal{I}})^i_-\bar{A}^n_\alpha(\widehat{\mathcal{I}})^{i'}_-\bullet\langle\bar{M}^n_\beta(\widehat{\mathcal{J}})^j,\bar{M}^{n}_\beta(\widehat{\mathcal{J}})^{j'}\rangle_t.
\end{equation*}
Since
\begin{align*}
\langle\bar{M}^{n}_\beta(\widehat{\mathcal{J}})^j,\bar{M}^{n}_\beta(\widehat{\mathcal{J}})^{j'}\rangle_t
=\sum_{q,q'=0}^{k_n-1}\beta^n_q \beta^n_{q'}\widehat{J}^{j+q}_-\widehat{J}^{j'+q'}_-\bullet\langle M^{n}\rangle_t,
\end{align*}
we have
\begin{align*}
\langle\mathbb{I}\rangle_t
=&\sum_{i,i',j,j'=1}^{\infty}\sum_{q,q'=0}^{k_n-1}\beta^n_q \beta^n_{q'}\bar{K}^{i j}_-\bar{K}^{i' j'}_-\bar{A}^n_\alpha(\widehat{\mathcal{I}})^i_-\bar{A}^n_\alpha(\widehat{\mathcal{I}})^{i'}_-\widehat{J}^{j+q}_-\widehat{J}^{j'+q'}_-\bullet\langle M^n\rangle_t\\
=&\sum_{i,i',q=1}^{\infty}\sum_{j,j'=(q-k_n+1)\vee 1}^{q}\beta^n_{q-j} \beta^n_{q-j'}\bar{K}^{i j}_-\bar{K}^{i' j'}_-\bar{A}^n_\alpha(\widehat{\mathcal{I}})^i_-\bar{A}^n_\alpha(\widehat{\mathcal{I}})^{i'}_-\widehat{J}^{q}_-\bullet\langle M^n\rangle_t.
\end{align*}
Moreover,  since $|\bar{A}^n_\alpha(\widehat{\mathcal{I}})^i_s|\lesssim\sum_{p=0}^{k_n-1}|\widehat{I}^{i+p}(t)|$ and $\bar{K}^{ij}_s\leq\bar{K}^{ij}_t$ if $s\leq t$ , we have
\begin{align*}
\langle\mathbb{I}\rangle_t
\lesssim\sum_{i,i',q=1}^{\infty}\sum_{j,j'=(q-k_n+1)\vee 1}^{q}\sum_{p,p'=0}^{k_n-1}\bar{K}^{i j}_t\bar{K}^{i' j'}_t|\widehat{I}^{i+p}(t)||\widehat{I}^{i'+p'}(t)|\langle M^n\rangle(\widehat{J}^{q})_t,
\end{align*}
and thus [SA4] and $(\ref{nagasa2})$ yields
\begin{align*}
\langle\mathbb{I}\rangle_t
\lesssim k_n\bar{r}_n^2\sum_{i,i',q=1}^{\infty}\sum_{j,j'=(q-k_n+1)\vee 1}^{q}\sum_{p=0}^{k_n-1}\bar{K}^{i j}_t\bar{K}^{i' j'}_t|\widehat{I}^{i+p}(t)|.
\end{align*}
Since $\sum_{i'=1}^{\infty}\bar{K}^{i' j'}_t\lesssim  k_n$, we have
\begin{align*}
\langle\mathbb{I}\rangle_t
\lesssim k_n^3\bar{r}_n^2\sum_{i,q=1}^{\infty}\sum_{j=(q-k_n+1)\vee 1}^{q}\sum_{p=0}^{k_n-1}\bar{K}^{i j}_t|\widehat{I}^{i+p}(t)|
=k_n^3\bar{r}_n^2\sum_{p,q=0}^{k_n-1}\sum_{i=1}^{\infty}|\widehat{I}^{i+p}(t)|\sum_{j=1}^{\infty}\bar{K}^{i j}_t.
\end{align*}
Since $\sum_{i=1}^{\infty}|\widehat{I}^{i+p}(t)|\leq t$ and $\sum_{j=1}^{\infty}\bar{K}^{i j}_t\lesssim  k_n$, we obtain $\langle\mathbb{I}\rangle_t\lesssim k_n^6 \bar{r}_n^2$, and thus we have $b_n^{-1/2}\langle\mathbb{I}\rangle_t=O_p(k_n^4\cdot b_n^{2\xi'-3/2})=o_p(k_n^4)$ because $\xi'>9/10$. The Lenglart inequality implies that $b_n^{-1/4}\sup_{0\leq s\leq t}|\mathbb{I}_{s}|=o_p(k_n^2)$ as desired.

(b) We rewrite the target quantity as
\begin{align*}
\mathbb{II}_{t}
=&\sum_{i,j=1}^{\infty}\sum_{p=0}^{k_n-1}\alpha^n_p\bar{K}^{i j}_-\bar{M}^{n}_\beta(\widehat{\mathcal{J}})^j_-\widehat{I}^{i+p}_-D^n\bullet A_t\\
=&\sum_{i,j=1}^{\infty}\sum_{p=0}^{k_n-1}\alpha^n_p A'_{\widehat{T}^{j}}\int_0^t\bar{K}^{i j}_s\widehat{I}^{i+p}_sD^n_s \bar{M}^{n}_\beta(\widehat{\mathcal{J}})^j_s\mathrm{d}s
+\sum_{i,j=1}^{\infty}\sum_{p=0}^{k_n-1}\alpha^n_p\int_0^t\bar{K}^{i j}_s\widehat{I}^{i+p}_s D^n_s\bar{M}^{n}_\beta(\widehat{\mathcal{J}})^j_s\{ A'_s-A'_{\widehat{T}^{j}}\}\mathrm{d}s\\
=:&\mathbb{II}_{1,t}+\mathbb{II}_{2,t}.
\end{align*}

First we claim that $b_n^{-1/4}\mathbb{II}_{1,t}=o_p(k_n^2)$ as $n\rightarrow\infty$. Since $A'_{\widehat{T}^{j}}A'_{\widehat{T}^{j'}}\bar{K}^{i j}_s\widehat{I}^{i+p}_sD^n_s\bar{K}^{i' j'}_u\widehat{I}^{i'+p'}_uD^n_u$ is $\mathcal{F}_{\widehat{T}^{j\vee j'}}$-measurable due to Lemma \ref{HYlem12.3} and $E[\bar{M}^{n}_\beta(\widehat{\mathcal{J}})^j_s\bar{M}^{n}_\beta(\widehat{\mathcal{J}})^{j'}_u|\mathcal{F}_{\widehat{T}^{j\vee j'}}]=0$ if $|j-j'|>k_n-1$ due to the optional sampling theorem, we have
\begin{align*}
&E[\mathbb{II}_{1,t}^2]\\
=&\sum_{i,i'}\sum_{j,j':|j-j'|\leq k_n}
\sum_{p,p'=0}^{k_n-1}\alpha^n_p\alpha^n_{p'}\int_0^t\int_0^t
E\left[A'_{\widehat{T}^{j}}A'_{\widehat{T}^{j'}}\bar{K}^{i j}_s\widehat{I}^{i+p}_s D^n_s\bar{K}^{i' j'}_u\widehat{I}^{i'+p'}_u D^n_u\bar{M}^{n}_\beta(\widehat{\mathcal{J}})^j_s\bar{M}^{n}_\beta(\widehat{\mathcal{J}})^{j'}_u\right]\mathrm{d}s\mathrm{d}u\\
\lesssim &\sum_{i,i'}\sum_{j,j':|j-j'|\leq k_n}
\sum_{p,p'=0}^{k_n-1}\int_0^t\int_0^t
E\left[\bar{K}^{i j}_s\widehat{I}^{i+p}_s\bar{K}^{i' j'}_u\widehat{I}^{i'+p'}_u \left\{|\bar{M}^{n}_\beta(\widehat{\mathcal{J}})^j_s|^2+|\bar{M}^{n}_\beta(\widehat{\mathcal{J}})^{j'}_u|^2\right\}\right]\mathrm{d}s\mathrm{d}u.
\end{align*}
On the other hand, the optional sampling theorem and $(\ref{nagasa2})$ yield
\begin{equation}\label{estM}
E\left[|\bar{M}^{n}_\beta(\widehat{\mathcal{J}})^j_s|^2\big|\mathcal{F}_{\widehat{T}^j}\right]
=E\left[\langle\bar{M}^{n}_\beta(\widehat{\mathcal{J}})^j_s\rangle\big|\mathcal{F}_{\widehat{T}^j}\right]
=\sum_{q=0}^{k_n-1}(\beta^n_q)^2E\left[\langle M^n\rangle(\widehat{J}^{j+q})_t\big|\mathcal{F}_{\widehat{T}^j}\right]
\lesssim k_n\bar{r}_n.
\end{equation}
Since $\bar{K}^{i j}_s\widehat{I}^{i+p}_s\bar{K}^{i' j'}_u\widehat{I}^{i'+p'}_u$ is $\mathcal{F}_{\widehat{T}^\wedge(j,j')}$-measurable by Lemma \ref{HYlem12.3}, we obtain
\begin{align*}
E[\mathbb{II}_{1,t}^2]
\lesssim &k_n\bar{r}_nE\left[\sum_{i,i'}\sum_{j,j':|j-j'|\leq k_n}
\sum_{p,p'=0}^{k_n-1}
\bar{K}^{i j}_t|\widehat{I}^{i+p}(t)|\bar{K}^{i' j'}_t|\widehat{I}^{i'+p'}(t)|\right].
\end{align*}
Therefore, [SA4] and $(\ref{sumbarK})$ imply that
\begin{align*}
E[\mathbb{II}_{1,t}^2]
\lesssim k_n^4\bar{r}_n^2 E\left[\sum_{i,j}\sum_{p=0}^{k_n-1}
\bar{K}^{i j}_t|\widehat{I}^{i+p}(t)|\right]
=k_n^4\bar{r}_n^2 E\left[\sum_{p=0}^{k_n-1}\sum_{i=1}^{\infty}|\widehat{I}^{i+p}(t)|\sum_{j=1}^{\infty}
\bar{K}^{i j}_t\right].
\end{align*}
Since $\sum_{i=1}^{\infty}|\widehat{I}^{i+p}(t)|\leq t$ and $\sum_{j=1}^{\infty}\bar{K}^{ij}_t\lesssim k_n$, we conclude that $E[\mathbb{II}_{1,t}^2]\lesssim k_n^6\bar{r}_n^2$, hence $b_n^{-1/4}\mathbb{II}_{1,t}=o_p(k_n^2)$.

Next we estimate $\mathbb{II}_{2,t}$. By Lemma \ref{HYlem12.3} we have
\begin{align*}
E[|\mathbb{II}_{2,t}|]
\lesssim\sum_{i,j=1}^{\infty}\sum_{p=1}^{k_n-1}E\left[\int_0^t\bar{K}^{i j}_s\widehat{I}^{i+p}_sE\left[|\bar{M}^{n}_\beta(\widehat{\mathcal{J}})^j_s( A'_s-A'_{\widehat{T}^{j}})|\big|\mathcal{F}_{T^j}\right]\mathrm{d}s\right],
\end{align*}
hence by the Schwarz inequality and $(\ref{estM})$ we obtain
\begin{align*}
E[|\mathbb{II}_{2,t}|]
\lesssim(k_n\bar{r}_n)^{\frac{1}{2}}\sum_{i,j=1}^{\infty}\sum_{p=1}^{k_n-1}E\left[\int_0^t\bar{K}^{i j}_s\widehat{I}^{i+p}_s\left\{E\left[( A'_s-A'_{\widehat{T}^{j}})^2\big|\mathcal{F}_{T^j}\right]\right\}^{1/2}\mathrm{d}s\right].
\end{align*}
Moreover, if $\bar{I}^i\cap\bar{J}^j\neq\emptyset$ we have $|s-\widehat{T}^j|\leq|\widehat{S}^{i+k_n}-\widehat{T}^j|\vee|\widehat{S}^i-\widehat{T}^j|\leq(\widehat{S}^{i+k_n}-\widehat{S}^i)+(\widehat{T}^{j+k_n}-\widehat{T}^j)\leq 2k_nr_n(t)$ for any $s\in\bar{I}^i$. Hence by [SA4] and $(\ref{SA5})$ we conclude that 
\begin{align*}
E[|\mathbb{II}_{2,t}|]
\lesssim (k_n\bar{r}_n)^{1-\frac{\lambda}{2}}\sum_{i,j=1}^{\infty}\sum_{p=1}^{k_n-1}E\left[\int_0^t\bar{K}^{i j}_s\widehat{I}^{i+p}_s\mathrm{d}s\right]
\leq (k_n\bar{r}_n)^{1-\frac{\lambda}{2}}E\left[\sum_{p=1}^{k_n-1}\sum_{i=1}^{\infty}|\widehat{I}^{i+p}(t)|\sum_{j=1}^{\infty}\bar{K}^{i j}_t\right]
\end{align*}
for some $\lambda\in(0,3/4)$. Since $\sum_{i=1}^{\infty}|\widehat{I}^{i+p}(t)|\leq t$ and $\sum_{j=1}^{\infty}\bar{K}^{ij}_t\lesssim k_n$, we conclude that $E[|\mathbb{II}_{2,t}|]\lesssim k_n^2(k_n\bar{r}_n)^{1-\lambda/2}$, hence $b_n^{-1/4}\mathbb{II}_{2,t}=O_p(k_n^2b_n^{(\xi'-1/2)(1-\lambda/2)-1/4})=o_p(k_n^2)$ because $\xi'>9/10$ and $\lambda\in(0,3/4)$.

Consequently, we obtain $b_n^{-1/4}\mathbb{II}_{t}=o_p(k_n^2)$ as $n\to\infty$ for every $t$.

(c) In light of (b), it is sufficient to show that $(b_n^{-1/4}k_n^{-2}\mathbb{II})_{n\in\mathbb{N}}$ is C-tight. 

Fix a $T>0$. Rewrite $\mathbb{II}$ as
\begin{align*}
\mathbb{II}_t=\sum_{j=1}^{\infty}\{\bar{M}^n_\beta(\widehat{\mathcal{J}})^j_-\Upsilon^{j}_-\}\bullet A_t,
\end{align*}
where $\Upsilon^{j}_s=\sum_{i=1}^{\infty}\sum_{p=0}^{k_n-1}\alpha^n_p\bar{K}^{ij}_s\widehat{I}^{i+p}_sD^n_s$ for each $s\in\mathbb{R}_+$. Then for $0\leq s<t\leq T$
\begin{align*}
&b_n^{-1/4}|\mathbb{II}_t-\mathbb{II}_s|
\leq b_n^{-1/4}\sum_{j=1}^{\infty}\int_s^t\left|\bar{M}^n_\beta(\widehat{\mathcal{J}})^j_u\right|\left|A'_u\right|\left|\Upsilon^{j}_u\right|\mathrm{d}u\\
\leq &\left\{\sum_{j=1}^{\infty}\int_s^t\left(b_n^{-1/4}\bar{M}^n_\beta(\widehat{\mathcal{J}})^j_u\right)^2\left|\Upsilon^{j}_u\right|\mathrm{d}u\right\}^{1/2}\left\{\sum_{j=1}^{\infty}\int_s^t\left( A'_u\right)^2\left|\Upsilon^{j}_u\right|\mathrm{d}u\right\}^{1/2}
\lesssim k_n(t-s)^{1/2}\Theta_n(T)^{1/2},
\end{align*}
where
\begin{align*}
\Theta_n(\cdot)=\sum_{j=1}^{\infty}\int_0^{\cdot}\left(\bar{M}^n_\beta(\widehat{\mathcal{J}})^j_u\right)^2\Upsilon^{j}_u\mathrm{d}u.
\end{align*}
Since $\Upsilon^{j}_u$ is $\mathcal{F}_{\widehat{T}^j}$-measurable due to Lemma \ref{HYlem12.3},
\begin{align*}
&E[\Theta_n(T)]=E\left[\sum_{j=1}^{\infty}\int_0^{T}E\left[\left(b_n^{-1/4}\bar{M}^n_\beta(\widehat{\mathcal{J}})^j_u\right)^2\big|\mathcal{F}_{\widehat{T}^j}\right]\Upsilon^{j}_u\mathrm{d}u\right]
=E\left[\sum_{j=1}^{\infty}\int_0^{T}b_n^{-1/2}\langle\bar{M}^n_\beta(\widehat{\mathcal{J}})^j\rangle_u\Upsilon^{j}_u\mathrm{d}u\right]\\
&\leq E\left[b_n^{-1/2}\sum_{p=0}^{k_n-1}\sum_{i,j=1}^{\infty}\langle\bar{M}^n_\beta(\widehat{\mathcal{J}})^j\rangle_T\bar{K}^{ij}_T|\widehat{I}^{i+p}(T)|\right]
\leq E\left[b_n^{-1/2}\sum_{p,q=0}^{k_n-1}\sum_{i,j=1}^{\infty}\bar{K}^{ij}_T|\widehat{I}^{i+p}(T)|\langle M^n\rangle(\widehat{J}^{j+q})_T\right].
\end{align*}
If $M^n=M^Y$ or $M^n=\mathfrak{M}^Y$, [SC1] yields
\begin{align*}
&b_n^{-1/2}\sum_{p,q=0}^{k_n-1}\sum_{i,j=1}^{\infty}\bar{K}^{ij}_T|\widehat{I}^{i+p}(T)|\langle M^n\rangle(\widehat{J}^{j+q})_T
\lesssim b_n^{-1/2}\sum_{p,q=0}^{k_n-1}\sum_{i,j=1}^{\infty}\bar{K}^{ij}_T|\widehat{I}^{i+p}(T)||\widehat{J}^{j+q}(T)|\\
\lesssim &b_n^{-1}\sum_{p,q=0}^{k_n-1}\left(\sum_{i=1}^{\infty}|\widehat{I}^{i+p}(T)|^2+\sum_{j=1}^{\infty}|\widehat{J}^{j+q}(T)|^2\right)
\leq 4 k_n^2\cdot b_n^{-1}H_n(T),
\end{align*}
and thus by [SA6] we obtain $E[\Theta_n(T)]\lesssim k_n^2$. On the other hand, if $M^n=\mathfrak{E}^Y$, [SC2] yields
\begin{align*}
&b_n^{-1/2}\sum_{p,q=0}^{k_n-1}\sum_{i,j=1}^{\infty}\bar{K}^{ij}_T|\widehat{I}^{i+p}(T)|\langle M^n\rangle(\widehat{J}^{j+q})_T
\lesssim k_n^{-1}\sum_{p,q=0}^{k_n-1}\sum_{i,j=1}^{\infty}\bar{K}^{ij}_T|\widehat{I}^{i+p}(T)|1_{\{\widehat{T}^{j+q}\leq T\}}\\
\lesssim& \sum_{p=0}^{k_n-1}\sum_{i=1}^{\infty}|\widehat{I}^{i+p}(T)|\sum_{j=1}^{\infty}\bar{K}^{ij}_T
\lesssim k_n^2,
\end{align*}
and thus again we obtain $E[\Theta_n(T)]\lesssim k_n^2$. 

After all, for any $\eta>0$ we have
\begin{align*}
\sup_{n\in\mathbb{N}}P\left[w(b_n^{-1/4}k_n^{-2}\mathbb{II};\delta,T)\geq\eta\right]
\leq\sup_{n\in\mathbb{N}}P\left[k_n^{-1}\delta^{1/2}\Theta_n(T)^{1/2}\geq\eta\right]
\leq\eta^{-2}\delta\sup_{n\in\mathbb{N}}k_n^{-2}E[\Theta_n(T)]\to0
\end{align*}
as $\delta\to0$ if $M^n\in\{M^Y,\mathfrak{E}^Y,\mathfrak{M}^Y\}$, and thus we complete the proof of (c).
\end{proof}

\begin{lem}\label{final}
Suppose that $[\mathrm{A}2]$-$[\mathrm{A}5]$, $[\mathrm{N}]$ and $[\mathrm{C}3]$ hold. Then:
\begin{enumerate}[\normalfont (a)]
\item Suppose $[\mathrm{A}6]$ holds. $b_n^{-1/4}\{\mathbf{M}^n-\widetilde{\mathbf{M}}^n\}\xrightarrow{ucp}0$ as $n\to\infty$, where
\begin{align*}
\widetilde{\mathbf{M}}^n_t=\frac{1}{(\psi_{HY}k_n)^2}\sum_{i,j=0}^{\infty}\bar{K}^{ij}_t\left\{\bar{L}_{g,g}(M^X,M^Y)^{ij}_t+\bar{L}_{g',g'}(\underline{\mathfrak{U}}^X,\underline{\mathfrak{U}}^Y)^{ij}_t+\bar{L}_{g,g'}(M^X,\underline{\mathfrak{U}}^Y)^{ij}_t+\bar{L}_{g',g}(\underline{\mathfrak{U}}^X,M^Y)^{ij}_t\right\}
\end{align*}
and $\underline{\mathfrak{U}}^X=\mathfrak{E}^X+(k_n\sqrt{b_n})^{-1}\mathfrak{M}^X$, $\underline{\mathfrak{U}}^Y=\mathfrak{E}^Y+(k_n\sqrt{b_n})^{-1}\mathfrak{M}^Y$.
\item $b_n^{-1/4}\langle\widetilde{\mathbf{M}}^n,N\rangle_t\to^p0$ as $n\to\infty$ for any $N\in\{M^X,M^Y,\underline{M}^X,\underline{M}^Y\}$ and every $t$.
\end{enumerate}
\end{lem}

\begin{proof}
By a localization procedure, we can assume that [SA3]-[SA6], [SN] and [SC3] instead of [A3]-[A6], [N] and [C3] respectively. Then, (a) immediately follows from Lemma \ref{HYlem13.1and13.2}. On the other hand, for $N\in\{M^X,M^Y,\underline{M}^X,\underline{M}^Y\}$ we have
\begin{align*}
&\langle\widetilde{\mathbf{M}}^n,N\rangle_t\\
=&\frac{1}{(\psi_{HY}k_n)^2}\sum_{i,j=0}^{\infty}\bigg(\bar{K}^{ij}_-\bullet\left\{\bar{M}^X_g(\widehat{\mathcal{I}})^i_-\bullet\overline{\langle M^Y,N\rangle}_g(\widehat{\mathcal{J}})^j\right\}_t+\bar{K}^{ij}_-\bullet\left\{\bar{M}^Y_g(\widehat{\mathcal{J}})^j_-\bullet\overline{\langle M^X,N\rangle}_g(\widehat{\mathcal{I}})^i\right\}_t\\
&\hphantom{\frac{1}{(\psi_{HY}k_n)^2}\sum_{i,j=0}^{\infty}(}+\bar{K}^{ij}_-\bullet\left\{\bar{\underline{\mathfrak{U}}}^Y_{g'}(\widehat{\mathcal{J}})^j_-\bullet\overline{\langle M^X,N\rangle}_g(\widehat{\mathcal{I}})^i\right\}_t+\bar{K}^{ij}_-\bullet\left\{\bar{M}^X_{g}(\widehat{\mathcal{I}})^i_-\bullet\overline{\langle \mathfrak{M}^Y,M^X\rangle}_{g'}(\widehat{\mathcal{J}})^j\right\}_t\\
&\hphantom{\frac{1}{(\psi_{HY}k_n)^2}\sum_{i,j=0}^{\infty}(}+\bar{K}^{ij}_-\bullet\left\{\bar{M}^Y_{g}(\widehat{\mathcal{J}})^j_-\bullet\overline{\langle \mathfrak{M}^X,N\rangle}_{g'}(\widehat{\mathcal{I}})^i\right\}_t+\bar{K}^{ij}_-\bullet\left\{\bar{\underline{\mathfrak{U}}}^X_{g'}(\widehat{\mathcal{I}})^i_-\bullet\overline{\langle M^Y,N\rangle}_g(\widehat{\mathcal{J}})^j\right\}_t\\
&\hphantom{\frac{1}{(\psi_{HY}k_n)^2}\sum_{i,j=0}^{\infty}(}+\bar{K}^{ij}_-\bullet\left\{\bar{\underline{\mathfrak{U}}}^Y_{g'}(\widehat{\mathcal{J}})^j_-\bullet\overline{\langle \mathfrak{M}^X,N\rangle}_{g'}(\widehat{\mathcal{I}})^i\right\}_t+\bar{K}^{ij}_-\bullet\left\{\bar{\underline{\mathfrak{U}}}^X_{g'}(\widehat{\mathcal{I}})^i_-\bullet\overline{\langle \mathfrak{M}^Y,N\rangle}_g(\widehat{\mathcal{J}})^j\right\}_t\bigg),
\end{align*}
hence Lemma \ref{HYlem13.1and13.2}(b) yields $\langle\widetilde{\mathbf{M}}^n,N\rangle_t=o_p(b_n^{1/4})$ because $(\ref{SA5})$ for $A=\langle L,M\rangle$ ($L,M=M^X,M^Y,\underline{M}^X,\underline{M}^Y$) holds by [SA3] and the processes $\mathfrak{I}$ and $\mathfrak{J}$ are $\mathbf{G}^{(n)}$-adapted by [A2].
\end{proof}

\begin{proof}[\upshape\bf Proof of Theorem \ref{HYthm6.1}]
Note that [A1$'$] implies [C3] due to Lemma \ref{HJYlem2.2}, in all cases (a), (b) and (c), [A2]-[A6], [N] and [C3] hold; hence Lemma \ref{final}(a) allows us to consider $\widetilde{\mathbf{M}}^n$ instead of $\mathbf{M}^n$. Moreover, [B2] holds by Proposition \ref{HYprop5.1},  [C1] hold by [A3], and for the case (b) [A1](i)-(iii) and [W] with $w$ given by $(\ref{avar})$ hold by Lemma \ref{HYlem4.1}(a)-(g), while for the case (c) [A1] and [W] with $w$ given by $(\ref{avarend})$ hold by Lemma \ref{HYlem4.1}. Therefore, we complete the proof due to Lemma \ref{final}(b) and Proposition \ref{HYprop3.2}.
\end{proof}

\begin{proof}[\upshape\bf Proof of Theorem \ref{HYthm6.2}]
Note that we do not need the condition [A6] in order to verify the condition [B1] (see the proof of Lemma \ref{HYlem13.1and13.2}(b) and Lemma \ref{final}), an argument similar to the proof of Theorem \ref{HYthm6.1} completes the proof. 
\end{proof}

\section{Proof of Theorem \ref{HKthm2.3}}\label{proofHKthm2.3}

By a localization procedure, we may assume that [SC1]-[SC3], $(\ref{SA4})$ and $(\ref{absmod2})$ hold.

According to Lemma \ref{lembasic}, it is sufficient to prove $\mathbf{M}^n\xrightarrow{ucp}0$. Therefore, it is sufficient to show that
\begin{equation}\label{aimConsistent}
\sup_{0\leq s\leq t}\left|\sum_{i,j}\bar{K}^{ij}_s\left\{\bar{V}_\alpha(\widehat{\mathcal{I}})^i_-\bullet\bar{W}_\beta(\widehat{\mathcal{J}})^j\right\}_s\right|=o_p(k_n^2),\qquad
\sup_{0\leq s\leq t}\left|\sum_{i,j}\bar{K}^{ij}_s\left\{\bar{W}_\beta(\widehat{\mathcal{J}})^j_-\bullet\bar{V}_\alpha(\widehat{\mathcal{I}})^i\right\}_s\right|=o_p(k_n^2)
\end{equation}
as $n\to\infty$ for any $(V,\alpha)\in\{(M^X,g),(\mathfrak{E}^X,g'),(\mathfrak{M}^X,g'),(A^X,g),(\mathfrak{A}^X,g')\}$, $(W,\beta)\in\{(M^Y,g),(\mathfrak{E}^Y,g'), (\mathfrak{M}^Y,g')$, $(A^Y,g),(\mathfrak{A}^Y,g')\}$ and $t>0$.

Consider the first equation of $(\ref{aimConsistent})$. Let $\mathbb{H}_s=\sum_{i,j}\bar{K}^{ij}_s\bar{V}_\alpha(\widehat{\mathcal{I}})^i_-\bullet\bar{W}_\beta(\widehat{\mathcal{J}})^j_s.$ First we assume that $(W,\beta)\in\{(M^Y,g),(\mathfrak{E}^Y,g'),(\mathfrak{M}^Y,g')\}$. By an argument similar to the proof of Lemma \ref{HYlem3.1}, we can rewrite $\mathbb{H}_s$ as $\mathbb{H}_s=\sum_{i,j}\bar{K}^{ij}_-\bullet\{\bar{V}_\alpha(\widehat{\mathcal{I}})^i_-\bullet\bar{W}_\beta(\widehat{\mathcal{J}})^j\}_s.$, and thus $\mathbb{H}$ is a locally square-integrable martingale. Therefore, it is sufficient to prove $\langle\mathbb{H}\rangle_t=o_p(k_n^4)$ as $n\to\infty$ for any $t>0$ due to the Lenglart inequality. Since $[\bar{W}_\beta(\widehat{\mathcal{J}})^j,\bar{W}_\beta(\widehat{\mathcal{J}})^{j'}]_t=0$ if $|j-j'|>k_n$, we have
\begin{align*}
[\mathbb{H}]_s=\sum_{i,j,i',j':|j-j'|\leq k_n}\bar{K}^{ij}_-\bar{K}^{i'j'}_-\bar{V}_\alpha(\widehat{\mathcal{I}})^i_-\bar{V}_\alpha(\widehat{\mathcal{I}})^{i'}_-\bullet[\bar{W}_\beta(\widehat{\mathcal{J}})^j,\bar{W}_\beta(\widehat{\mathcal{J}})^{j'}]_s.
\end{align*}
Hence, $(\ref{absmod2})$, $(\ref{noiseint})$, $(\ref{noiseest2})$, the Schwarz inequality and the fact that $\sum_i\bar{K}^{ij}\lesssim k_n$ for every $j$ yield
\begin{align*}
E_0\left[[\mathbb{H}]_s\right]\lesssim k_n^2\cdot k_n\bar{r}_n|\log b_n|\sum_{j,j':|j-j'|\leq k_n}[\bar{W}_\beta(\widehat{\mathcal{J}})^j,\bar{W}_\beta(\widehat{\mathcal{J}})^{j'}]_s,
\end{align*}
and thus [SC1]-[SC3], the Kunita-Watanabe inequality and the inequality of arithmetic and geometric means imply that $E_0\left[[\mathbb{H}]_s\right]\lesssim k_n^4\cdot k_n\bar{r}_n|\log b_n|=o(k_n^4)$ because $\xi'>1/2$. Therefore, we obtain the desired result due to Proposition 4.50 in \cite{JS}. 

Next we assume that $(W,\beta)\in\{(A^Y,g),(\mathfrak{A}^Y,g')\}$. Then, [SC1], $(\ref{absmod2})$, $(\ref{noiseest2})$ and the Schwarz inequality yield
\begin{align*}
E_0\left[\sup_{0\leq s\leq t}|\mathbb{H}_s|\right]\lesssim\sqrt{k_n\bar{r}_n|\log b_n|}\sum_{i,j}\bar{K}^{ij}_t\sum_{q=0}^{k_n-1}|\widehat{J}^{j+q}(t)|,
\end{align*}
hence the fact that $\sum_i\bar{K}^{ij}\lesssim k_n$ for every $j$ implies $E_0\left[\sup_{0\leq s\leq t}|\mathbb{H}_s|\right]\lesssim\sqrt{k_n\bar{r}_n|\log b_n|}\cdot k_n^2=o(k_n^2)$ because $\xi'>1/2$. 

Consequently, the first equation of $(\ref{aimConsistent})$ holds. By symmetry we also obtain the second equation of $(\ref{aimConsistent})$, and thus we complete the proof.\hfill $\Box$



\section*{Acknowledgements}

\addcontentsline{toc}{section}{Acknowledgements}

The author would like to express his gratitude to Professor Nakahiro Yoshida, who took care of his research and had valuable and helpful discussions.

 
\end{document}